\documentclass[pdflatex,sn-nature]{sn-jnl}

\usepackage{graphicx}%
\usepackage{multirow}%
\usepackage{amsmath,amssymb,amsfonts}%
\usepackage{amsthm}%
\usepackage{mathrsfs}%
\usepackage[title]{appendix}%
\usepackage{xcolor}%
\usepackage{textcomp}%
\usepackage{manyfoot}%
\usepackage{booktabs}%
\usepackage{algorithm}%
\usepackage{algorithmicx}%
\usepackage{algpseudocode}%
\usepackage[capitalise]{cleveref}
\usepackage{nicefrac}
\usepackage{float}
\usepackage{subcaption}

\usepackage{mathrsfs}
\usepackage{txfonts}
\usepackage{textcomp, gensymb}
\usepackage[utf8]{inputenc}
\usepackage[title]{appendix}
\usepackage{pdfpages}
\usepackage{hyperref}
\usepackage{bbm}
\usepackage{mathtools}

\usepackage[pagewise]{lineno}
\usepackage{diagbox}

\newcommand{\R}{\mathbb{R}}
\newcommand{\Z}{\mathbb{Z}}
\newcommand{\N}{\mathbb{N}}
\newcommand{\e}{\varepsilon}
\newcommand{\di}[1]{\,\mathrm{d}#1}
\newcommand{\dive}{\operatorname{div}}
\newcommand{\interior}{\operatorname{int}}
\newcommand{\ddt}{\frac{\operatorname{d}}{\operatorname{d}t}}
\newcommand{\twosc}{\stackrel{2}{\rightharpoonup}}
\newcommand{\stwosc}{\stackrel{2}{\rightarrow}}
\newcommand{\jump}[1]{\llbracket #1\rrbracket}
\newcommand{\thetaIdx}[1]{\theta^{#1}_\varepsilon}
\newcommand{\OmegaIdx}[1]{\Omega^{#1}_\varepsilon}
\newcommand{\alphafg}{\alpha}
\newcommand{\Gammafs}{\Gamma_\varepsilon}
\newcommand{\GammaD}{\Gamma_d}
\newcommand{\tildex}{\Tilde{x}}
\newcommand{\Sigmas}{\Sigma^{s}_\varepsilon}
\newcommand{\Sigmaf}{\Sigma^{f}_\varepsilon}
\newcommand{\SigmaIn}{\Sigma^{\text{in}}}
\newcommand{\SigmaOut}{\Sigma^{\text{out}}}
\newcommand{\SigmaInEps}{\Sigma^{\text{in}}_\e}
\newcommand{\SigmaOutEps}{\Sigma^{\text{out}}_\e}

\newcommand{\uEpsBC}{u_{\e, BC}}
\newcommand{\uBC}{u_{BC}}
\newcommand{\underlinehigh}[1]{\underline{\smash{#1}}}

\newcommand{\nSigma}{n_{\Sigma}}

\newcommand\restr[2]{{
  \left.\kern-\nulldelimiterspace 
  #1 
  \vphantom{\big|} 
  \right|_{#2} 
  }}

\makeatletter
\newcommand{\myitem}[1]{%
\item[#1]\protected@edef\@currentlabel{#1}%
}
\makeatother

\newtheorem{corollary}{Corollary}
\newtheorem{lemma}{Lemma}

\crefname{lemma}{Lemma}{lemmas}
\Crefname{lemma}{Lemma}{Lemmas}
\crefname{thm}{theorem}{theorems}
\Crefname{thm}{Theorem}{Theorems}
\Crefname{algocf}{Algo.}{Algorithm}

\theoremstyle{thmstyleone}%
\newtheorem{theorem}{Theorem}
\newtheorem{proposition}[theorem]{Proposition}%

\theoremstyle{thmstyletwo}%
\newtheorem{remark}{Remark}%

\theoremstyle{thmstylethree}%
\newtheorem{definition}{Definition}%

\begin{document}
\title[Fluid-Heat System in a Thin, Rough Layer]{Analysis and Simulation of a Fluid-Heat System in a Thin, Rough Layer in Contact With a Solid Bulk Domain}


\author*[1]{\fnm{Tom} \sur{Freudenberg}}\email{tomfre@uni-bremen.de}

\author[2,3]{\fnm{Michael} \sur{Eden}}\email{Michael.Eden@mathematik.uni-regensburg.de}

\affil*[1]{\orgdiv{Center for Industrial Mathematics}, \orgname{University of Bremen}, \country{Germany}}

\affil[2]{\orgdiv{Faculty of Mathematics}, \orgname{University Regensburg}, \country{Germany}}

\affil[3]{\orgdiv{Mathematics Department}, \orgname{Karlstad University},  \country{Sweden}}


\abstract{We investigate the effective coupling between heat and fluid dynamics within a thin fluid layer in contact with a solid structure via a rough surface.
Moreover, the opposing vertical surfaces of the thin layer are in relative motion.
This setup is motivated by grinding processes, where cooling lubricants interact with the rough surface of a rotating grinding wheel.
The resulting model is nonlinearly coupled through $(i)$ temperature-dependent viscosity and $(ii)$ convective heat transport.
The underlying geometry is highly heterogeneous due to the thin rough surface characterized by a small parameter \(\epsilon > 0\) that represents both the height of the layer and the periodicity of the roughness.

We analyze this nonlinear system for existence, uniqueness, and energy estimates and study the limit behavior $\e \to 0$ within the framework of two-scale convergence in thin domains.
In this limit, we derive an effective interface model in 3D (a line in 2D) for the heat-fluid interactions inside the fluid.
We implement the system numerically and validate the limit problem through a direct comparison with the $\e$-model.
Furthermore, we investigate the influence of the temperature-dependent viscosity and various geometrical configurations with simulation experiments. 
The corresponding numerical code is freely available on GitHub.}


\maketitle

\section{Introduction}
This research is motivated by grinding processes, which typically involve three main components: the grinding wheel, the workpiece, and a cooling lubricant.
Here, the workpiece is the material that is processed or shaped through grinding and the cooling lubricant is a fluid that is used to reduce friction and remove heat during the grinding process \cite{klocke2009}.
Given the complex granular surface of the grinding wheel, thin gaps and channels emerge in the contact area of the wheel and the workpiece, allowing the coolant to flow through.
Understanding the formation and the influence of these microstructures is crucial in optimizing the usage of coolants \cite{TONSHOFF1992, Wiederkehr2023}.
Naturally, resolving the granular wheel structure and the potentially rough workpiece surface is numerically not feasible, creating a need for homogenized models that account for microscopic effects in an effective way.
Macroscopic models describing the heat dynamics in the grinding gap via an interface temperature field can be found in the literature; see \cite{GU2004} and the references therein.
These models are heuristically derived and the microstructures and fluid flow in the gap are not accounted for.
The fluid dynamics in the gap is strongly influenced by the viscosity of the coolant, which in reality is temperature dependent \cite{Garca-Coln1989, Perez_2008}; and given the large temperature variations during the grinding process, this effect should not be ignored as it can influence the quality of the results \cite{Paiva2021,Wiesener2023}. Numerical simulations to predict the fluid pressure in the contact zone are usually done with the help of the Reynolds equation \cite{Thunich2023, VESALI2014}, but often without taking into account the temperature variations.

Similar setups appear in lubrication problems and bearings.
There, the aim is not to remove material, but instead to reduce the friction between two moving objects with the help of oils.
Often, temperature variations are encountered (either by friction-induced heat or external supply) that influence the viscosity of the fluid \cite{BAIR2001, HABCHI2010}.
Another related area is fractured media, where thin cracks are present in a larger bulk domain.
Solutes and/or heat can be transported through these thin cracks, see \cite{KUMAR2020} and the references therein.
We also point to \cite{KRZACZEK2023} for a numerical study on temperature-dependent viscosity on flows through fractures.
Although there are differences, one typically expects much lower velocities than in our case, the analytical results presented in this article for the temperature could still be transferred.

We point out that while our model is motivated by a grinding process, we do not try to capture all or even most of the relevant aspects of such a process; for example, we do not include solid-to-solid contact, time-dependent contact zones, or mechanical deformations. Instead, we consider a simplified framework to study some phenomena in a mathematically rigorous setting. We acknowledge that our assumptions are strong idealizations and may not directly reflect a real-world application. Many of the not included aspects are more challenging to model mathematically, and remain open from a rigorous mathematical standpoint. However, we believe that the mathematical results in this article still offer a starting point for further studies and are also relevant to other problems.

Starting with an idealized periodic setup with a small parameter $\e\ll1$ characterizing both the length scale of the periodicity as well as the average thickness of the thin layer, we rigorously derive an effective, homogenized model accounting for the nonlinear coupling between heat and fluid flow dynamics.
This is accomplished by conducting a limiting procedure $\e\to0$ within the framework of two-scale homogenization for thin domains. 

From a mathematical point of view, the problem consists of two main aspects: a large bulk domain where the rough interface acts similarly to a rough boundary, and a thin rough layer. For bulk domains with a rough boundary, homogenization procedures, including diffusion or fluid flow problems, can be found in \cite{Bassion2008, CHECHKIN1999, Mikeli2009}, or in the case of rough interfaces, we point to \cite{Donato_2019}. 
The case where the amplitude of the roughness does not scale with $\e$ has also been considered in the literature, see for example \cite{CHECHKIN1999, Donato_2019, Nevard1997} and the references therein. 
Additionally, there are some connections to dimension reductions in fractured media, where inside a bulk domain a thin layer of width $\e$ with different material properties appears (for example, a locally varying permeability in a porous medium). Both formal as well as rigorous limit procedures have been carried out for such problems and the thin layer reduces to effective interface equations, see \cite{AHMED2017, KUMAR2020, Pop2016}.

On the other hand, homogenization of diffusion problems in thin layers with periodic structures has been extensively studied before with the concept of two-scale convergence; see \cite{Gahn17, Neuss07}.
Effective fluid flow through thin layers was rigorously investigated in \cite{Bayada_1986} and later extended for a periodic roughness in \cite{Bayda_1989}. Starting with the stationary Stokes equation, they proved the validity of the Reynolds equation \cite{Reynolds1886} when the layer height is small compared to the length of the roughness period.
If the periodicity and the height are of the same order (usually called the Stokes roughness regime), a model similar to a Darcy equation is derived.
The Reynolds equation plays an important role in lubrication, which has a setup similar to our underlying scenario.
The equation has been further studied and homogenized in the literature \cite{ALMQVIST2007, BENHABOUCHA_2024, Chupin2012, Fabricius2017}; see also \cite[Section 1]{LUKKASSEN_2011} for a good overview.
For layers without roughness, the first derivation of the Reynolds equation with the concept of two-scale convergence was achieved in \cite{Marusic2000}.
Flow through a thin layer, without oscillations, coupled with a temperature-dependent viscosity was studied in \cite{Almqvist2002}, where a Reynolds equation with temperature-dependent viscosity was derived via asymptotic expansion. Furthermore, fluid flow in thin (porous) layers has been rigorously homogenized with the concept of two-scale convergence \cite{Fabricius2022, Fabricius2023} and flow through a porous interface \cite{Gahn2024}.

To be more specific, we consider a geometry incorporating a rough, thin layer, where the boundaries of the layer are additionally in relative motion to each other.
In this geometry, we are considering the coupling between a quasi-stationary Stokes system for the fluid flow and the heat dynamics.
Here, strong coupling comes into play via the convection and temperature-dependent viscosity.
We assume that the viscosity is a Lipschitz continuous function with respect to the temperature.
We show that the resulting nonlinear problem has a weak solution (\cref{thm:solution_operator_exitsence}) via a Schauder fixed point argument. A priori estimates of the temperature are obtained with energy estimates. The pressure is estimated with the help of Bogovskii operators for thin layers, which can be found in existing literature \cite{Fabricius2022, Fabricius2023}.
Uniqueness follows under additional regularity assumptions at the solution (\cref{thm:eps_problem_uniqueness}).
In particular, any strong solution is necessarily unique.

The main novel aspect of our work is the derivation of $\e$-independent $L^\infty$ bounds of the temperature field in the thin layer (\Cref{lemma:infinity_estimate}), utilizing results of \cite{Ladyvzenskaja1968}. These are used to obtain strong two-scale convergence, which is required to pass to the limit in the nonlinear terms. To this end, we also show in \cref{sec:auxiliary_results} multiple ($\e$-dependent) trace estimates and embedding estimates for both the thin layer and the bulk domain.
The limit procedure itself can then be carried out with arguments already established in prior works \cite{Bhattacharya2022, Gahn2024, Fabricius2023}, with slight modifications of the arguments to obtain a two-pressure representation and the boundary conditions. The homogenized model is stated in \cref{thm:homogenized_model}.
Similarly, as for the $\e$-problem, we are able to establish uniqueness for the homogenized model under additional regularity assumptions (\cref{thm:homogenized_model_uniqueness}).

In the end, a small convergence study $\e\to0$ is carried out numerically, investigating the order of convergence of the error between the effective model and the microscale model.
Moreover, we conduct several numerical experiments showcasing and investigating the effects of the temperature-dependent viscosity, the roughness pattern, as well as the inflow condition.
The code of our implementation is uploaded and made freely accessible on GitHub \cite{Code}.

The limit model (see \cref{thm:homogenized_model}) consists of a heat equation for the solid bulk domain coupled to heat and fluid dynamics on a lower-dimensional surface. 
The effective equation for the fluid temperature recovers the models derived for diffusion problems in thin connected layers \cite{Gahn17, Gahn2024}, but with an additional advection term induced by the effective flow field. 
The effective fluid velocity is given by a two-pressure Stokes system, which can be reduced to a Darcy-like equation with a temperature-dependent viscosity.
Ignoring the viscosity, the same effective equation as in \cite{Bayda_1989} is obtained; with fewer mathematical assumptions thanks to the use of two-scale convergence. 
The effective flow model is also similar to \cite{Fabricius2023} where pressure-driven flows in thin porous layers were analyzed. The cell problems (\ref{eq:cell_problem_stokes}) used to compute the effective permeability in the layer are the same as in classical porous media \cite[Eq. (1.11)]{Allaire1989} and also appear, with slightly modified boundary conditions, in \cite[Eq. (6.4)]{Fabricius2023}. 
In addition to the permeability, a term induced by the relative motion appears in the Darcy equation, see \Cref{eq:cell_problem_stokes_bc_movement}. A similar cell problem occurs in \cite[Eq. (20)]{Gahn2024} when computing the effective flow through a thin porous interface. 
We conclude that individual parts of the derived limit model have appeared in the literature in different related setups.
The main novelty of this work is to combine these interesting features, in particular with respect to the coupling through the temperature-dependent viscosity and the advection term.
Mathematically speaking, this introduces several difficulties. For one, the nonlinear coupling necessitates a careful fixed-point argument to ensure existence of solutions where uniqueness is generally only to be expected under higher regularity assumptions (see \cref{thm:eps_problem_uniqueness}).
In addition, uniform (w.r.t.~the scale parameter $\e$) $L^\infty$-estimates are needed for this fixed-point argument to work. Due to the domain interactions (bulk domain connected to a thin fluid layer), this calls for the careful use of $\e$-controlled embeddings (see \cref{lem:embedding_constant,lemma:infinity_estimate}).

From an application point of view, this is of particular interest whenever large temperature variations are possible as this generally implies changes in viscosity.
Returning specifically to the application of grinding, a similar model was heuristically introduced in \cite{GU2004}. There, the coolant is also modeled by a function defined on the interface between the wheel and the workpiece. We could therefore give some mathematical justification for the model used in the literature. However, the heuristic model is not completely the same; the fluid velocity is predetermined, no effective parameters depending on the microstructure appear, and the heat diffusion is set to zero along the interface.

This paper is structured as follows: In \cref{sec:setup_section}, we start by introducing the geometrical setup, the mathematical model, and the assumptions needed to pass to the limit $\e \to 0$. In \cref{sec:micro_scale_model} the underlying nonlinear model is analyzed regarding existence and uniqueness. Additionally, $\e$-dependent solution estimates are derived. \cref{sec:homogenization} demonstrates the homogenization process and in \cref{thm:homogenized_model} the effective model is presented. Finally, in \cref{sec:simulations}, we investigate and compare the effective model and the resolved micro model with the help of numerical simulations. 
\section{Problem formulation and assumptions}\label{sec:setup_section}
In this section, we provided the setup of the underlying geometry and the concrete differential equations on the micro scale.
\subsection{Description of the geometry}
We start by describing the geometry and the assumptions on the regularity of the interface. In the following, the superscript $f$ denotes that a property belongs to the fluid layer, while $s$ denotes objects belonging to the solid bulk atop the fluid layer. The considered time interval is denoted by $S=(0, T)$, for $T>0$. 

The general setup is presented in \cref{fig:domain_picture}. We consider spatial domains of dimension $d=2,3$.
The unit vectors are denoted with $e_i\in \R^d$ for $i=1,\dots, d$ and the $i$-dimensional unit cube is denoted by $Y^i = (0, 1)^i$.
The overall domain $\Omega \subset \R^d$ is assumed to be of the form $\Omega= \Sigma\times (0, H)$ where $\Sigma\subset\R^{d-1}$ is a bounded $d-1$ dimensional cuboid $\Sigma\subset \R^{d-1}$ with corner coordinates in $\Z^{d-1}$.\footnote{So $\Omega$ is also a $d$-dimensional cuboid.}
By a slight abuse of notation, we understand $\Sigma$ both as a set in $\R^{d-1}$ and the lower boundary part $\Sigma\times \{0\}\subset\partial\Omega$.
Any point $x \in \Omega$ is also denoted by $x = (\Tilde{x}, x_d) \in \Sigma\times (0, H)$, the same notation is used for $y \in Y^d$. 
\begin{figure}[ht]
    \centering
    \includegraphics[width=0.7\linewidth]{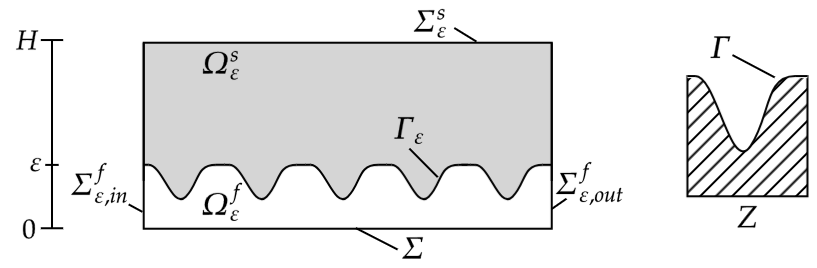}
    \caption{Schematic depiction of the geometric setup. Left: The complete domain with the different subdomains and boundaries. Note that $\Omega=\OmegaIdx{s}\cup\OmegaIdx{f}\cup\Gamma_\e$. Right: The reference cell.}
    \label{fig:domain_picture}
\end{figure}

To construct the rough thin layer, we introduce the open periodic reference cell $Z \subseteq Y^d$ for the fluid domain. We assume, for technical reasons, that 
\begin{itemize}
    \item $Z$ is Lipschitz and connected, $Y^d \setminus Z$ is either empty or also Lipschitz and connected, 
    \item there is a $\gamma_0 > 0$ such that $y_d \geq \gamma_0$ for all $y \in Y^d\setminus Z$,
    \item the domain is periodic in the lateral directions, meaning 
        \begin{equation*}
            \partial Z \cap \{y_i = 0\} + e_i = \partial Z \cap \{y_i = 1\}, \quad \text{for } i=1,\dots,d-1,
        \end{equation*}
    \item if $Y^d \setminus Z \neq \emptyset$, then $Y^d \setminus Z$ touches the top boundary:
        \begin{equation*}
            \left|\partial (Y^d \setminus Z) \cap \{y_d = 1\}\right| > 0.
        \end{equation*}
\end{itemize}
The lower bound $\gamma_0>0$ ensures that the solid does not touch the bottom boundary. The reference interface $\Gamma$ and the combined upper and lower interface of $Z$, denoted by $\Gamma_d$, are given by
\begin{linenomath*}\begin{equation*}
        \Gamma \coloneqq \interior  ( \partial Z \cap \partial(Y\setminus Z)  ),\quad\GammaD:=\Gamma\cup\{y_d=0\}.
\end{equation*}\end{linenomath*}
Since the corners of $\Sigma$ are at integer coordinates, we can tile $\Sigma$ with $\e Y^{d-1}$ for $\e\coloneqq \e_n=n^{-1}$ for any $n \in \N$.
In the following, we suppress the $n$ and note that any limit $\e\to0$ is understood as a limit $n\to\infty$. With that in mind, we define $\mathcal{K}_\e \coloneqq \{k \in \Z^{d-1} \times \{0\} : \e(Y + k) \subset \Omega\}$.
We can introduce the rough, thin fluid layer $\OmegaIdx{f}$, the rough solid domain $\OmegaIdx{s}$, and the rough interface $\Gamma_\e$ by (assuming that $\e<H$)
\begin{linenomath*}\begin{align*}
    \OmegaIdx{f} \coloneqq \interior\left(\bigcup_{k\in \mathcal{K}_\e} \e(\overline{Z}+k) \right), \quad 
    \OmegaIdx{s}\coloneqq \Omega \setminus \overline{\Omega}_\e^f, \quad 
    \text{and} \quad 
    \Gamma_\e \coloneqq \interior \left( \partial \OmegaIdx{f} \cap \partial \OmegaIdx{s} \right).
\end{align*}\end{linenomath*}
By construction, $\OmegaIdx{f,s}$ are connected, and we assume they are both also Lipschitz.
Note that $\Omega=\OmegaIdx{s}\cup\OmegaIdx{f}\cup\Gamma_\e$.
We denote the outer boundary of the solid domain by $\Sigmas = \partial\OmegaIdx{s} \setminus \Gammafs$.
Now, let $\SigmaIn,\SigmaOut\subset\partial\Sigma$ be two disjoint sets with positive surface measure such that $\SigmaIn\cup\SigmaOut=\partial\Sigma$.
We then split the boundary of $\OmegaIdx{f}$ into
\[
\partial\OmegaIdx{f}=\Sigma\cup\Gamma_\e\cup\SigmaInEps\cup\SigmaOutEps
\]
where the sets $\SigmaInEps, \SigmaOutEps$ are given by
\begin{linenomath*}\begin{equation*}
    \SigmaInEps \coloneqq (\SigmaIn \times (0, \e)) \cap \partial \OmegaIdx{f}
    \quad \text{ and }  \quad
    \SigmaOutEps \coloneqq (\SigmaOut \times (0, \e)) \cap \partial \OmegaIdx{f}.
\end{equation*}\end{linenomath*}
Finally, with $\Omega_\e = \Sigma\times (0, \e)$, we denote the thin layer without any roughness.
Note, that both $H$ and $|\Sigma|$ are independent of $\e$ and are therefore of order $\mathcal{O}(1)$.
The volume of the fluid layer $|\OmegaIdx{f}|$ and the area of the lateral boundary $|\SigmaInEps|, |\SigmaOutEps|$ are of order ${O}(\varepsilon)$.
Furthermore, for the length of the rough interface and the remaining boundary of the solid domain, it holds $|\Gammafs|, |\Sigmas| = \mathcal{O}(1)$.

With $n$, we denote the outer normal vector on the domains $\Omega, \OmegaIdx{f,s}$ and $Z$. On the common interface $\Gamma_\e$ we further write $n^{f,s}$ to denote the outer normal vector corresponding to $\OmegaIdx{f,s}$, but the superscript is suppressed whenever the direction is clear from context. Additionally, the normal vector $\nSigma$ is introduced in \Cref{eq:inflow_condition} for the lateral boundaries of the fluid domain.
 
With the subscript $\#$, we indicate a space of periodic functions in the directions $1,\dots, d-1$; in particular, we use
\begin{linenomath*}\begin{align*}
    C_{\#}(Z) &\coloneqq \left\{u\in C(\R^d)\ : u(x+e_i) = u(x) \text{ for all } x \in \R^d \text{ and } i=1,\dots, d-1 \right\},\\
    L^2_{0,\#}(Z)&\coloneqq\left\{u\in L^2_{\text{loc}}(\R^d)\ : \  u_{|Z} \in L^2_0(Z),\ u(x+e_i) = u(x) \text{ for a.a. } x \in \R^d \text{ and } i=1,\dots, d-1 \right\},\\
    H^1_{\#}(Z) &\coloneqq \left\{ 
         u \in H^1_{\text{loc}}(\R^d) : u_{|Z} \in H^1(Z), u(x+e_i) = u(x) \text{ for a.a. } x \in \R^d \text{ and } i=1,\dots, d-1
     \right\},\\
    H^1_{\#,\Gamma_d}(Z) & \coloneqq \left\{ 
         u \in H^1_{\#}(Z)\ :\ u=0\text{ for a.a. }y\in\GammaD
     \right\}.
\end{align*}\end{linenomath*}
Additionally, we define $\nabla_{\tildex} u = (\partial_{x_1} u, \dots, \partial_{x_{d-1}} u, 0)$, where a $0$ is appended such that the gradient it compatible with the usual gradient in $\R^d$. With $\| \cdot \|_{H^1(\OmegaIdx{f} \cup \OmegaIdx{s})}$ we denote the combined $H^1$ norm on both subdomains, e.g.,
\begin{equation*}
    \| \cdot \|_{H^1(\OmegaIdx{f} \cup \OmegaIdx{s})} \coloneqq \sqrt{\| \cdot \|_{H^1(\OmegaIdx{f})}^2 
    + \| \cdot \|_{H^1(\OmegaIdx{s})}^2}.
\end{equation*}

\subsection{The micro model}
For a function $\varphi \in L^2(\Omega)$, we use the notation $\varphi_\e^k\coloneqq \varphi_{|\Omega_\e^k}$ ($k=f,s$) for the restriction of $\varphi$ to the subdomain $\Omega_\e^k$.
A similar notation holds for the coefficients of the models, which are assumed to be constant in each subdomain. 
We suppress superscripts and the index $\e$ when possible to improve readability.

To describe the temperature $\theta_\e = (\thetaIdx{s}, \thetaIdx{f})$, we utilize heat equations in both domains
\begin{subequations}\label{eq:epsilon-problem}
\begin{linenomath*}\begin{alignat}{2}
    \partial_t\theta_\e - \dive\left(\kappa \nabla \theta_\e\right) &= 0 \quad \quad &&\text{ in } S\times \OmegaIdx{s}, 
    \\
    \frac{1}{\varepsilon} \partial_t\theta_\e - \frac{1}{\varepsilon} \dive\left(\kappa \nabla \theta_\e - u_\e \theta_\e \right) &= 0 \quad \quad &&\text{ in } S\times \OmegaIdx{f}, \label{eq:fluid_eps_pde}
    \\
    \theta_\e(0, \cdot) &= \theta_{\e, 0}  &&\text{ in } \Omega.
\end{alignat}\end{linenomath*}
Here, $\kappa=\kappa_f$ in $\OmegaIdx{f}$ and $\kappa=\kappa_s$ in $\OmegaIdx{s}$ are the heat conductivities in each subdomain and $u_\e$ is the fluid velocity given via a Stokes system below.
As they have no specific bearing on the mathematical analysis or the homogenization result, we have set the \textit{volumetric heat capacity} to 1 in both phases.

At the outer boundary, we apply a Dirichlet condition on $\SigmaInEps$ for a prescribed cooling temperature (for simplification set to 0, e.g., the temperature is seen relative to the inflow temperature, which is possible in this case since we always assume $\dive{u_\e} = 0$), and on the other boundary sections we apply a homogeneous Neumann condition:
\begin{linenomath*}\begin{alignat}{2}
    \theta_\e &= 0 &&\text{ on } S \times \SigmaInEps,\label{eq:cooling_fluid}
    \\
    -\kappa \nabla \theta_\e \cdot n &= 0 &&\text{ on } S \times \left( \SigmaOutEps \cup \Sigmas \cup \Sigma \right).
\end{alignat}\end{linenomath*}
The temperature profiles of the two domains are coupled with a Robin exchange condition:
\begin{linenomath*}\begin{alignat}{2}
    -\frac{1}{\varepsilon} \kappa^f \nabla \thetaIdx{f} \cdot n^f &= \alphafg \jump{\theta_\e} + f_\e^f \quad &&\text{ on } S \times \Gammafs, \label{eq:exchange_eps_fluid} \\
    -\kappa^s \nabla \thetaIdx{s} \cdot n^s &= \alphafg \jump{\theta_\e} + f_\e^s \quad &&\text{ on } S \times \Gammafs,
\end{alignat}\end{linenomath*}
where $\jump{\theta} = (\theta^f - \theta^s)$ denotes the usual jump across an interface, $\alphafg$ denotes the heat exchange parameter, and $f_\e^{f,s}$ describe heat sources.

Note, the chosen $\e$-scaling appears in similar problems \cite{Eden2024, Gahn17} and is designed to keep the contributions of the equations in the effect model for $\e \to 0$. From a physical point of view, the scaling originates from assumptions about the physical parameters (e.g., a fast heat conductivity in the thin layer compared to the height of the layer) and can be incorporated by a non-dimensionalization step; see for example \cite[Section 2.1]{KUMAR2020}, where this is explained for fluid flow through thin cracks.

Next to the temperature, we also model the fluid flow inside $\OmegaIdx{f}$. Fluid velocity $u_\e$ and pressure $p_\e$ are given as the solution of a Stokes system with temperature-dependent viscosity $\mu$,
\begin{linenomath*}\begin{alignat}{2}
    -\dive\left(\mu\left(\thetaIdx{f}\right) \nabla u_\e \right) + \nabla p_\e &= 0 \quad \quad &&\text{ in } S\times \OmegaIdx{f}, 
    \\
    \dive\left(u_\e \right) &= 0 \quad \quad &&\text{ in } S\times \OmegaIdx{f}, 
    \\
    u_\e &= \uEpsBC &&\text{ on } S \times \partial \OmegaIdx{f},
    \\
    \int_{\OmegaIdx{f}} p_\e \di{x} &= 0 &&\text{ in } S.
\end{alignat}\end{linenomath*}
Here $\uEpsBC \in H^{1/2}(\partial \OmegaIdx{f}) \cap L^\infty(\partial \OmegaIdx{f})$ describes a given boundary flow.
\end{subequations}

\subsection{Assumptions on data}\label{sec:assumption_on_data}
To ensure the existence of solutions and to pass to the limit $\e \to 0$, we need multiple assumptions on the data; these are collected in this section.
\begin{enumerate}
    \myitem{(A1)}\label{item:A1} For the coefficients we assume $\alphafg, \kappa^f, \kappa^s > 0$ and that they are of order $\mathcal{O}(1)$.
    \myitem{(A2)}\label{item:A2} The temperature dependent viscosity is given by a Lipschitz continuous function ${\mu: \R \to \R}$ that fulfills $\underlinehigh{\mu} \leq \mu(x) \leq \overline{\mu}$ for some $0 < \underlinehigh{\mu} \leq \overline{\mu} < \infty$, independent of $\e$. 
    \myitem{(A3)}\label{item:A3} The initial conditions fulfill that $\theta_{\e, 0} \in L^\infty(\Omega)$ and $C_0 \coloneqq \sup_{\e > 0} \| \theta_{\e, 0}\|_{L^\infty(\Omega)} < \infty$.
    \myitem{(A4)}\label{item:A4} For the heat sources it holds $f_\varepsilon^{f,s} \in L^\infty(S \times \Gammafs)$ and $ \sup_{\varepsilon>0} \|f_\varepsilon^{f,s}\|_{L^\infty(S\times \Gammafs)} < \infty$. 
    \myitem{(A5)}\label{item:A5} For the Dirichlet condition on the lateral fluid boundaries, we assume there is $\uBC \in H^{1/2}(\partial \Sigma \times (0, 1))^d \cap L^{\infty}(\partial \Sigma \times (0, 1))^d$ with the properties:
        \begin{itemize}
            \item $\uBC \cdot e_d = 0$,
            \item $\uBC(\tildex, 0) = u_\text{motion}$, with $u_\text{motion} \in \R^d$ and $u_\text{motion} \cdot e_d = 0$,
            \item $\uBC(\tildex, x_d) = 0$ for almost all $\tildex \in \partial \Sigma$ and $x_d \geq \gamma_0$,
            \item inflow and outflow conditions $\left( \uBC \cdot n_\Sigma \right)_{|\SigmaIn \times (0, 1)} < 0  \text{ and } \left( \uBC \cdot n_\Sigma\right)_{|\SigmaOut \times (0, 1)} \geq 0$.
        \end{itemize}
    In addition, we require \begin{linenomath*}\begin{equation}\label{eq:inflow_condition}
        \int_{\partial \Sigma \times (0, 1)} \uBC \cdot n_\Sigma \di{\sigma_x}= 0
    \end{equation}\end{linenomath*} 
    to ensure the existence of solutions. Here, $n_\Sigma$ is the normal vector that points outward from $\Sigma \times (0, 1)$. 
    %
    %
    \myitem{(A6)}\label{item:A6} For the initial conditions there are limit functions $\theta^s_0 \in L^2(\Omega)$ and $\theta^f_0 \in L^2(\Sigma)$ such that $\chi_{\OmegaIdx{s}}\theta_{\e, 0} \to \theta^s_0$ in $L^2(\OmegaIdx{s})$ and $\chi_{\OmegaIdx{f}}\theta^f_{\e, 0} \twosc \chi_{Z}\theta^f_0$ in $L^2(\Omega_\e)$, for $\e\to 0$. 
    Here, $\twosc$ denotes two-scale convergence in the thin layer which is introduced in \cref{def:two_scale}.
    \myitem{(A7)}\label{item:A7} For the heat sources there are $f^{f,s}\in L^2(S\times\Sigma\times \Gamma)$ such that $f_\varepsilon^{f,s} \twosc f^{f,s}$, for $\e \to 0$.
\end{enumerate}
The Assumptions \ref{item:A1} and \ref{item:A2} are often encountered in similar problems, especially the properties of $\mu$ are needed to ensure the existence of a solution.

The assumption \ref{item:A5} characterizes the boundary condition on the scaled boundary $\partial \Sigma \times (0, 1)$, the microscopic boundary condition is defined via
\begin{equation*}
    \uEpsBC(x) \coloneqq \begin{cases}
        \uBC(\tildex, x_d/\e) \quad &\text{for } x \in \SigmaInEps \cup \SigmaOutEps, \\
        u_\text{motion} &\text{for } x \in \Sigma, \\
        0 &\text{for } x \in \Gamma_\e.
    \end{cases}   
\end{equation*}
The inflow and outlfow conditions of $\uBC$ are transferred to $\uEpsBC$, to be precise 
\begin{equation}\label{eq:inflow_and_outflow_condition}
    \left( \uEpsBC \cdot n\right)_{|\SigmaInEps} < 0 
    \quad \text{ and } \quad 
    \left( \uEpsBC \cdot n\right)_{|\SigmaOutEps} \geq 0. 
\end{equation}
Examples of functions that fulfill the assumptions can be found in \cref{sec:simulations}.
Furthermore, by assumption \ref{item:A5}, we can construct a continuation $U \in H^1(\Sigma \times (0, 1))$ of $\uBC$, with $U(\tildex, y) = 0$ for $y \geq \gamma_0$ and $\dive{U}= 0$, \cite[Theorem. 3.5]{Girault_1981}.
By setting $U_\e(x) = (U_1, \dots, \e U_{d})(\tildex, \frac{x_d}{\e})$ the extension can also be transferred to the thin rough layer $\OmegaIdx{f}$ (this requires $\uBC \cdot e_d = 0$ such that $U_\e$ has the correct boundary condition), this construction was already derived in \cite{Bayda_1989}. 
%
\begin{remark}\label{rem:allow_contact}
    We provide some additional comments regarding our model and our mathematical assumptions:
    \begin{itemize}
        \item[$(i)$]$u_\text{motion}$ describes a prescribed velocity at the bottom boundary (for example, the movement induced by a rotating grinding wheel in our motivation). 
        It is not mandatory that $u_\text{motion}$ is constant but could also be function in $H^{1/2}(\Sigma)^d \cap L^\infty(\Sigma)^d$ (as long as $u_\text{motion} \cdot e_d = 0$). 
        To simplify the notation, we use a constant $u_\text{motion}$. Similarly the boundary condition could also be time-dependent, $\uEpsBC \in L^\infty(S; H^{1/2}(\OmegaIdx{f})^d \cap L^\infty(\OmegaIdx{f}))$, as this would not influence the following results.
        \item[$(ii)$] The specific assumptions regarding $\uBC$ and, as a consequence, $\uEpsBC$ are a bit technical, but they ensure that there is an inflow and an outflow region and that $\Sigma$ and $\Gamma_\e$ are in relative motion to each other.
        The regularity assumption ensures compatibility across the different boundary parts.
        \item[$(iii)$] The relationship between temperature and the viscosity of fluids is complex and often modeled heuristically.
        A commonly used example is the \textit{Vogel-Fulcher-Tammann} equation, given by 
        \[
        \mu(\theta)=\mu_0\exp\left(\frac{\alpha}{\theta-T_{0}}\right)
        \]
        where $\mu_0,\alpha,T_{0}$ are empirical fitting parameters \cite{Garca-Coln1989}.
        The model suggests that the relationship between viscosity and temperature tends to be smooth which means that the Lipschitz continuity is a reasonable assumption.
        However, this function $\mu$ is not bounded from above and therefore does not satisfy Assumption (A2). 
        In practice, this law is only considered valid within a specific temperature range—namely, the interval where the material remains in its liquid state.
        Therefore, it makes sense to restrict this law to such a temperature range and extend it to $\R$ by its boundary values, which also preserves Lipschitz regularity.
        Moreover, the $L^\infty$-estimates we establish for the temperature field (see \cref{lemma:infinity_estimate}) provide \textit{a posteriori} a qualitative justification for this restriction.
    \end{itemize}
\end{remark}
\section{Analysis of the microscale model}\label{sec:micro_scale_model}
We start by analyzing the model on the $\e$-scale to ensure that it is well-posed and to derive solution bounds needed for the limiting procedure.
To this end, we first list some technical lemmas needed in the following sections.

\subsection{Auxiliary results}\label{sec:auxiliary_results}
\begin{lemma}[Extension operators and transformation]\label{lem:extension_operator}
For all $\e >0$ and $p \in [1, \infty)$ there exists
\begin{itemize}
    \item[$(i)$] a linear extension operator $\mathbb{E}_\e^f: W^{1,p}(\OmegaIdx{f}) \to W^{1,p}(\Omega_\e)$ such that 
    \begin{equation*}
        \|\mathbb{E}_\e^f \psi\|_{W^{1,p}(\Omega_\e)} \leq  C \|\psi\|_{W^{1,p}(\OmegaIdx{f})}\quad \text{ for all } \psi \in   W^{1,p}(\OmegaIdx{f}).
    \end{equation*}
    Here, $C>0$ may depend on $p$ but is independent of $\e$.
    \item[$(ii)$] a linear extension operator $\mathbb{E}_\e^s: W^{1,p}(\OmegaIdx{s}) \to W^{1,p}(\Omega)$ such that 
    \begin{equation*}
        \|\mathbb{E}_\e^s \varphi\|_{W^{1,p}(\Omega)} \leq  C \|\varphi\|_{W^{1,p}(\OmegaIdx{s})}\quad \text{ for all } \varphi \in   W^{1,p}(\OmegaIdx{s}),
    \end{equation*}
    with $C>0$ depended on $p$ but independent of $\e$.
\end{itemize}
\end{lemma}
\begin{proof}
    Both extension operators can be constructed using a classical extension procedure \cite{ACERBI_1992}.
    For thin layers, this was carried out in detail and further studied in \cite[Theorem 1]{Gahn_2021}. See also \cite[Lemma A.6]{Gahn2024} for a generalization to other boundary conditions in the case $p=2$.
    The general idea is to first construct an extension operator for a single cell, in the case of the fluid from $W^{1,p}(Z)$ to $W^{1,p}(Y^d)$. Which exists since $\Gamma$ is Lipschitz. The extension is $\e$-independent (as there is no $\e$ involved yet).
    The statement is then obtained by gluing these extension operators together via a periodic partition of unity in conjunction with a scaling argument.
\end{proof}
\begin{lemma}[Trace estimates]\label{lem:trace_estimate}
    \begin{itemize}
        \item[$(i)$] For every $p \in [1, \infty)$, there is a constant $C>0$ (independent of $\e$) such that
        \[
        \|\psi\|_{L^p(\Gammafs)}^p 
        \leq  \frac{C}{\e} \left(\|\psi\|_{L^{p}(\OmegaIdx{f})}^p 
        + \e^p\|\nabla \psi\|_{L^{p}(\OmegaIdx{f})}^p\right)
        \quad \text{for all} \quad \psi \in W^{1, p}(\OmegaIdx{f}).
        \]
        \item[$(ii)$] For every $p \in [1, \infty)$, there is $C>0$ (independent of $\e$) such that
        \[
        \|\phi\|_{L^p(\Gamma_\e)} \leq C\|\phi\|_{W^{1, p}(\OmegaIdx{s})}
        \quad \text{for all} \quad \varphi \in W^{1, p}(\OmegaIdx{s}).
        \]
        \item[$(iii)$] Let $\Sigmaf \coloneqq \partial \OmegaIdx{f} \setminus (\Sigma \cup \Gamma_\e)$ denote the lateral boundary of the fluid domain.
        For every $p \in [1, \infty)$, there is a constant $C>0$ (independent of $\e$) such that
        \begin{equation*}
            \|\psi\|_{L^p(\Sigmaf)} \leq C \|\psi\|_{W^{1,p}(\OmegaIdx{f})}
            \quad \text{for all} \quad \psi \in W^{1,p}(\OmegaIdx{f}).
        \end{equation*}
    \end{itemize}
\end{lemma}
\begin{proof}
    $(i)$. This estimate for $\psi \in W^{1, p}(\OmegaIdx{f})$ is quite standard in homogenization and follows by a splitting argument and a change of variables, see for example \cite[Lemma 5]{Gahn2016} for a more detailed explanation.

    $(ii)$. This was proven in \cite[Lemma A.5]{Gahn2024}.
    
    $(iii).$ The estimate can be derived by a scaling argument. With the extension operator from \cref{lem:extension_operator}, it holds
    \begin{align*}
        \|\psi\|_{L^p(\Sigmaf)}  &= \|\mathbb{E}_\e^f\psi\|_{L^p(\Sigmaf)}
        \leq \|\mathbb{E}_\e^f\psi\|_{L^p(\partial \Sigma \times (0, \e))}
        = \e^{\frac{1}{p}} \|\mathbb{E}_\e^f\psi(\tildex, \e x_d)\|_{L^p(\partial \Sigma \times (0, 1))}
        \\
        &\leq C \e^{\frac{1}{p}} \|\mathbb{E}_\e^f\psi(\tildex, \e x_d)\|_{W^{1,p}(\Omega)}
        \leq  C \|\mathbb{E}_\e^f\psi\|_{W^{1,p}(\Omega_\e)}
        \leq C \|\psi\|_{W^{1,p}(\OmegaIdx{f})}.
    \end{align*}
    Where we used the trace theorem on $\Omega$. Note, as before, the estimate could be improved by a $\e$ in front of the term with the $\partial_{x_d}$ norm, but this is not needed in the following analysis.
\end{proof}

We later need to show boundedness of the temperature solutions to get better estimates for the time derivatives.
There we rely on the following lemma, which follows by standard results via a natural scaling argument:
\begin{lemma}[Interpolation embedding for thin domains]\label{lem:embedding_constant}
    There is $C>0$ such that, for any $\psi\in L^\infty(S;L^2(\Omega_\e))\cap L^2(S;H^1(\Omega_\e))$,
    \[
    \|\psi\|_{L^{r_d}(S\times\Omega_\e)}\leq C\e^{-\frac{1}{d+2}}\left(\|\psi\|_{L^{\infty}(S;L^2(\Omega_\e))}+\|\psi\|_{L^2(S;H^1(\Omega_\e))}\right).
    \]
    Here, $r_d = \frac{2(d+2)}{d}$, $d=2,3$.
\end{lemma}
\begin{proof}
For any $\psi\in L^\infty(S;L^2(\Omega_\e)\cap L^2(S;H^1(\Omega_\e))$, let $\bar{\psi}\in L^\infty(S;L^2(\Omega))\cap L^2(S;H^1(\Omega))$ be the corresponding upscaled function given by $\bar{\psi}(t,\tildex,x_d)=\psi(t,\tildex,\e x_d)$.
Using standard interpolation embeddings (see, e.g., \cite[Chapter 2, Inequality 3.8]{Ladyvzenskaja1968}), it holds
\[
\|\bar{\psi}\|_{L^{r_d}(S\times\Omega)}\leq C\left(\|\bar{\psi}\|_{L^{\infty}(S;L^2(\Omega))}+\|\bar{\psi}\|_{L^2(S;H^1(\Omega))}\right)
\]
for some constant $C>0$ which does not depend on $\e$.
With a standard scaling argument, we also see that
\[
\|\bar{\psi}\|_{L^{r_d}(S\times\Omega)}=\e^{\frac{-1}{r_d}}\|\psi\|_{L^{r_d}(S\times\Omega_\e)},\quad \|\bar{\psi}\|_{L^{\infty}(S;L^2(\Omega))}=\e^{\frac{-1}{2}}\|\psi\|_{L^{\infty}(S;L^2(\Omega_\e))}.
\]
Similarly, for the term including the $H^1$-norm, we estimate (note that $\e<1$)
\[
\|\bar{\psi}\|_{L^2(S;H^1(\Omega))}^2=\e^{-1}\left(\|\psi\|_{L^2(S\times\Omega_\e)}^2+\|\nabla_{\tildex} \psi\|_{L^2(S \times \Omega_\e)}^2
        + \e \|\partial_{x_d} \psi\|_{L^2(S \times \Omega_\e)}^2\right)\leq \e^{-1}\|\psi\|_{L^2(S;H^1(\Omega_\e))}^2.
\]
Since $\nicefrac{1}{r_d}-\nicefrac{1}{2}=\nicefrac{-1}{d+2}$, the statement follows.
\end{proof}
To show the $L^\infty$ estimate of the temperature field, we use a truncation argument. In the following lemma, we state some properties of the time derivative of the truncated function in the case that the original time derivative is only an element in the dual space. We denote by $\langle \cdot, \cdot \rangle_{X'}$ the dual pairing between elements from the space $X'$ and $X$.
\begin{lemma}\label{lem:time_derivative_cutoff}
    Let $\psi \in L^2(S;H^1(\Omega)) \cap H^1(S;H^1(\Omega)')$. For $k >0$ define the truncated function 
    \begin{equation*}
        \psi_k(t, x) \coloneqq \begin{cases}
            \psi(t,x) - k &\text{if } \psi(t,x) > k,
            \\
            0 &\text{otherwise}.
        \end{cases}
    \end{equation*}
    Then $\psi_k \in L^2(S;H^1(\Omega)) \cap H^1(S;H^1(\Omega)')$. Furthermore, for all $\phi \in L^2(S,H^1(\Omega))$ with $\phi(t,x)=0$ if $\psi(t,x)\geq k$, it holds 
    $\langle \partial_t \psi_k, \phi \rangle_{L^2(S;H^1(\Omega)')}=0$.
    If additionally $\psi(0) \in L^\infty(\Omega)$. Then for $k \geq \|\psi(0)\|_{L^\infty(\Omega)}$ and almost all $t \in S$ it holds
    \begin{equation}\label{eq:truncated_time_derivative}
        \int_0^t \langle \partial_t \psi(\tau), \psi_k(\tau)\rangle_{H^1(\Omega)'} \di{\tau}
        = 
        \int_0^t \langle \partial_t \psi_k(\tau), \psi_k(\tau)\rangle_{H^1(\Omega)'} \di{\tau}
        = 
        \frac{1}{2}(\psi_k(t), \psi_k(t))_{L^2(\Omega)}.
    \end{equation}
\end{lemma}
\begin{proof}
    That $\psi_k \in L^2(S;H^1(\Omega)) \cap H^1(S;H^1(\Omega)')$ follows via an approximation with a smooth mollification of $\psi$ and that the cut-off function is Lipschitz continuous. This mollification also directly yields the property $\langle \partial_t \psi_k, \phi \rangle_{L^2(S;H^1(\Omega)')}=0$ for $\phi=0$ if $\psi\geq k$.

    For the integral equality (\ref{eq:truncated_time_derivative}), we use the fundamental theorem of calculus (with $k \geq \|\psi(0)\|_{L^\infty(\Omega)}$) to obtain
    \begin{equation}\label{eq:1_time}
        (\psi_k(t), \psi_k(t))_{L^2(\Omega)} = 2\int_0^t \langle \partial_t \psi_k(\tau), \psi_k(\tau)\rangle_{H^1(\Omega)'} \di{\tau}.
    \end{equation}
    Furthermore, it holds 
    \begin{equation}\label{eq:2}
        \begin{split}
            (\psi_k(t), \psi_k(t))_{L^2(\Omega)} 
            &= (\chi_{\{\psi(t) > k\}}(\psi(t) - k), \psi_k(t))_{L^2(\Omega)}
            \\
            &= (\psi(t) - k, \psi_k(t))_{L^2(\Omega)}
            \\
            &= \int_0^t \langle \partial_t \psi(\tau), \psi_k(\tau)\rangle_{H^1(\Omega)'} \di{\tau}
            +
            \int_0^t \langle \partial_t \psi_k(\tau), \psi(\tau) - k\rangle_{H^1(\Omega)'} \di{\tau}.            
        \end{split}
    \end{equation}
    In the last integral, we can split $\psi - k$ into positive and negative parts. Since $(\psi - k)_{-} = 0$ if $\psi \geq k$, this leads to
    \begin{align*}
        \int_0^t \langle \partial_t \psi_k(\tau), \psi(\tau) - k\rangle_{H^1(\Omega)'} \di{\tau} 
        &= \int_0^t \langle \partial_t \psi_k(\tau), (\psi(\tau) - k)_{+}\rangle_{H^1(\Omega)'} \di{\tau} 
        + \underbrace{\int_0^t \langle \partial_t \psi_k(\tau), (\psi(\tau) - k)_{-}\rangle_{H^1(\Omega)'} \di{\tau}}_{=0}
        \\
        &= \int_0^t \langle \partial_t \psi_k(\tau), \psi_k(\tau)\rangle_{H^1(\Omega)'} \di{\tau}.
    \end{align*}
    Plugging the last equation into \Cref{eq:2} and comparing with \Cref{eq:1_time} yields the desired result.
\end{proof}
\subsection{Well-posedness and a prior estimates}\label{sec:analysis_micro_model}
Next, we investigate the underlying model on the microscale. First, we define 
\begin{equation*}
    L^2_0(\OmegaIdx{f}) \coloneqq \left\{
        \varphi \in L^2(\OmegaIdx{f}) : \int_{\OmegaIdx{f}} \varphi \di{x} = 0
    \right\}
    \quad \text{and} \quad 
    H^1_A(\OmegaIdx{f}) \coloneqq \left\{
        \varphi \in H^1(\OmegaIdx{f}) : \varphi_{|A} = 0
    \right\},
\end{equation*}
where $A \subseteq \partial \OmegaIdx{f}$. To shorten the notation we use, when applicable, the usual notation $H^1_0(\OmegaIdx{f}) = H^1_{\partial \OmegaIdx{f}}(\OmegaIdx{f})$. For the solutions of the temperature space we define $\Theta_\varepsilon \coloneqq H^1(\OmegaIdx{s}) \times H^1_{\SigmaInEps}(\OmegaIdx{f})$.
The dual of a finite product space is isometrically isomorph to the product of the dual spaces.
For any element $\psi \in \Theta_\e'$, there are unique $\psi^s \in H^1(\OmegaIdx{s})'$, $\psi^f \in H^1_{\SigmaInEps}(\OmegaIdx{f})'$ such that for all $\varphi\in \Theta_\e$
\begin{equation*}
    \langle \psi, \varphi \rangle_{\Theta_\e'} = 
    \langle \psi^s, \varphi^s \rangle_{H^1(\OmegaIdx{s})'} 
    + 
    \langle \psi^f, \varphi^f \rangle_{H^1_{\SigmaInEps}(\OmegaIdx{f})'}.
\end{equation*}
In the following, we identify $\psi=(\psi^s, \psi^f) \in \Theta_\e'$ and suppress the superscript when the context is clear from the dual pairing.

For the fluid velocity and pressure, we define
\begin{linenomath*}\begin{align*}
    V_\varepsilon &\coloneqq \left\{ v \in L^\infty(S; H^1(\OmegaIdx{f})^d) : v = \uEpsBC \text{ on } \partial \OmegaIdx{f} \text{ and } \dive{v} = 0 \text{ a.e. in } \OmegaIdx{f} \right\}
    \quad \text{ and } \\
    &\hspace{3cm}Q_\e \coloneqq L^\infty(S; L^2_0(\OmegaIdx{f})).
\end{align*}\end{linenomath*}

The weak formulation of the heat equation, for a fixed $v \in V_\e$, is given by: Find $\theta_\e = (\thetaIdx{s}, \thetaIdx{f})\in L^2(S; \Theta_\e)$ such that $\partial_t \theta_\e = (\partial_t \theta_\e^s, \partial_t \theta_\e^f)\in L^2(S; \Theta_\e')$ and for all $\varphi \in L^2(S;\Theta_\e)$ and almost all $t\in S$
\begin{linenomath*}\begin{equation}\label{eq:weak_formulation_temp}
    \begin{split}
        \langle \partial_t \theta_\e, \varphi \rangle_{H^1(\OmegaIdx{s})'} 
        &+ (\kappa\nabla\theta_\e, \nabla\varphi)_{ \OmegaIdx{s}} 
        + \frac{1}{\varepsilon}\langle\partial_t \theta_\e, \varphi\rangle_{H^1_{\SigmaInEps}(\OmegaIdx{f})'} 
        + \frac{1}{\varepsilon} (\kappa\nabla\theta_\e, \nabla\varphi)_{ \OmegaIdx{f}} 
        \\
        &+\frac{1}{\e} (v \cdot \nabla\theta_\e, \varphi)_{\OmegaIdx{f}} 
        + \alphafg (\jump{\theta_\varepsilon}, \jump{\varphi})_{ \Gammafs}
        = (f_\e^s, \varphi^s)_{\Gammafs} + (f_\e^f, \varphi^f)_{\Gammafs}.
    \end{split}
\end{equation}\end{linenomath*}
Here, $(\cdot, \cdot)_D$ denotes the usual $L^2$ inner product over the domain $D$. Given a $\psi_\e \in L^2(S;\Theta_\e)$, the weak form of the Stokes equation is: Find $(u_\e, p_\e) \in V_\e \times Q_\e$ such that for all $v \in H^1_0(\OmegaIdx{f})^d $ and almost all $t\in S$
\begin{linenomath*}\begin{equation}\label{eq:weak_formulation_stokes}
    \begin{split}
        &(\mu(\psi_\e) \nabla u_\e , \nabla v)_{ \OmegaIdx{f}} 
        - (p_\e, \dive{(v)})_{ \OmegaIdx{f}} = 0.
    \end{split}
\end{equation}\end{linenomath*}
For a fixed $\varepsilon > 0$ we now prove that there exists a solution of the coupled nonlinear system given by \cref{eq:weak_formulation_temp} and (\ref{eq:weak_formulation_stokes}). To show this, we follow a similar strategy as in \cite{Perez_2008}, where we split up the problem, first consider both weak formulations on their own, and then show a solution of the coupled system by a fix-point argument of an iterative scheme.

\begin{lemma}\label{lem:temperature_existence_bounds}
    Let the Assumptions \ref{item:A1}, \ref{item:A3} and \ref{item:A4} be fulfilled. For a fixed $\e > 0$ and a given $v \in V_\e$ there exists a unique solution $\theta_\varepsilon \in L^2(S; \Theta_\varepsilon)$ satisfying $\partial_t \theta_\e \in L^2(S; \Theta_\e')$, $\theta_\e(0, \cdot) = \theta_{\e, 0}$ a.e. and \cref{eq:weak_formulation_temp}. In addition, it holds
    \begin{linenomath*}\begin{equation}
    \label{eq:solution_esitmate}
        \begin{split}
            \|\theta_\e\|_{L^\infty(S; L^2(\OmegaIdx{s}))}  
            &+ \frac{1}{\sqrt{\e}} \|\theta_\e\|_{L^\infty(S; L^2(\OmegaIdx{f}))} 
            + \|\nabla \theta_\e\|_{L^2(S; L^2(\OmegaIdx{s}))} 
            \\
            &+ \frac{1}{\sqrt{\e}} \|\nabla \theta_\e\|_{L^2(S; L^2(\OmegaIdx{f}))}
            + \|\jump{\theta_\e}\|_{L^2(S\times\Gammafs)} 
             \leq C,
        \end{split}
    \end{equation}\end{linenomath*}
    where $C<\infty$ can be chosen independent of $\varepsilon$ and $v$. 
\end{lemma}
\begin{proof}
    We start by showing the existence and uniqueness of a solution for a fixed $v \in V_\e$. First, note that our problem is of the form
    \begin{equation}\label{eq:parabolic_form}
          \langle\partial_t \theta_\e, \varphi \rangle_{H^1(\OmegaIdx{s})'}
        + \frac{1}{\varepsilon} \langle\partial_t \theta_\e, \varphi \rangle_{H^1_{\SigmaInEps}(\OmegaIdx{f})'} 
        + A_\e(\theta_\e, \varphi; t) = (f_\e^s, \varphi^s)_{\Gammafs} + (f_\e^f, \varphi^f)_{\Gammafs},
    \end{equation}
    with the time-dependent bilinear form 
    \begin{multline*}
        A_\e(\cdot, \cdot ,t): L^2(S; \Theta_\e) \times L^2(S; \Theta_\e) \to \R,
        \\
        A_\e(\psi, \varphi; t) \coloneqq (\kappa\nabla\psi(t), \nabla\varphi(t))_{ \OmegaIdx{s}} 
        + \frac{1}{\varepsilon} (\kappa\nabla\psi(t), \nabla\varphi(t))_{ \OmegaIdx{f}} 
        +\frac{1}{\e} (v \cdot \nabla\psi(t), \varphi(t))_{\OmegaIdx{f}} 
        + \alphafg (\jump{\psi(t)}, \jump{\varphi(t)})_{ \Gammafs}.
    \end{multline*}
    By the embedding $H^1(\Omega) \hookrightarrow L^4(\Omega)$, the convective term $(v \cdot \nabla \theta_\e, \varphi)_{ \OmegaIdx{f}} $ is well defined. We obtain that $A_\e$ is bounded, to be precise for all $\psi, \varphi \in L^2(S;\Theta_\e)$ and almost all $t\in S$
    \begin{equation*}
        |A_\e(\psi, \varphi; t)| \leq C \left(1 + \frac{1}{\e} + \|v\|_{L^\infty(S; H^1(\OmegaIdx{f}))}\right)
        \|\psi(t)\|_{H^1(\OmegaIdx{f} \cup \OmegaIdx{s})}
        \|\phi(t)\|_{H^1(\OmegaIdx{f} \cup \OmegaIdx{s})},
    \end{equation*}
    for a $C < \infty$.
    Additionally, via integration by parts, we have (suppressing the time variable $t$) 
    \begin{equation}\label{eq:convection_indentity}
        (v \cdot \nabla \psi, \varphi)_{ \OmegaIdx{f}} 
        = 
        - (v \cdot \nabla \varphi, \psi)_{ \OmegaIdx{f}} 
        - (\dive{(v)} \psi, \varphi)_{ \OmegaIdx{f}} 
        + (v \cdot n, \psi \varphi)_{ \partial \OmegaIdx{f}}.
    \end{equation}
    Plugging in $\varphi = \psi$, that $\dive{(v)}=0$ and the boundary conditions on $\partial \OmegaIdx{f}$ into \cref{eq:convection_indentity} leads to 
    \begin{equation*}
        (v \cdot \nabla \psi, \psi)_{ \OmegaIdx{f}} = \frac{1}{2} \int_{\SigmaOutEps}  \uEpsBC \cdot n \, \psi^2 \di{\sigma_x} \geq 0
    \end{equation*}
    for almost all $t \in S$. Therefore, there exists a $\lambda > 0$ such that
    \begin{equation*}
        A_\e(\psi, \psi; t) + \lambda \left(\|\psi(t)\|^2_{L^2(\OmegaIdx{f})} + \|\psi(t)\|^2_{L^2(\OmegaIdx{s})}\right) \geq C \|\psi(t)\|^2_{H^1(\OmegaIdx{f} \cup \OmegaIdx{s})}
    \end{equation*}
    for all $\psi \in L^2(S;\Theta_\e)$ and almost all $t \in S$.
    Given the representation (\ref{eq:parabolic_form}) and the above estimates for the bilinear form $A_\e$, there exists a unique solution $\theta_\e \in L^2(S;\Theta_\e)$ \cite[Chapter 3, Proposition 2.3]{Showalter1997}. 
    
    For the norm estimate, we utilize the above computations of the convection term, use $\varphi=\theta_\e$ in the weak formulation, and integrate afterward over the time interval $(0, t)$ to obtain (with the trace estimates from \Cref{lem:trace_estimate})
    \begin{linenomath*}\begin{equation*}
        \begin{split}
            \frac{1}{2} \|\theta_\e(t) \|^2_{L^2(\OmegaIdx{s})} 
            &+ 
            \frac{1}{2\varepsilon} \|\theta_\e(t)\|^2_{L^2(\OmegaIdx{f})} 
            + 
            \kappa^s \| \nabla \theta_\e\|^2_{L^2((0, t)\times \OmegaIdx{s})}
            +
            \frac{1}{\e} \kappa ^f\| \nabla \theta_\e\|_{L^2((0, t)\times\OmegaIdx{f})}^2
            +
            \alphafg \|\jump{\theta_\varepsilon}\|^2_{L^2((0, t)\times\Gammafs)}
            \\
            &\leq C\|f^s_\varepsilon\|_{L^2((0, t)\times\Gammafs)}
            \|\theta_\e\|_{L^2((0, t); H^1(\OmegaIdx{s}))}
            + \frac{C}{\sqrt{\e}}\|f^f_\varepsilon\|_{L^2((0, t)\times\Gammafs)}
            \|\theta_\e\|_{L^2((0, t); H^1(\OmegaIdx{f}))}
            \\
            &\hspace{10pt}+ \frac{1}{2}\|\theta_{\varepsilon, 0}\|^2_{L^2(\OmegaIdx{s})} + 
            \frac{1}{2\varepsilon}\|\theta_{\e, 0}\|^2_{L^2(\OmegaIdx{f})}.
        \end{split}
    \end{equation*}\end{linenomath*}
    Applying Young's inequality followed by Gronwall’s inequality leads to the desired estimate (\ref{eq:solution_esitmate}), where $C$ is independent of $\e$.
\end{proof}
\begin{lemma}\label{lemma:infinity_estimate}
    Let the same assumptions as in \cref{lem:temperature_existence_bounds} be fulfilled and let $\theta_\e \in L^2(S; \Theta_\e)$ be the solution of \cref{eq:weak_formulation_temp}. Then there exists a $C > 0$ independent of $\e$ such that $\|\theta_\e\|_{L^\infty(S \times \Omega)} \leq C.$
\end{lemma}
\begin{proof}
    To show the estimate, we observe that the solution is linear in the source terms $f_\e^{s,f}$. Therefore, we can show the estimates for the two cases $f_\e^{s}=0, f_\e^{f}\neq0$ and $f_\e^{s}\neq 0, f_\e^{f}=0$ separately. Both cases follow via similar arguments, where we present the case $f_\e^{s}=0, f_\e^{f}\neq0$ and in the end shortly comment on the other case. To avoid cluttering, we do not introduce an additional index to denote this case.

    First, for any $k > 0$ construct the functions
    \begin{equation*}
        \theta_{\e,k}^s \coloneqq \begin{cases}
            \theta_{\e}^s - k \quad &\text{if } \theta_{\e}^s > k,
            \\
            0 \quad &\text{if } |\theta_{\e}^s| \leq k,
            \\
            \theta_{\e}^s + k \quad &\text{if } \theta_{\e}^s < -k,
        \end{cases}
        \quad \text{and} \quad 
        \theta_{\e,k}^f \coloneqq \begin{cases}
            \theta_{\e}^f - k \quad &\text{if } \theta_{\e}^f > k,
            \\
            0 \quad &\text{if } |\theta_{\e}^f| \leq k,
            \\
            \theta_{\e}^f + k \quad &\text{if } \theta_{\e}^f < -k,
        \end{cases}
    \end{equation*}
    and the sets
    \begin{equation*}
        A_{\e, k}^s(\tau) \coloneqq \left\{
            x \in \Omega_\e^s : |\theta_{\e}^s(\tau, x)| > k
        \right\}
        \quad \text{and} \quad 
            A_{\e, k}^f(\tau) \coloneqq \left\{
            x \in \Omega_\e^f : |\theta_{\e}^f(\tau, x)| > k
        \right\}.
    \end{equation*}
    Note $\nabla \theta_{\e,k}^{s,f} = \chi_{A_{\e, k}^{s,f}} \nabla \theta_{\e, k}^{s,f}$. Since $\theta_{\e, k}^{s,f} = (\theta_\e^{s,f} - k)_{+} + (\theta_\e^{s,f} + k)_{-}$, applying \cref{lem:time_derivative_cutoff} to both the positive and negative part yields, $\partial_t \theta_{\e, k}^{s,f} \in L^2(S;H^1(\OmegaIdx{s,f})')$.
    By Assumption \ref{item:A3} we set $k > C_0 = \sup_{\e > 0}\|\theta_{\e, 0}\|_{L^\infty(\Omega)}$. Next, we take $\theta_{\e, k} = (\theta_{\e,k}^s, \theta_{\e,k}^f)$ as a test function in \cref{eq:weak_formulation_temp}, leading, for almost all $\tau \in S$, to
    \begin{equation*}
        \begin{split}
            \langle \partial_t \thetaIdx{s}(\tau), \theta_{\e,k}^s(\tau) \rangle_{H^1(\OmegaIdx{s})'} 
            &+ \frac{1}{\e} \langle \partial_t \thetaIdx{f}(\tau), \theta_{\e,k}^f(\tau) \rangle_{H^1_{\SigmaInEps}(\OmegaIdx{f})'}  
            + \kappa^s  \|\nabla \theta_{\e,k}^s(\tau)\|^2_{L^2(\OmegaIdx{s})} 
            + \frac{1}{\e} \kappa^f  \|\nabla \theta_{\e,k}^f(\tau)\|^2_{L^2(\OmegaIdx{f})} 
            \\
            &+ \frac{1}{\e} (v(\tau) \cdot \nabla\theta_\e(\tau), \theta_{\e,k}^f(\tau))_{\OmegaIdx{f}} 
            + \alpha \|\llbracket \theta_{\e, k}(\tau) \rrbracket\|^2_{L^2(\Gamma_\e)} = (f^f_\e(\tau), \theta_{\e,k}^f(\tau))_{\Gammafs}.            
        \end{split}        
    \end{equation*}
    Integrating over the time interval $(0, t) \subset S$ and using \Cref{eq:truncated_time_derivative} from \Cref{lem:time_derivative_cutoff} we obtain
    \begin{align*}
        \frac{1}{2}\|\theta_{\e,k}^s(t)\|^2_{L^2(\OmegaIdx{s})} 
        &+ \frac{1}{2\e} \|\theta_{\e,k}^f(t)\|^2_{L^2(\OmegaIdx{f})} 
        + \kappa^s  \|\nabla \theta_{\e,k}^s\|^2_{L^2((0, t) \times \OmegaIdx{s})} 
        + \frac{1}{\e} \kappa^f  \|\nabla \theta_{\e,k}^f\|^2_{L^2((0, t) \times \OmegaIdx{f})} 
        \\
        &+ \frac{1}{\e} (v \cdot \nabla\theta_\e, \theta_{\e,k}^f )_{(0, t) \times \OmegaIdx{f}} 
        + \alpha \|\llbracket \theta_{\e, k} \rrbracket\|^2_{L^2((0, t) \times \Gamma_\e)} = (f_\e^f, \theta_{\e,k}^f)_{(0, t) \times \Gammafs}
    \end{align*}
    Since $\nabla \theta_{\e,k}^f = \chi_{A_{\e, k}^f} \nabla \theta_{\e}^f$, for the the convection it holds
    \begin{align*}
        \frac{1}{\e} (v \cdot \nabla\theta_\e, \theta_{\e,k}^f )_{(0, t) \times \OmegaIdx{f}} 
        = \frac{1}{\e} (v \cdot \nabla \theta_{\e,k}^f, \theta_{\e,k}^f )_{(0, t) \times \OmegaIdx{f}}.
    \end{align*}
    \cref{eq:convection_indentity} and $\dive{v} = 0$ yield 
    \begin{align*}
        \frac{1}{\e} (v \cdot \nabla \theta_{\e,k}^f, \theta_{\e,k}^f )_{(0, t) \times \OmegaIdx{f}} =
        \frac{1}{2 \e} (v \cdot n \theta_{\e,k}^f, \theta_{\e,k}^f )_{(0, t) \times \partial \OmegaIdx{f}}
        = 
        \frac{1}{2 \e} (v \cdot n \theta_{\e,k}^f, \theta_{\e,k}^f )_{(0, t) \times \partial \SigmaOutEps} \geq 0.
    \end{align*}
    Therefore we obtain
    \begin{equation}\label{eq:helper_infty_estimates}
        \begin{split}
            \frac{1}{2}\|\theta_{\e,k}^s(t)\|^2_{L^2(\OmegaIdx{s})} 
            &+ \frac{1}{2\e} \|\theta_{\e,k}^f(t)\|^2_{L^2(\OmegaIdx{f})} 
            + \kappa^s  \|\nabla \theta_{\e,k}^s\|^2_{L^2((0, t) \times \OmegaIdx{s})} 
            + \frac{1}{\e} \kappa^f  \|\nabla \theta_{\e,k}^f\|^2_{L^2((0, t) \times \OmegaIdx{f})} 
            \\
            &+ \alpha \|\llbracket \theta_{\e, k} \rrbracket\|^2_{L^2((0, t) \times \Gamma_\e)} 
            \leq  (f_\e^f, \theta_{\e,k}^f)_{(0, t) \times \Gammafs}.            
        \end{split}
    \end{equation}
    By Assumption \ref{item:A4} and the trace estimate in \cref{lem:trace_estimate} the source term on the right-hand side can be estimated by
    \begin{equation}\label{eq:trace_transform_temperature}
        \begin{split}
           \left | (f_\e^f, \theta_{\e,k}^f)_{(0, t) \times \Gammafs} \right | 
           &\leq \|f_\e^f\|_{L^\infty((0, t) \times \Gamma_\e)} \int_0^t \int_{\Gamma_\e} |\theta_{\e, k}^f|
           \di{x}\di{\tau}
           \\
           &\leq \frac{C}{\e} \|f_\e^f\|_{L^\infty((0, t) \times \Gamma_\e)} \left( \|\theta_{\e, k}^f\|_{L^1((0, t) \times \OmegaIdx{f})} + \|\nabla \theta_{\e, k}^f\|_{L^1((0, t) \times \OmegaIdx{f})} \right)
           \\
           &\leq \frac{C_\delta}{\e}\|f_\e^f\|^2_{L^\infty((0, t) \times \Gamma_\e)} \|1\|_{L^2((0, t) \times A_{\e, k}^f)}^2 
           + \frac{C}{\e} \|\theta_{\e, k}^f\|_{L^2((0, t) \times \OmegaIdx{f})}^2 
           + \frac{\delta}{\e} \|\nabla \theta_{\e, k}^f\|_{L^2((0, t) \times \OmegaIdx{f})}^2.            
        \end{split}
    \end{equation}
    In the last step, Young's inequality with a $\delta > 0 $ was used. We choose $\delta$ such that we can include the gradient term on the left side and additionally assume $k \geq \|f_\e\|_{L^\infty((0, t) \times \Gamma_\e)}$, to arrive at
    \begin{align*}
        \frac{1}{\e} \|\theta_{\e,k}^f(t)\|^2_{L^2(\OmegaIdx{f})} 
        + \frac{1}{\e}\|\nabla \theta_{\e,k}^f\|^2_{L^2((0, t) \times \OmegaIdx{f})} 
        \leq \frac{C}{\e} \left [\|\theta_{\e, k}^f\|_{L^2((0, t) \times \OmegaIdx{f})}^2 + k^2 \int_0^t |A_{\e, k}^f(\tau)| \di{\tau}  \right ].
    \end{align*}
    With the extension operator $\mathbb{E}_\e^f$ from \cref{lem:extension_operator}, we transform to the non-rough thin layer. Define $B^f_{\e, k}(\tau) \coloneqq \{x \in \Omega_\e : \mathbb{E}_\e^f \theta^f_{e, k}(\tau, x) \neq 0 \}$, then clearly $|A^f_{\e, k}(\tau)| \leq |B^f_{\e, k}(\tau)|$. Multiplying by $\e$, we obtain that on the non-rough thin layer, it holds
    \begin{align*}
        \|\mathbb{E}_\e^f\theta_{\e,k}^f(t)\|^2_{L^2(\Omega_\e)} 
        + \|\nabla \mathbb{E}_\e^f \theta_{\e,k}^f\|^2_{L^2((0, t) \times \Omega_\e)} 
        \leq C \left [\|\mathbb{E}_\e^f\theta_{\e, k}^f\|_{L^2((0, t) \times \Omega_\e)}^2 + k^2 \int_0^t |B_{\e, k}^f(\tau)| \di{\tau}  \right ],
    \end{align*}
    and consequently by Gronwall's inequality
    \begin{equation*}
        \|\mathbb{E}_\e^f\theta_{\e,k}^f\|_{L^\infty(S;L^2(\Omega_\e))\cap L^2(S;H^1(\Omega_\e))}\leq Ck\left(\int_S |B_{\e, k}^f(\tau)| \di{\tau}\right)^{\frac{1}{2}}.
    \end{equation*}
    Applying \cite[Chapter 2, Theorem 6.1]{Ladyvzenskaja1968}, we now have to be careful regarding the $\e$ dependency of the estimate. First we obtain 
    \begin{equation*}
        \|\mathbb{E}_\e^f \theta_{\e,k}^f\|_{L^\infty(S\times\Omega_\e)}\leq C\left(1+C\beta_\e^{\frac{d+2}{2}}|S|^{\frac{1}{2}}|\Omega_\e|^{\frac{1}{2}}\right)\leq C\left(1+C\e^\frac{1}{2}\beta_\e^{\frac{d+2}{2}}\right),
    \end{equation*}
    where $\beta_\e= \mathcal{O}(\e^{-\frac{1}{d+2}})$ is the embedding constant from \cref{lem:embedding_constant}.
    Therefore $\|\mathbb{E}_\e^f \theta_{\e}^f\|_{L^\infty(S\times\Omega_\e)}\leq C$ independent of $\e$.
    Consequently $\|\theta_{\e}^f\|_{L^\infty(S\times\OmegaIdx{f})} = \|\mathbb{E}_\e^f \theta_{\e}^f\|_{L^\infty(S\times\OmegaIdx{f})} \leq C$, and $\|\theta_{\e}^f\|_{L^\infty(S\times \Gammafs)} \leq C$.

    Plugging this result back into \Cref{eq:helper_infty_estimates} and choosing $k$ large enough, we obtain
    \begin{equation*}
        \|\theta_{\e, k}^s(t)\|_{L^2(\OmegaIdx{s})} = 0 \quad \text{ for almost all } t \in S,
    \end{equation*}
    resulting in $\theta_{\e, k}^s = 0$ in $L^2(S\times \OmegaIdx{s})$ and consequently $\thetaIdx{s} \in L^\infty(S\times \OmegaIdx{s})$.

    As mentioned at the beginning, the estimate for $f_\e^{s}\neq 0, f_\e^{f}=0$ follows by the same steps only with the roles of $\thetaIdx{s}$ and $\thetaIdx{f}$ switched, the other case is also slightly simpler since the estimates in the solid domain do not depend on $\e$. Adding the solutions for both cases yields the solution for the general case $f_\e^{s}\neq 0, f_\e^{f}\neq0$ and consequently the general $L^\infty$ estimate.
\end{proof}
\begin{lemma}\label{lemma:time_derivative_temp}
    Let the same assumptions be fulfilled as in \cref{lem:temperature_existence_bounds} and let $\theta_\e \in L^2(S; \Theta_\e)$ be the solution of \cref{eq:weak_formulation_temp}, with $\partial_t \theta_\e \in L^2(S; \Theta_\e')$. 
    If $\|v_\e\|_{L^2(S \times \OmegaIdx{f})} \leq C \sqrt{\e}$, then there exists a $C > 0$ independent of $\e$ such that
    \begin{linenomath*}\begin{equation}
    \label{eq:time_derivative_estimate}
        \begin{split}
            \|\partial_t \theta_\e\|_{L^2(S; H^{1}(\OmegaIdx{s})')}  \leq C
            \quad \text{ and } \quad 
            \|\partial_t \theta_\e\|_{L^2(S; H^{1}_{\SigmaInEps}(\OmegaIdx{f})')}
            \leq C \sqrt{\e}.
        \end{split}
    \end{equation}\end{linenomath*}
\end{lemma}
\begin{proof}
    By testing \cref{eq:weak_formulation_temp} with arbitrary $\varphi \in L^2(S;\Theta_\e)$ one obtains for almost all $t \in S$
    \begin{linenomath*}\begin{equation*}
        \begin{split}
            &\bigl | \langle \partial_t \theta_\e, \varphi \rangle_{H^1(\OmegaIdx{s})'} 
            +
            \frac{1}{\e} \langle \partial_t \theta_\e, \varphi\rangle_{H^1_{\SigmaInEps}(\OmegaIdx{f})'}\bigr | 
            \\
            &\leq
            |(\kappa \nabla \theta_\e, \nabla\varphi)_{ \OmegaIdx{s}}|
            +
            \frac{1}{\e}|(\kappa \nabla\theta_\e, \nabla \varphi)_{\OmegaIdx{f}}|
            +
            \alphafg |(\jump{\theta_\varepsilon}, \jump{\varphi})_{\Gammafs} |
            +
            |(f_\varepsilon^s, \varphi^s)_{\Gammafs}|
            +
            |(f_\varepsilon^f, \varphi^f)_{\Gammafs}|
            + |\frac{1}{\e} (v_\e \cdot \nabla \theta_\e, \varphi)_{\OmegaIdx{f}}|
            \\
            &\leq \kappa^s
            \|\nabla \theta_\e\|_{L^2(\OmegaIdx{s})} \|\nabla \varphi^{g}\|_{L^2(\OmegaIdx{s})} 
            +
            \frac{\kappa^f}{\e} \|\nabla \theta_\e\|_{L^2(\OmegaIdx{f})} \|\nabla \varphi^{f}\|_{L^2(\OmegaIdx{f})} 
            + |\frac{1}{\e} (v_\e \cdot \nabla \theta_\e, \varphi)_{\OmegaIdx{f}}|
            \\
            &\hspace{10pt} + \alpha \|\jump{\theta_\e}\|_{L^2(\Gammafs)} (\|\varphi^f\|_{L^2(\Gammafs)} + \|\varphi^s\|_{L^2(\Gammafs)}) +
            C\|f_\varepsilon^s\|_{L^2(\Gammafs)}
            \|\varphi^s\|_{H^1(\OmegaIdx{s})}
            +\|f_\varepsilon^f\|_{L^2(\Gammafs)}
            \|\varphi^f\|_{L^2(\Gammafs)}.
        \end{split}
    \end{equation*}\end{linenomath*}
    Choosing $\varphi^f = 0$, integrating over the time interval $S$, followed by Hölder's inequality and applying (\ref{eq:solution_esitmate}) gives the desired estimate for $\partial_t \theta_\e^s$. 

    For estimating $\partial_t \thetaIdx{f}$ we set $\varphi^s = 0$ then the only problematic part is the convection term (the trace over $\Gamma_\e$ can be estimated by \cref{lem:trace_estimate}). By \cref{eq:convection_indentity} and $\dive{v_\e} = 0$
    \begin{linenomath*}\begin{equation*}
        \begin{split}
             \frac{1}{\e} (v_\e \cdot \nabla \theta_\e, \varphi^f)_{S \times \OmegaIdx{f}} 
             = 
             - \frac{1}{\e} (v_\e \theta_\e, \nabla \varphi^f)_{S \times\OmegaIdx{f}}  
             + 
             \frac{1}{\e} (v_\e \cdot n \theta_\e, \varphi^f)_{S \times \partial \OmegaIdx{f}}  .
        \end{split}
    \end{equation*}\end{linenomath*}
    For the first term on the right-hand side, we use the results from \cref{lemma:infinity_estimate} and the assumption on $v_\e$ to obtain
    \begin{equation*}
           \left|\frac{1}{\e} (v_\e \cdot \theta_\e, \nabla\varphi^f)_{S \times\OmegaIdx{f}} \right| 
           \leq \frac{1}{\e} \|\theta\|_{L^\infty(S \times \OmegaIdx{f})} \|v_\e\|_{L^2(S \times \OmegaIdx{f})} \|\nabla \varphi^f\|_{L^2(S \times \OmegaIdx{f})}  
           \leq \frac{C}{\sqrt{\e}} \|\varphi^f\|_{L^2(S; H^1(\OmegaIdx{f}))}.
    \end{equation*}
    For the second term, we again use our choice of boundary conditions, leading to
    \begin{linenomath*}\begin{equation*}
        \begin{split}
             \left| \frac{1}{\e} (v_\e \cdot n \theta_\e, \varphi^f)_{S \times \partial \OmegaIdx{f}} \right|
             \leq \frac{1}{\e} \|v_\e\|_{L^\infty(S \times \partial \OmegaIdx{f})} \|\theta_\e\|_{L^2(S \times \SigmaOutEps)} \|\varphi^f\|_{L^2(S \times \SigmaOutEps)}.
        \end{split}
    \end{equation*}\end{linenomath*}
    With \cref{lem:trace_estimate}. $(iii)$ and the estimates of $\thetaIdx{f}$ from \cref{lem:temperature_existence_bounds}, we obtain
    \begin{linenomath*}\begin{equation*}
        \begin{split}
             \left| \frac{1}{\e} (v_\e \cdot n \theta_\e, \varphi^f)_{S \times \partial \OmegaIdx{f}} \right|
             \leq \frac{C}{\e} \|\theta_\e\|_{L^2(S \times \SigmaOutEps)} \|\varphi^f\|_{L^2(S \times \SigmaOutEps)}
             \leq \frac{C}{\sqrt{\e}} \|\varphi^f\|_{L^2(S; H^1(\OmegaIdx{f}))} .
        \end{split}
    \end{equation*}\end{linenomath*}
    Multiplying the complete inequality by $\e$ we get the desired result.
\end{proof}
The additional assumption $\|v_\e\|_{L^2(S \times \OmegaIdx{f})} \leq C \sqrt{\e}$ is not a new requirement to the problem, as the solution of the Stokes problem satisfies this estimate naturally. 
\begin{lemma}\label{lemma:existence_bounds_stokes}
    Let the Assumptions \ref{item:A2} and \ref{item:A5} be fulfilled. Then for all $\e > 0$ and $\psi_\e \in L^2(S \times \OmegaIdx{f})$ there exists a unique solution $(u_\e, p_\e) \in V_\e \times Q_\e$ of \cref{eq:weak_formulation_stokes}. In addition, it holds
    \begin{linenomath*}\begin{equation}\label{eq:solution_esitmate_stokes}
        \frac{1}{\sqrt{\e}}\|u_\e\|_{L^\infty(S; L^2(\OmegaIdx{f}))} 
        + 
        \sqrt{\e} \|\nabla u_\e\|_{L^\infty(S; L^2(\OmegaIdx{f}))} \leq C
        \quad \text{ and } \quad 
        \sqrt{\e}\|p_\e\|_{L^\infty(S; L^2(\OmegaIdx{f}))} \leq \frac{C}{\e}
    \end{equation}\end{linenomath*}
    for an $C<\infty$ independent of $\varepsilon$ and $\psi_\e$.
\end{lemma}
\begin{proof}
    Existence and uniqueness follow from \cite[Theorem 5.1]{Girault_1981}, where the proof for constant viscosity is stated, which is transferable to our setup given the bounds assumed in \ref{item:A2}. 
    Furthermore, by Assumption \ref{item:A5}, and the extension of the boundary function $U_e$, the velocity estimates follow by testing with $u_\e - U_\e$ and the Poincaré inequality, as carried out in \cite[Section 3.2]{Bayda_1989}. 
    The pressure estimate follows by combining the estimate of the velocity gradient with the results in \cite[Lemma 4.7]{Fabricius2023} for the Bogovskii operator in thin porous layers (which can be directly transferred to our setup). See also \cite{Fabricius2022} for further studies of the Bogovskii operator in thin non-porous layers.
\end{proof}

Next, we introduce the operators $\mathbb{T}_\e: V_\e \to L^2(S \times \Omega)$, which maps a prescribed velocity field to the corresponding temperature field via \cref{eq:weak_formulation_temp}, and the Stokes operator $\mathbb{S}_\e: L^2(S \times \Omega) \to V_\e$ of \cref{eq:weak_formulation_stokes}, which returns the velocity field for a given temperature distribution.
We are looking for fixed points of the operator
\begin{equation}\label{eq:solution_operator}
    \mathbb{F}_\e : L^2(S \times \Omega) \to L^2(S \times \Omega), \quad \psi_\e \mapsto \mathbb{T}_\e\left(\mathbb{S}_\e(\psi_\e)\right).
\end{equation}
and point out that any such fixed point is automatically a weak solution to the full problem.
In the following, we first investigate the continuity properties of $\mathbb{T}_\e$ and $\mathbb{S}_\e$.
\begin{lemma}\label{lem:stokes_continuous}
    Let the Assumptions \ref{item:A2} and \ref{item:A5} be fulfilled.
    Given a sequence $\{\psi_{k}\}_{k \in \N} \subset L^2(S \times \Omega)$ with $\psi_k \xrightarrow[]{k \to \infty} \psi$ in $L^2(S \times \Omega)$, it holds
    \begin{equation*}
        \mathbb{S}_\e(\psi_k) \xrightharpoonup[]{k\to \infty}\mathbb{S}_\e(\psi) \quad \text{ in } \quad L^2(S; H^1(\OmegaIdx{f})^d).
    \end{equation*}
\end{lemma}
\begin{proof}
    For each $\psi_k$, there exists a unique solution $(u_k, p_k) \in V_\e \times Q_\e$ with $\|u_k\|_{L^\infty(S; H^1(\OmegaIdx{f}))} + \|p_k\|_{L^\infty(S; L^2(\OmegaIdx{f})} \leq C_\e$, where $C_\e$ denotes a constant that may depend on $\e$.
    Given this bound, we can find a $u \in L^2(S; H^1(\OmegaIdx{f})^d)$ and a $p \in L^2(S\times \OmegaIdx{f})$ such that there is a subsequence of $(u_k, p_k)$, still denoted with $k$, that weakly converges:
    \begin{equation*}
        u_k \rightharpoonup u \quad \text{in } L^2(S; H^1(\OmegaIdx{f})^d) 
        \quad \text{ and } \quad 
        p_k \rightharpoonup p \quad \text{in } L^2(S; \OmegaIdx{f}).
    \end{equation*}
    Given the Lipschitz continuity and bounds of $\mu$, it holds $\mu(\psi_k) \to \mu(\psi)$ in $L^2(S \times \Omega)$. Therefore, we obtain the weak convergence 
    \begin{equation*}
        \mu(\psi_k) \nabla u_k \rightharpoonup \mu(\psi) \nabla u \quad \text{in } L^1(S \times \OmegaIdx{f})^d. 
    \end{equation*}    
    Since the viscosity is in $L^\infty(S \times \OmegaIdx{f})$, we have that $\mu(\psi_k) \nabla u_k$ and $\mu(\psi) \nabla u$ are bounded in $L^2(S \times \OmegaIdx{f})^d$. This lets us conclude
    \begin{equation*}
        \mu(\psi_k) \nabla u_k \rightharpoonup \mu(\psi) \nabla u \quad \text{in } L^2(S \times \OmegaIdx{f})^d, 
    \end{equation*}
    for a subsequence. We are now able to pass to the limit in the weak formulation and obtain 
    \begin{align*}
        0 &= \lim_{k\to \infty} \left[(\mu(\psi_k) \nabla u_k , \nabla v)_{S \times \OmegaIdx{f}} 
        - (p_k, \dive{(v)})_{S \times \OmegaIdx{f}} 
       \right]
        = (\mu(\psi) \nabla u , \nabla v)_{S \times \OmegaIdx{f}} 
        - (p, \dive{(v)})_{S \times \OmegaIdx{f}} ,
    \end{align*}
    for all $v \in H^1_0(\OmegaIdx{f})^d$.
    By localization in time, we see that, for almost all $t\in S$, $(u, p)$ is the solution to the Stokes problem (\ref{eq:weak_formulation_stokes}) in $V_\e \times Q_\e$ corresponding to $\psi$. Because the solution is unique, the whole sequence weakly converges to the same limit.
\end{proof}
\begin{lemma}\label{lem:temp_continuous}
    Let Assumptions \ref{item:A1} and \ref{item:A3}--\ref{item:A5} be satisfied.
    Given a sequence $\{u_{k}\}_{k \in \N} \subset V_\e$ and $u \in V_\e$ such that $\sup_{k\in\N}\|u_k\|_{V_\e}<\infty$ and $u_k \xrightharpoonup[]{k \to \infty} u$ in $L^2(S; H^1(\OmegaIdx{f}))$, it holds
    \begin{equation*}
        \mathbb{T}_\e(u_k) \xrightarrow[]{k\to \infty}\mathbb{T}_\e(u) \quad \text{ in } \quad L^2(S \times \Omega).
    \end{equation*}
    %
\end{lemma}
\begin{proof}
    Given the assumptions, the sequences $\{\theta_k\}_{k \in \N}\coloneqq\{\mathbb{T}_\e(u_k) \}_{k \in \N} \subset L^2(S;\Theta_\e)$ and $ \{\partial_t \theta_k\}_{k \in \N} \subset L^2(S;\Theta_\e')$ are bounded uniformly in $k\in\N$ (see \cref{lem:temperature_existence_bounds}).
    Therefore, there is a subsequence, still indexed with $k$, and a weak limit $\theta \in L^2(S;\Theta_\e)$ with $\partial_t\theta \in L^2(S;\Theta_\e')$ such that $\theta_k \rightharpoonup \theta$ in $ L^2(S;\Theta_\e)$ and $\partial_t\theta_k \rightharpoonup \partial_t\theta$ in $L^2(S;\Theta_\e')$.

    Note that the embeddings $H^1(\OmegaIdx{j}) \hookrightarrow L^2(\OmegaIdx{j})$ are compact and the embeddings $L^2(\OmegaIdx{j}) \hookrightarrow H^{1}(\OmegaIdx{j})'$ are continuous for $j=f,g$.
    Therefore, we obtain strong convergence $\theta_k \to \theta$ in $L^2(S \times \Omega)$ by the Aubin-Lions lemma. 

    It remains to be shown that $\theta = \mathbb{T}_\e(u)$.
    Looking at the weak form \cref{eq:weak_formulation_temp}, we can immediately pass to the limit in most terms with only the convection term as a potential issue.
    Now, since $\theta_k\to\theta$ strongly in $L^2(S\times\Omega_\e^f)$ and $u_k\rightharpoonup u$ weakly in $L^2(S\times\Omega_\e^f)^d$, we have $u_k\theta_k\rightharpoonup u\theta$ weakly in $L^1(S\times\Omega_\e^f)^d$.   
   Via the embedding $H^1(\OmegaIdx{f}) \hookrightarrow L^4(\OmegaIdx{f})$ and $\sup_{k\in\N}\|u_k\|_{V_\e}<\infty$, we also have
    \begin{equation*}
        \|u_k \theta_k\|_{L^2(S \times \OmegaIdx{f})}
        \leq 
        C_\e \|u_k\|_{L^\infty(S; H^1(\OmegaIdx{f}))}\|\theta_k\|_{L^2(S; H^1(\OmegaIdx{f}))} \leq C_\e
    \end{equation*}
    where $C_\e>0$ may depend on $\e$ but not on $k$.
    This implies that, at least along a subsequence,  $u_k\theta_k$ also converges weakly in $L^2(S\times\OmegaIdx{f})^d$ which then implies $u_k\theta_k\rightharpoonup u\theta$ in $L^2(S\times\OmegaIdx{f})^d$.

    This allows us to pass to the limit $k \to \infty$ in the weak formulation and deduce that $\theta$ is the solution corresponding to the velocity field $u$, i.e., $\theta=\mathbb{T}_\e(u)$.
    Finally, since the solution to the temperature problem is unique, the whole sequence converges. 
\end{proof}
Combining the previous two lemmas, we obtain the continuity of the concatenated operator $\mathbb{F}_\e$.
\begin{corollary}\label{corrollary_continuity}
    The operator $\mathbb{F}_\e$ defined in \cref{eq:solution_operator} is continuous from $L^2(S\times\Omega)$ to $L^2(S\times\Omega)$.
\end{corollary}
With the continuity, we can now tackle the existence of a fixed point. 
\begin{theorem}[Existence]\label{thm:solution_operator_exitsence}
    Let Assumptions \ref{item:A1}-\ref{item:A5} be satisfied. Then, for all $\e > 0$, the operator $\mathbb{F}_\e$ defined in \cref{eq:solution_operator} has a fixed point in $L^2(S;\Theta_{\e})$.
\end{theorem}
\begin{proof}
    By \cref{lem:temperature_existence_bounds} and \ref{lemma:time_derivative_temp} there exists a $\e$-dependent $C_\e > 0$ such that 
    \begin{equation*}
        \|\mathbb{F}_\e(\psi) \|_{L^2(S; H^1(\OmegaIdx{s} \cup \OmegaIdx{f}))} + \|\partial_t (\mathbb{F}_\e(\psi))\|_{L^2(S; \Theta_\e')} \leq C_\e, \quad \text{ for all } \psi \in L^2(S\times\Omega).
    \end{equation*}
    Define the set
    \begin{equation*}
        M \coloneqq \left\{\psi \in L^2(S; \Theta_\e) : \partial_t \psi \in L^2(S; \Theta_\e'), 
        \|\psi\|_{L^2(S; H^1(\OmegaIdx{s} \cup \OmegaIdx{f}))} + \|\partial_t \psi\|_{L^2(S; \Theta_\e')} \leq C_\e\right\}.
    \end{equation*}
    By construction, $M$ is closed, bounded, convex, and nonempty. Furthermore, by \cref{lem:temperature_existence_bounds} and \ref{lemma:existence_bounds_stokes} it holds $\mathbb{F}_\e : M \to M$.
    We know that $\mathbb{F}_\e$ is $L^2(S\times\Omega)$-continuous (\cref{corrollary_continuity}).
    As in the proof of \cref{lem:temp_continuous}, we can use the Aubin-Lions lemma to conclude that $M$, and also $\mathbb{F}_\e(M)$, are compact in $L^2(S \times \Omega)$.
    Therefore, we apply Schauder's fixed point theorem \cite[Theorem 1.C]{zeidler1999applied} and obtain the existence of at least one fixed point of the operator $\mathbb{F}_\e$ in $M$, and subsequently in $L^2(S;\Theta_{\e})$.
\end{proof}
%
The previous theorem ensures that, under the given assumptions, there exists a weak solution $(\theta_\e, u_\e, p_\e)$ of the problem (\ref{eq:epsilon-problem}) for all $\e > 0$. By construction of the operator $\mathbb{F}_\e$ the fixed point also fulfills the estimates of \cref{lem:temperature_existence_bounds} and \cref{lemma:existence_bounds_stokes}. To ensure uniqueness, further assumptions on the regularity of the solution are required, as denoted in the following theorem.
\begin{theorem}[Uniqueness]\label{thm:eps_problem_uniqueness}
    Let Assumptions \ref{item:A1}--\ref{item:A5} be satisfied.
    If there is a solution $(\theta_\e,u_\e,p_\e)\in L^2(S;\Theta_\e)\times V_\e\times Q_\e$ of the coupled system \cref{eq:weak_formulation_temp} and (\ref{eq:weak_formulation_stokes}), such that
    \begin{itemize}
        \item if $d=2:$ $\nabla \theta_\e\in L^{\infty}(S;L^{1+\delta}(\Omega_\e^f))^2$ and $\nabla u_\e\in L^{\infty}(S;L^{2+\delta}(\Omega_\e^f))^{2\times 2}$,
        \item if $d=3:$ $\nabla \theta_\e\in L^{\infty}(S;L^{\nicefrac{3}{2} + \delta }(\Omega_\e^f))^3$ and $\nabla u_\e\in L^{\infty}(S;L^{3 + \delta}(\Omega_\e^f))^{3\times 3}$,
    \end{itemize}
    for arbitrary $\delta > 0$.
    Then, this is the only solution to the problem.
\end{theorem}
\begin{proof}
We provide the proof in the case $d=3$, the other case follows by the same steps but with a better estimate in the Gagliardo-Nirenberg inequality.
Suppose there are two sets of solutions $(\theta_\e^i,u_\e^i,p_\e^i)\in L^2(S; \Theta_\e) \times V_\e\times Q_\e$ ($i=1,2$) with their differences denoted by $(\overline{\theta_\e},\overline{u_\e},\overline{p_\e})$.
Furthermore, let $(\theta_\e^1,u_\e^1)$ satisfy the additional gradient regularity.
Using the weak formulations, it holds
\begin{linenomath*}
    \begin{multline*}
        \langle\partial_t \overline{\theta_\e}, \varphi\rangle_{H^1(\OmegaIdx{s})'} 
        + (\kappa^s\nabla\overline{\theta_\e}, \nabla\varphi)_{ \OmegaIdx{s}} 
        + \frac{1}{\varepsilon}\langle \partial_t \overline{\theta_\e}, \varphi\rangle_{H^1_{\SigmaInEps}(\OmegaIdx{f})'} 
        + \frac{1}{\varepsilon} (\kappa^f\nabla\overline{\theta_\e}, \nabla\varphi)_{ \OmegaIdx{s}} 
        \\
        +\frac{1}{\e} ( \overline{u_\e} \cdot \nabla{\theta_\e}^1, \varphi)_{ \OmegaIdx{f}}
        +\frac{1}{\e} ( u_\e^2 \cdot \nabla\overline{\theta_\e}, \varphi)_{ \OmegaIdx{f}}
        + \alphafg (\jump{\overline{\theta_\varepsilon}}, \jump{\varphi})_{ \Gammafs} 
        =0,
    \end{multline*}
\end{linenomath*}
\begin{linenomath*}\begin{equation*}
    \begin{split}
        &(\mu(\theta_\e^2) \nabla \overline{u_\e} , \nabla v)_{ \OmegaIdx{f}}+((\mu(\theta_\e^2)-\mu(\theta_\e^1))) \nabla u_\e^1 , \nabla v)_{ \OmegaIdx{f}} - (\overline{p_\e}, \dive{v})_{ \OmegaIdx{f}} = 0.
    \end{split}
\end{equation*}\end{linenomath*}
for all test functions $(\varphi,v)\in L^2(S;\Theta_\e)\times H^1_0(\OmegaIdx{f})^d$ and almost all $t\in S$.
Now, choosing $(\varphi,v)=(\overline{\theta_\e},\overline{u_\e})$ as test functions, we obtain
\begin{linenomath*}
    \begin{multline*}
        \frac{1}{2}\ddt\|\overline{\theta_\e}\|_{L^2(\Omega_\e^s)}^2
        + \kappa_g\|\nabla\overline{\theta_\e}\|_{L^2(\Omega_\e^s)}^2
        +\frac{1}{2\varepsilon}\ddt\|\overline{\theta_\e}\|_{L^2(\Omega_\e^f)}^2
        + \frac{\kappa_f}{\varepsilon}\|\nabla\overline{\theta_\e}\|_{L^2(\Omega_\e^f)}^2
        + \alphafg \|\overline{\theta_\varepsilon}\|_{L^2(\Gammafs)}^2
        +\frac{1}{\e} ( u_\e^2 \cdot \nabla\overline{\theta_\e}, \overline{\theta_\e})_{ \OmegaIdx{f}}
        \\
        =-\frac{1}{\e} (\overline{u_\e} \cdot \nabla\theta_\e^1, \overline{\theta_\e})_{ \OmegaIdx{f}},
    \end{multline*}
\end{linenomath*}
\begin{linenomath*}\begin{equation*}
    \begin{split}
        &\underline{\mu}\|\nabla \overline{u_\e}\|_{L^2(\Omega_\e^f)}^2=
        ((\mu({\theta_\e}^1)-\mu({\theta_\e}^2)) \nabla u_\e^1 , \nabla\overline{u_\e})_{ \OmegaIdx{f}}.
    \end{split}
\end{equation*}\end{linenomath*}
Take $r = \nicefrac{2(3+\delta)}{1+\delta}$, then $1=\nicefrac{1}{2} + \nicefrac{1}{r} + \nicefrac{1}{3+\delta}$. 
Using the Lipschitz property of the viscosity and the Hölder and the Poincaré inequalities, we are led to
\begin{linenomath*}
\begin{equation}\label{eq:unique_estimate_velocity}
    \|\overline{u_\e}\|_{H^1(\Omega_\e^f)}^2
    \leq C\|\overline{\theta_\e}\|_{L^{r}(\Omega_\e^f)}^2\|\nabla u_\e^1\|_{{L^{3+\delta}}(\Omega_\e^f)}^2.
\end{equation}
\end{linenomath*}
For the heat problem, we first note that
\[
(u_\e^2 \cdot \nabla\overline{\theta_\e}, \overline{\theta_\e})_{ \OmegaIdx{f}} \geq 0
\]
via \cref{eq:convection_indentity}, resulting in (for $q = \nicefrac{18+12\delta}{3+10\delta}$)
\begin{linenomath*}
    \begin{equation}
        \begin{split}\label{eq:unique_estimate_heat}
        \ddt\|\overline{\theta_\e}\|_{L^2(\Omega_\e^s)}^2
        +\|\nabla\overline{\theta_\e}\|_{L^2(\Omega_\e^s)}^2
        +\ddt\|\overline{\theta_\e}\|_{L^2(\Omega_\e^f)}^2
        +\|\nabla\overline{\theta_\e}\|_{L^2(\Omega_\e^f)}^2
        +\|\jump{\overline{\theta_\varepsilon}}\|_{L^2(\Gammafs)}^2
        \\
        \leq C_\e\|\overline{u_\e}\|_{L^6(\Omega_\e^f)}\|\overline{\theta_\e}\|_{L^{q}(\Omega_\e^f)}\|\nabla{\theta_\e}^1\|_{{L^{\nicefrac{3}{2}+ \delta }}(\Omega_\e^f)}.
        \end{split}
    \end{equation}  
\end{linenomath*}
Summing \cref{eq:unique_estimate_velocity,eq:unique_estimate_heat} and integrating over time, we get
\begin{linenomath*}
    \begin{equation}\label{eq:inbetween_unique}
        \begin{split}
            \|\overline{\theta_\e}(t)\|_{L^2(\Omega_\e^s)}^2
            +\|\overline{\theta_\e}(t)\|_{L^2(\Omega_\e^f)}^2
            +\int_0^t\left(\|\nabla\overline{\theta_\e}\|_{L^2(\Omega_\e^s)}^2
            +\|\nabla\overline{\theta_\e}\|_{L^2(\Omega_\e^f)}^2
            +\|\jump{\overline{\theta_\varepsilon}}\|_{L^2(\Gammafs)}^2
            + \|\overline{u_\e}\|_{H^1(\Omega_\e^f)}^2\right)\di{\tau}\\
            \leq C_\e\int_0^t\left(\|\overline{u_\e}\|_{L^{6}(\Omega_\e^f)}\|\overline{\theta_\e}\|_{L^{q}(\Omega_\e^f)}\|\nabla{\theta_\e}^1\|_{{L^{\nicefrac{3}{2}+\delta}}(\Omega_\e^f)}
            +\|\overline{\theta_\e}\|_{{L^{r}}(\Omega_\e^f)}^2\|\nabla u_\e^1\|_{{L^{3+\delta}}(\Omega_\e^f)}^2\right)\di{\tau}.
        \end{split}
    \end{equation}
    Focusing on the second term on the right-hand side of \Cref{eq:inbetween_unique}, we have
    \[
    \int_0^t\|\overline{\theta_\e}\|_{{L^{r}}(\Omega_\e^f)}^2\|\nabla u_\e^1\|_{{L^{3+\delta}}(\Omega_\e^f)}^2\di{\tau}
    \leq \|\nabla u_\e^1\|_{L^{\infty}(S;{L^{3+\delta}}(\Omega_\e^f))}^2\|\overline{\theta_\e}\|_{L^{2}(S;{L^{r}}(\Omega_\e^f))}^2.
    \]
    The Gagliardo-Nirenberg interpolation inequality with $\lambda = \nicefrac{3}{3+\delta}$, yields
    \begin{equation*}
        \|\overline{\theta_\e}\|_{L^{2}(S;{L^{r}}(\Omega_\e^f))} \leq C \|\nabla \overline{\theta_\e}\|^\lambda_{L^2(S \times \OmegaIdx{f})} \|\overline{\theta_\e}\|^{1-\lambda}_{L^2(S \times \OmegaIdx{f})}.
    \end{equation*}    
    Using $\nabla u_\e^1\in L^{\infty}(S;L^{3+\delta}(\Omega_\e^f))^d$ and the scaled Young's inequality leads to
    \[
    \int_0^t\|\overline{\theta_\e}\|_{{L^{r}}(\Omega_\e^f)}^2\|\nabla u_\e^1\|_{{L^{3+\delta}}(\Omega_\e^f)}^2\di{\tau}
    \leq  \frac{1}{4}\|\nabla\overline{\theta_\e}\|_{L^{2}(S\times\Omega_\e^f)}^{2}+C_\e\|\overline{\theta_\e}\|_{L^{2}(S\times\Omega_\e^f)}^{2}.
    \]
    For the first term on the right-hand side of \Cref{eq:inbetween_unique}, we use $\nabla \theta_\e^1\in L^{\infty}(S;{L^{\nicefrac{3}{2} + \delta}(\Omega_\e^f)})$ to estimate
    \[
    \int_0^t\|\overline{u_\e}\|_{L^{6}(\Omega_\e^f)}\|\overline{\theta_\e}\|_{{L^{q}}(\Omega_\e^f)}\|\nabla\theta_\e^1\|_{L^{\nicefrac{3}{2} + \delta }(\Omega_\e^f)}\di{\tau}
    \leq C_\e\|\overline{u_\e}\|_{L^2(S;H^1(\Omega_\e^f))}\|\overline{\theta_\e}\|_{L^2(S;L^6(\Omega_\e^f))}
    \]
    which (again using Young's inequality and the Gagliardo-Nirenberg interpolation) gives
    \[
    C_\e \int_0^t\|\overline{u_\e}\|_{L^{6}(\Omega_\e^f)}\|\overline{\theta_\e}\|_{{L^{q}}(\Omega_\e^f)}\|\nabla\theta_\e^1\|_{L^{\nicefrac{3}{2} + \delta }(\Omega_\e^f)}\di{\tau}
    \leq
    \frac{1}{2}\|\overline{u_\e}\|_{L^2(S;H^1(\Omega_\e^f))}^2
    +\frac{1}{4}\|\nabla\overline{\theta_\e}\|_{L^{2}(S\times\Omega_\e^f)}^{2}
    +C_\e\|\overline{\theta_\e}\|_{L^{2}(S\times\Omega_\e^f)}^{2}.
    \]
    Putting everything together and utilizing Gronwall's inequality shows $\overline{u_\e}=0$ and $\overline{\theta_\e}=0$.
    
\end{linenomath*}
\end{proof}
In particular, any classical solution must be unique via \cref{thm:eps_problem_uniqueness}.
We point to an analogous statement for the limit problem (\cref{thm:homogenized_model_uniqueness}).


\section{Homogenization}\label{sec:homogenization}
In this section, we investigate the limit behavior for $\e\to0$ of the microscopic solution $(\theta_\e,u_\e,p_\e)$ and identify the macroscopic limit problem which is solved by this limit (see~\cref{thm:homogenized_model}).
For this limit analysis, we utilize the concept of two-scale convergence for thin layers first introduced in
\cite[Definition 4.1]{Neuss07} and further developed in following articles, see for example \cite{Bhattacharya2022, Gahn17} and the references therein.
We start by introducing the definition, then consider the existence of limit functions, and then pass to the limit in the weak formulation. Lastly, we analyze the existence and uniqueness of solutions to the effective model.
\subsection{Two scale convergence}
\begin{definition}[Two-scale convergence on thin domains]\label{def:two_scale}\hspace{5pt}
\begin{itemize}
    \item[i)] A sequence $u_\varepsilon\in L^2(S\times \Omega_\e)$ is said to weakly two-scale converge to a function $u \in L^2(S\times \Sigma \times Y^{d})$ (notation $u_\varepsilon \twosc u$) if 
    \begin{linenomath*}\begin{equation}
        \lim_{\varepsilon\to 0} \frac{1}{\varepsilon}\int_S \int_{\Omega_\e} u_\varepsilon(t, x)\varphi\left(t, \Tilde{x}, \nicefrac{x}{\varepsilon}\right) \di{x}\di{t} = \int_S \int_{\Sigma} \int_{Y^d} u(t,\Tilde{x},y)\varphi(t,\Tilde{x},y) \di{y}\di{\Tilde{x}}\di{t},
    \end{equation}\end{linenomath*}
    for all $\varphi \in C(\overline{S \times \Sigma}; C_{\#}(Y^d))$. If in addition 
    \begin{equation*}
        \lim_{\e \to 0} \frac{1}{\sqrt{\e}} \|u_\e\|_{L^2(S\times \Omega_\e)} = \|u\|_{L^2(S\times \Sigma \times Y^d)} 
    \end{equation*}
    then the sequence is said to strongly two-scale converge and we write $u_\varepsilon \stwosc u$.
    \item[ii)] A sequence $u_\varepsilon\in L^2(S\times \Gammafs)$ is said to weakly two-scale converge to a function $u \in L^2(S\times \Sigma \times \Gamma)$ if 
    \begin{linenomath*}\begin{equation}
        \lim_{\varepsilon\to 0} \int_S \int_{\Gammafs} u_\varepsilon(t, x)\varphi\left(t, \Tilde{x}, \nicefrac{x}{\varepsilon}\right) \di{\sigma_x}\di{t} = \int_S \int_{\Sigma} \int_{\Gamma} u(t,\Tilde{x},y)\varphi(t,\Tilde{x},y) \di{\sigma_y}\di{\Tilde{x}}\di{t},
    \end{equation}\end{linenomath*}
    for all $\varphi \in C(\overline{S\times \Sigma};C_{\#}(\Gamma))$.
\end{itemize}
\end{definition}
%

We remark that \Cref{def:two_scale} is given for the thin layer $\Omega_\e = \Sigma \times (0, \e)$, since there the limit behavior is more straightforward to define. However, the definition can, with the extension operator from \Cref{lem:extension_operator}, be directly transferred to the rough thin layer $\OmegaIdx{f}$. A function $u_\e \in  L^2(S \times \OmegaIdx{f})$ is said to two-scale converge to a function $u \in L^2(S \times \Sigma \times Y^d)$ if
\begin{equation*}
        \lim_{\varepsilon\to 0} \frac{1}{\varepsilon}\int_S \int_{\OmegaIdx{f}} u_\varepsilon(t, x)\varphi\left(t, \Tilde{x}, \nicefrac{x}{\varepsilon}\right) \di{x}\di{t} = \int_S \int_{\Sigma} \int_{Z} u(t,\Tilde{x},y)\varphi(t,\Tilde{x},y) \di{y}\di{\Tilde{x}}\di{t},
\end{equation*}
for all admissible test functions $\varphi$. In this case we write $u_\e \twosc \chi_Z u$.

Given the assumptions on the data and the estimate derived in \cref{sec:analysis_micro_model} we obtain the following limit functions.
\begin{proposition}[Existence of limit functions]\label{pro:existence_limit}
    Under Assumptions \ref{item:A1}-\ref{item:A5}, weak solutions $(\theta_\e, u_\e, p_\e) \in L^2(S;\Theta_\e) \times V_\e \times Q_\e$ of \cref{eq:epsilon-problem} have convergent subsequences.
    More specifically, there are limit functions (without relabeling for the subsequences)
    \begin{itemize}
        \item[i)] $\theta^s \in L^2(S; H^1(\Omega))  \cap H^1(S; H^1(\Omega)')$ such that
        \begin{equation*}
           \theta_\e^s \to \theta^s \text{ in } L^2(S \times \Omega),
            \quad \partial_t \theta_\e^s \rightharpoonup \partial_t \theta^s \text{ in } L^2(S;H^1(\Omega)'),
            \quad 
            \nabla  \theta_\e^s \rightharpoonup \nabla \theta^s \text{ in } L^2(S \times \Omega)^d,
        \end{equation*}
        \item[ii)] $\theta^f \in L^2(S; H^1(\Sigma)) \cap H^1(S; H^1(\Sigma)')$ and $\theta_1^f \in L^2(S \times \Sigma; H^1_{\#}(Y^d))$ such that
        \begin{equation*}
            \thetaIdx{f} \stwosc \chi_{Z} \theta^f
            \quad \text{ and } \quad
            \nabla\thetaIdx{f} \twosc
            \chi_{Z} \left(
            \nabla_{\Tilde{x}} \theta^f + \nabla_y \theta_1^f
            \right).
        \end{equation*}
        Further, it holds
        \[
        \lim_{\e\to0}\frac{1}{\e}\int_S\langle\partial_t\theta_\e^f,\psi+\varphi_\e\rangle_{H^1_\#(\OmegaIdx{f})}\di{t}=\int_S\int_Z\langle\partial_t\theta^f,\psi+\varphi\rangle_{H^1_\#(\Sigma)}\di{y}\di{t}
        \]
        for all $\eta\in L^2(S;H^1_\#(\Sigma))$ and $\varphi_\e(t,x)=\varphi(t,\tilde{x},\nicefrac{x}{\e})$ with $\varphi\in C_0^\infty(S;C_\#^\infty(\Sigma;C_\#(Z))).$
        \item[iii)] $u \in L^2(S\times \Sigma; H^1_{\#}(Y^d))^d$ with 
        \begin{equation*}
            u_\e \twosc \chi_{Z}  u \quad \text{ and }
            \quad 
            \e \nabla u_\e \twosc \chi_{Z}  \nabla_y u.
        \end{equation*}
        It holds $\dive_y{u(t,\tildex, y)} = 0$ and $\dive_{\tildex}{\left(\int_Z u(t,\tildex, y)\di{y}\right)} = 0$ for almost all $(t,\tildex, y) \in S\times\Sigma \times Z$. Moreover, $u = u_\text{motion}$ on $S\times\Sigma\times\{y_d=0\}$ and $u = 0$ for $S\times\Sigma\times\Gamma$ almost everywhere. 
        \item[iv)] $p \in L^2(S \times \Sigma;L^2_{0,\#}(Z))$ with
        \begin{equation*}
            \e^2 \chi_{\OmegaIdx{f}}  p_\e \twosc \chi_{Z}  p.
        \end{equation*} 
    \end{itemize}
    Additionally, on the rough interface, it holds
    \begin{equation}\label{eq:limitInterface}
        \lim_{\e \to 0} \int_S \int_{\Gammafs} \alpha (\thetaIdx{s} - \thetaIdx{f}) \varphi(t,\tilde{x},\nicefrac{x}{\e}) \di{\sigma_x} \di{t} = \int_S \int_{\Sigma} \alpha |\Gamma| (\theta^s - \theta^f) \varphi(t,\tilde{x},y) \di{\Tilde{x}} \di{t}, 
    \end{equation}
    for all $\varphi\in C(\overline{S\times\Sigma};C_\#(\Gamma))$.
\end{proposition}
\begin{proof}
$i).$ $\theta_\e^s$ is uniformly bounded in $L^2(S;H^1(\OmegaIdx{s}))\cap H^1(S;H^1(\OmegaIdx{s})')$ as a result of \cref{lem:temperature_existence_bounds,lemma:time_derivative_temp}.
From here, these convergence results follow via \cite[Theorem 4]{Donato_2019}.

$ii).$ These convergence results are due to \cref{lem:temperature_existence_bounds,lemma:time_derivative_temp} together with the properties of two-scale convergence.
For $\thetaIdx{f} \twosc \chi_{Z} \theta^f$ and  $\nabla\thetaIdx{f} \twosc\chi_{Z} \left(\nabla_{\Tilde{x}} \theta^f + \nabla_y \theta_1^f\right)$,  we refer to \cite[Proposition 4.3]{Gahn17}; the same arguments can also be found in \cite{Bhattacharya2022, Freudenberg2024}.
For the convergence of the time derivative $\partial_t \thetaIdx{f}$ as well as the strong two-scale convergence of $\thetaIdx{f}$, we point to \cite[Proposition 5]{Gahn2024}.

$iii)$-$iv.)$ There are already several very similar results for fluid and pressure limits in thin layers available in the literature, see for example \cite{Fabricius2023, Gahn2022Stokes}.
Here, we rely on the estimates given via \cref{lemma:existence_bounds_stokes}.
The specific convergence results in a stationary setting can be found in \cite[Lemma 4.8]{Fabricius2023} and can be easily transferred to our quasi-stationary setup due to the uniform $L^2$-bounds.
The general proof idea for the divergence properties is the same as for porous media in non-thin domains \cite[Proposition 1.14]{Allaire92}.

Finally, the last integral identity over interfaces (see \cref{eq:limitInterface}) is also a common limit in similar setups, e.g., \cite{Bhattacharya2022,Neuss07,Gahn2024}.
We observe that $\theta_\e^f$ and $\theta_\e^s$ are both uniformly bounded in $L^2(S\times\Gamma_\e)$ via \cref{lem:temperature_existence_bounds} together with \cref{lem:trace_estimate}.
As a consequence, there are converging subsequences for which the limit \cref{eq:limitInterface} holds  (\cite[Theorem 4.4.$(ii)$]{Bhattacharya2022} and \cite[Lemma B.1, Corollary B.2]{Gahn2024}).

\end{proof}
Given the strong two-scale convergence and the properties of $\mu$ we obtain the following convergence of the viscosity function.
\begin{corollary}\label{corrollary:mu}
    Given the Assumption \ref{item:A2} it holds
    $\chi_{\OmegaIdx{f}}\mu(\thetaIdx{f}) \twosc \chi_Z\mu(\theta^f)$
    strongly in $L^2(S \times \OmegaIdx{f})$.
\end{corollary}
\subsection{Homogenization limit}
We now carry out the homogenization process in the weak formulation.
Some of the steps in this limit process are standard in this case; therefore we do not explain everything in detail but only comment on the general behavior and explain the important terms more in-depth.

\begin{theorem}[Homogenized model]\label{thm:homogenized_model}
    Let the Assumptions \ref{item:A1}-\ref{item:A7} be fulfilled and let $(\theta^s, \theta^f, u, p)$ be the limit functions given by \cref{pro:existence_limit}. Set $\Bar{u}(t, \tildex)= \int_Z u(t, \tildex, y) \di{y}$.
    Then $(\theta^s, \theta^f, \Bar{u}, p)$ are weak solutions of the following problem:
    \begin{subequations}\label{eq:effective_model}
    \begin{alignat}{2}
        \partial_t\theta^s - \kappa^s \Delta \theta^s &= 0 \quad \quad &&\text{ in } S \times \Omega, 
        \\
        \theta^s(0, \cdot) &= \theta^s_{0} &&\text{ in } \Omega, 
        \\
        -\kappa^s \nabla \theta^{s} \cdot \nu &= 0 &&\text{ on } S \times \partial \Omega \setminus \Sigma,
        \\ 
        -\kappa^s \nabla \theta^{s} \cdot \nu &= \alphafg |\Gamma| \left(\theta^f -\theta^s\right) + \Bar{f}^s \quad &&\text{ on } S \times \Sigma,
    \end{alignat}
    where $\Bar{f}^{f,s}(t, \tildex) \coloneqq \int_\Gamma f^{f,s}(t, \tildex, y) \di{\sigma_y}$. The fluid temperature $\theta^f$ satisfies
    \begin{alignat}{2}
        |Z| \partial_t\theta^f - \dive_{\tildex}{\left(\Tilde{\kappa} \nabla_{\tildex} \theta^f + \Bar{u} \theta^f\right)} &= \alphafg |\Gamma| \left(\theta^f -\theta^s\right) + \Bar{f}^f
        \quad \quad 
        &&\text{ in } S\times \Sigma, 
        \\
        \theta^f(0, \cdot) &= \theta^f_{0} &&\text{ in } \Sigma, 
        \\
        -\Tilde{\kappa} \nabla_{\tildex} \theta^{f} \cdot n_\Sigma &= 0 &&\text{ on } S \times \SigmaOut,
        \\ 
        \theta^{f} &= 0 \quad &&\text{ on } S \times \SigmaIn.
    \end{alignat}
    The fluid movement in the layer is governed by
    \begin{alignat}{2}
        \mu(\theta^f) \Bar{u} &= -K \nabla_{\tildex} p + \mu(\theta^f) \Bar{\xi}_0 \quad &&\text{ in } S\times \Sigma, \label{eq:darcy_layer}
        \\
        \dive_{\tildex}(\Bar{u}) &= 0 \quad &&\text{ in } S\times \Sigma, \label{eq:divergence_layer}
        \\
        \int_\Sigma p \di{\tilde{x}} &= 0 \quad &&\text{ in } S,
        \\
        \Bar{u}(t, \tildex) \cdot n_\Sigma &= \int_0^1 \uBC(\tildex, y) \di{y} \cdot n_\Sigma \quad &&\text{ on } S\times \partial \Sigma.\label{eq:bc_flow_layer}
    \end{alignat}
    \end{subequations}
    The constant effective conductivity $\tilde{\kappa}\in\R^{(d-1)\times(d-1)}$, effective permeability $K\in \R^{(d-1)\times(d-1)}$ and induced flow $\bar{\xi}_0 \in \R^{d-1}$ are given as
    \[
    \tilde{\kappa}_{ij} = \int_{Z} \kappa^f (\delta_{ij} + \nabla_y \omega_i \cdot e_j) \di{y},\quad K_{ij} = \int_Z \xi_j \cdot e_i \di{y}, \quad
    (\bar{\xi}_0)_i = \int_Z (\xi_0)_i(y) \di{y},
    \]
    where the cell solutions $w_i\in H^1_\#(Z)$, $\xi_{i}\in H^1_\#(Z)^d$, and $\xi_0\in H^1_\#(Z)^d$ ($i=1,...,d-1$) are the unique solutions to the local cell problems given by \cref{eq:cell_problem_temp,eq:cell_problem_stokes,eq:cell_problem_stokes_bc_movement}.
\end{theorem}
\begin{proof}
    The existence of limit functions $(\theta^s, \theta^f, \Bar{u}, p)$ is guaranteed by \cref{pro:existence_limit}.
    The precise structure of the limit model \cref{eq:effective_model} is derived with the concept of two-scale convergence as outlined in \cref{limit_procedure_heat,limit_procedure_stokes} below.
\end{proof}
\begin{remark}
    Note, that by \cref{pro:direction_d_zero} and the equality \cref{eq:darcy_layer} we obtain that $\Bar{u} \cdot e_d = 0$. This is reasonable since we remain with an averaged interface velocity in the effective model and a nonzero flow in direction $d$ would imply a flow into (or out of) the solid domain.
    We recovered, ignoring the temperature-dependent viscosity, the same effective equation as in \cite{Bayda_1989}, where the needed assumptions for convergence are reduced, thanks to the two-scale concept. 

    Additionally, in 2D the effective problem for the fluid simplifies further. By \cref{eq:divergence_layer} it holds $\dive_{\tildex}\Bar{u} = \partial_{x_1} \Bar{u}_1 = 0$ and the velocity is a constant solely given by the boundary function $\uBC$ in \cref{eq:bc_flow_layer}. This linearizes the problem since the temperature equation is independent of the pressure $p$.
\end{remark}

The existence of a solution of model (\ref{eq:effective_model}) is ensured by the existence of the original problem and the two-scale limits.
Similar to the micro model (see \cref{thm:eps_problem_uniqueness}), we need additional regularity assumptions to ensure uniqueness.
These are, however, slightly lower as we have better embeddings to work with due to the dimension reduction in the fluid layer
Please note that, in this case, all convergences in \cref{pro:existence_limit} hold for the full sequences and not just for subsequences.
\begin{theorem}[Uniqueness of the homogenized model]\label{thm:homogenized_model_uniqueness}
If there is a weak solution of the homogenized problem where $\nabla_{\tilde{x}} \theta^f\in L^\infty(S;L^{2+\delta}(\Sigma))$ and $\nabla_{\tilde{x}} p\in L^\infty(S;L^{2+\delta}(\Sigma))$ for some $\delta>0$, then this solution is the only one.
\end{theorem}
\begin{proof}
    Assume that $(\theta^f_i,\theta^s_i,\bar{u}_i,p_i)$, $i=1,2$, are two sets of weak solutions of the homogenization limit given by System \eqref{eq:effective_model} (the variational formulations are also stated in \cref{eq:weak_formulation_temp_limit,eq:identiy_with_cell_solutions,eq:weak_formulation_fluid_limit}).
    In the following, we denote their differences by $(\bar{\theta}^f,\bar{\theta}^s,\bar{\bar{u}},\bar{p})$.
    W.l.o.g. we assume that $(\theta_2^f,p_2)$ is the solution enjoying the higher gradient integrability.

    From \cref{eq:identiy_with_cell_solutions} we obtain for $\bar{\bar{u}}$, that
    \[
    \|\bar{\bar{u}}\|_{L^2(S\times\Sigma)}^2
    \leq C\left|\left(\nabla_{\tilde{x}} \bar{p},\bar{\bar{u}}\right)_{S\times\Sigma}\right|
    +\left|\left(\left[\frac{1}{\mu(\theta^f_1)}-\frac{1}{\mu(\theta^f_2)}\right]\nabla_{\tilde{x}} p_2,\bar{\bar{u}}\right)_{S\times\Sigma}\right|
    \]
    which, using the Lipschitz continuity of $\mu$ as well as the higher regularity of $p$, yields
    \[
    \|\bar{\bar{u}}\|_{L^2(S\times\Sigma)}
    \leq C\left(\|\nabla_{\tilde{x}} \bar{p}\|_{L^2(S\times\Sigma)}
    +\|\nabla_{\tilde{x}}p_2\|_{L^\infty(S;L^{2+\delta}(\Omega))}\|\bar{\theta}^f\|_{L^2(S;L^{2+\frac{4}{\delta}}(\Sigma))}\right).
    \]
    Testing \cref{eq:weak_formulation_fluid_limit} with $\psi = \bar{p}$, we get:
    \begin{equation*}
    c\|\nabla_{\tilde{x}}\bar{p}\|_{L^2(S\times\Sigma)}
    \leq C
    \|\nabla_{\tilde{x}}p_2\|_{L^\infty(S;L^{2+\delta}(\Omega))}\|\bar{\theta}^f\|_{L^2(S;L^{2+\frac{4}{\delta}}(\Sigma))}
    \end{equation*}
    As a result, the uniqueness of velocity and pressure follows from the uniqueness of the fluid temperatures via
    \begin{equation}\label{eq:uniqueness_fluid_limit}
    \|\bar{\bar{u}}\|_{L^2(S\times\Sigma)}+\|\nabla_{\tilde{x}}\bar{p}\|_{L^2(S\times\Sigma)}
    \leq C\|\bar{\theta}^f\|_{L^2(S;L^{2+\frac{4}{\delta}}(\Sigma))}.
    \end{equation}

    We now look at the weak forms for the heat system (\cref{eq:weak_formulation_temp_limit}), take their difference and test with $(\bar{\theta}^s,\bar{\theta}^f)\in L^2(S;H^1(\Omega)\times H^1(\Sigma))$ to get
    \begin{multline*}
    (\partial_t\bar{\theta}^s,\bar{\theta}^s)_{S\times\Omega}
    +|Z|(\partial_t\bar{\theta}^f,\bar{\theta}^f)_{S\times\Sigma}
    +\kappa^s(\nabla\bar{\theta}^s,\nabla\bar{\theta}^s)_{S\times\Omega}
    +(\tilde{\kappa}\nabla_{\tilde{x}}\bar{\theta}^f,\nabla\bar{\theta}^f)_{S\times\Sigma}\\
    +(\bar{u}_1\cdot\nabla_{\tilde{x}}\bar{\theta}^f,\bar{\theta}^f)_{S\times\Sigma}+\alpha|\Gamma|(\bar{\theta}^s-\bar{\theta}^f,\bar{\theta}^s-\bar{\theta}^f)_{S\times\Sigma}
    =-(\bar{\bar{u}}\cdot\nabla_{\tilde{x}}\theta^f_2,\bar{\theta}^f)_{S\times\Sigma}
    \end{multline*}
    for all $(\varphi,\psi)\in L^2(S;H^1(\Omega)\times H^1(\Sigma))$.
    For the advection term on the left-hand side, we exploit that $\bar{u}_1$ is divergence-free (cf. \cref{eq:convection_indentity}):
    \[
    (\bar{u}_1\cdot\nabla_{\tilde{x}}\bar{\theta}^f,\bar{\theta}^f)_{S\times\Sigma}=\frac{1}{2}(u_{BC}\cdot n,|\bar{\theta}^f|^2)_{S\times\Sigma}\geq0.
    \]
    The advective term on the right-hand side can be handled via the additional regularity of $\nabla_{\tilde{x}}\theta^f_2$; in particular we have
    \[
    |(\bar{\bar{u}}\cdot\nabla_{\tilde{x}}\theta^f_2,\bar{\theta}^f)_{S\times\Sigma}|\leq\|\bar{u}\|_{L^2(S\times\Sigma)}\|\bar{\theta}^f\|_{L^2(S;L^{2+\frac{4}{\delta}}(\Sigma))}\|\nabla\theta_2^f\|_{L^2(S;L^{2+\delta}(\Sigma))}.
    \]
    Please note that $\Sigma$ is at most two-dimensional.
    For that reason, the Sobolev embedding $H^1\subset L^{2+\frac{4}{\delta}}$ holds for all $\delta>0$.
    With the same line of argument as in \cref{thm:eps_problem_uniqueness}, uniqueness now follows.
\end{proof}

\subsubsection{Limit procedure for the heat system.}\label{limit_procedure_heat}
The $\e$-solutions $\theta_\e\in L^2(S;\Theta_\e)\cap H^1(S;\Theta_\e')$ satisfy
\begin{multline}\label{eq:weak_form_micro}
        \langle \partial_t \theta_\e, \phi \rangle_{H^1(\OmegaIdx{s})'} 
        + (\kappa\nabla\theta_\e, \nabla\phi)_{ \OmegaIdx{s}} 
        + \frac{1}{\varepsilon}\langle\partial_t \theta_\e, \phi\rangle_{H^1_{\SigmaInEps}(\OmegaIdx{f})'} 
        + \frac{1}{\varepsilon} (\kappa\nabla\theta_\e, \nabla\phi)_{ \OmegaIdx{f}} 
        \\
        +\frac{1}{\e} (u_\e \cdot \nabla\theta_\e, \phi)_{\OmegaIdx{f}} 
        + \alphafg (\jump{\theta_\varepsilon}, \jump{\phi})_{ \Gammafs}
        = (f_\e^s, \phi^s)_{\Gammafs} + (f_\e^f, \phi^f)_{\Gammafs}
\end{multline}
for all test functions $\phi=(\phi^s,\phi^f)\in L^2(S;\Theta_\e)$.
As a reminder, $\Theta_\e=H^1(\OmegaIdx{s})\times H^1(\OmegaIdx{f})$.

We start with the limit inside $\OmegaIdx{s}$ by choosing test functions $\phi=(\phi^s,0)$.
Given the Assumptions \ref{item:A6} and \ref{item:A7} for the convergence of the initial condition and the heat source as well as the convergence results for $\theta_\e^s$ (\cref{{pro:existence_limit}}.$(i)$), passing to the limit in the weak formulation can be carried out directly without relying on two-scale convergence.
We point to \cite[Section 5]{Donato_2019} for a similar setup with a rough interface.
In the limit, we get
\begin{equation}\label{eq:HomogenizationThermoS}
        \langle \partial_t \theta^s, \phi^s \rangle_{H^1(\Omega)'} 
        + (\kappa^s\nabla\theta^s, \nabla\phi^s)_{ \Omega}
        + \alphafg|\Gamma| (\theta^s-\theta^f, \phi^s)_{\Sigma}
        = (\bar{f}^s, \phi^s)_{\Sigma}
\end{equation}
which holds for all test functions $\phi^s\in H^1(\Omega)$ and almost all $t\in S$.


Most aspects of the temperature limit in $\OmegaIdx{f}$ are also quite standard and we point to \cite{Bhattacharya2022, Neuss07, Gahn17} for similar setups.
As usual for such diffusion problems, the limit $\theta^f_1$ can be represented as
\begin{equation}\label{eq:identiy_with_cell_solutions_temp}
    \theta^f_1(t, \tildex, y) = \sum_{i=1}^{d-1} \omega_i(y) \partial_{\tildex_i} \theta^f(t, \tildex)   
\end{equation}
where $\omega_i \in H^1_{\#}(Z)$, $i=1,..,d-1$, are weak solutions (unique up to a constant) of
\begin{equation}\label{eq:cell_problem_temp}
    \begin{split}
        -\Delta w_i &= 0 \quad 
        \text{ in } Z, \\
        (\nabla_y \omega_i + e_i) \cdot n &= 0 \quad \text{ on } \GammaD,\\
        w_i(\cdot+e_i)&=w(\cdot)\quad \text{in $Z$, $i=1,...,d-1$.}
    \end{split}
\end{equation}
Now, take $\varphi_\e = \varphi_0(t, \tildex) + \e \varphi_1\left(t, \tildex, \nicefrac{x}{\e}\right)$ 
with $\varphi_0 \in C^\infty(\overline{S \times \Sigma})$ 
and $\varphi_1 \in C^\infty(\overline{S \times \Sigma}, C_{\#}^\infty(Z))$ such that $\varphi_0(t, \tildex) = 0$ and $\varphi_1(t, \tildex, y) = 0$ for $\tildex \in \SigmaIn$. When testing \cref{eq:weak_form_micro} with $\phi=(0,\varphi_\e)$ and passing to the limit $\e\to0$, the only problematic term is the convection term.
Here, Assumption \ref{item:A5} as well as the strong two-scale convergence of $\thetaIdx{f} \stwosc \theta^f$ are used to pass to the limit.
Via an integration by parts, we have
\begin{align*}
    (u_\e \cdot \nabla \theta_\e, \varphi_\e)_{S \times \OmegaIdx{f}} &=
    -(u_\e \theta_\e, \nabla_{\tildex} \varphi_0)_{S \times \OmegaIdx{f}} 
    - \e (u_\e \theta_\e, \nabla_{\tildex} \varphi_1)_{S \times \OmegaIdx{f}}
    - (u_\e \theta_\e, \nabla_{y} \varphi_1)_{S \times \OmegaIdx{f}}
    \\
    &\hspace{1.5cm} + (u_\e \cdot n \theta_\e, \varphi_0)_{S \times \partial \OmegaIdx{f}}
    + \e (u_\e \cdot n \theta_\e, \varphi_1)_{S \times \partial \OmegaIdx{f}}
\end{align*}
At the boundary $\partial\OmegaIdx{f}$, we have $u_\e\cdot n=0$ outside of $\Sigma_\e^{in}\cup\Sigma_\e^{out}$ and $u_\e=\uEpsBC$.
At $\Sigma_\e^{in}$, we have $\theta_\e^f=0$, therefore
\[
(u_\e \cdot n \theta_\e, \varphi_0)_{S \times \partial \OmegaIdx{f}}
    + \e (u_\e \cdot n \theta_\e, \varphi_1)_{S \times \partial \OmegaIdx{f}}
=(u_\e \cdot n \theta_\e, \varphi_0)_{S \times \Sigma_\e^{out}}
    + \e (u_\e \cdot n \theta_\e, \varphi_1)_{S \times\Sigma_\e^{out}}.
\]
Moreover, as $\theta_\e\in L^\infty(S\times\OmegaIdx{f})\cap L^2(S;H^1(\OmegaIdx{f}))$ uniformly in $\e$, its traces are also uniformly bounded in $L^\infty(S\times\Sigma_\e^{out})$ which is also the case for $\uEpsBC$.
As a consequence, with $|\Sigma_\e^{out}|\to0$ in mind:
\[
\lim_{\e\to0}\frac{1}{\e}\e (u_\e \cdot n \theta_\e, \varphi_1)_{S \times\Sigma_\e^{out}}=\lim_{\e\to0}(\uEpsBC \cdot n \theta_\e, \varphi_1)_{S \times \Sigma_\e^{out}}\to0.
\]
Looking at the remaining boundary term, we have 
\[
(u_\e \cdot n \theta_\e, \varphi_0)_{S \times \Sigma_\e^{out}}=(\chi_{\Sigma_\e^{out}}\uEpsBC \cdot n \theta_\e, \varphi_0)_{S \times \partial \Sigma\times(0,\e)}.
\]
Now, transforming to an $\e$-independent domain:
\[
(u_\e \cdot n \theta_\e, \varphi_0)_{S \times \Sigma_\e^{out}}=\e\left(\chi_{\Sigma^{out}}\uBC \cdot n \theta_\e(\tilde{x},\e x_d), \varphi_0\right)_{S \times \partial \Sigma\times(0,1)}
\]
Now, with the uniform bounds on $\theta_\e^f$, $\frac{1}{\e}(u_\e \cdot n \theta_\e, \varphi_0)_{S \times \Sigma_\e^{out}}$ must converge and since $\theta_\e^f$ is the only $\e$-dependent function on the right-hand side, it holds
\[
\lim_{\e\to0}\frac{1}{\e}(u_\e \cdot n \theta_\e, \varphi_0)_{S \times \Sigma_\e^{out}}=(\uBC \cdot n \theta^f,\varphi_0)_{S\times\Sigma^{out}\times(0, 1)}
\]
Utilizing $\thetaIdx{f} \stwosc \chi_Z\theta^f$, $u_\e\twosc \chi_Zu$, and the above computation, we pass to the limit:
\begin{align*}
    \lim_{\e \to 0} \frac{1}{\e} (u_\e \cdot \nabla \theta_\e, \varphi_\e)_{S \times \OmegaIdx{f}} = 
    &- (u \theta^f, \nabla_{\tildex} \varphi_0)_{S \times \Sigma \times Z} 
    - (u \theta^f, \nabla_{y} \varphi_1)_{S \times \Sigma \times Z} 
    \\
    &+ (\uBC(\tildex, y_d) \cdot n_\Sigma \theta^f, \varphi_0)_{S \times \partial \Sigma \times (0,1)}.
\end{align*}
Since $\theta^f$ is a scalar function independent of $y\in Z$, $\dive_y(u\theta^f)=0$ follows from $\dive_yu=0$.
Also, with $u=0$ almost everywhere on $S\times\Sigma\times\Gamma$ and the periodicity of $u$ with respect to $y\in Z$, it holds 
\[
-(u \theta^f, \nabla_{y} \varphi_1)_{S \times \Sigma \times Z} =(\dive_y(u \theta^f),\varphi_1)_{S \times \Sigma \times Z}=0. 
\]
We introduce $\Bar{u}(t, \tildex)= \int_Z u(t, \tildex, y) \di{y}$ and note that $\dive_{\tilde{x}}\Bar{u}=0$ (\cref{pro:existence_limit}.$(iii)$).
This yields
\begin{align*}
    \lim_{\e \to 0} \frac{1}{\e} (u_\e \nabla \theta_\e, \varphi_\e)_{S \times \OmegaIdx{f}} 
    &=  
    (\Bar{u} \nabla_{\tildex} \theta^f, \varphi_0)_{S \times \Sigma} 
    + \left(\Bar{u} \cdot n_\Sigma - \int_{0}^1 \uBC(\tildex, s) \cdot n_\Sigma \di{s}, \theta^f \varphi_0 \right)_{S \times \partial \Sigma}.
\end{align*}
Consequently, the limit for the heat flow inside $\OmegaIdx{f}$ takes the form
\begin{multline*}
        |Z|\langle\partial_t\theta^f,\varphi_0\rangle_{L^2(S; H^1_{\SigmaIn}(\Sigma)')}
        +(\tilde{\kappa}\nabla_{\tilde{x}}\theta^f,\nabla_{\tilde{x}}\varphi_0)_{S\times\Sigma}
        +(\bar{u}\cdot\nabla_{\tilde{x}}\theta^f,\varphi_0)_{S\times\Sigma}\\
        -\alpha|\Gamma|(\theta^s-\theta^f,\varphi_0)_{S\times\Sigma}
        +\left(\theta^f\left[\Bar{u} \cdot n_\Sigma - \int_{0}^1 \uBC(\tildex, s) \cdot n_\Sigma \di{s}\right], \varphi_0 \right)_{S \times \partial \Sigma}
        = (\bar{f}^f,\varphi_0)_{S\times\Sigma}
\end{multline*}
which, by density, holds for all $\varphi_0\in L^2(S;H^1_{\Sigma_{in}}(\Sigma))$.
In this formulation, however, the additional term over $\partial\Sigma$ is a bit puzzling.
As a next step, we show that it is vanishing.
To that end, we investigate the outer boundary behavior of $\Bar{u}$.
With the two-scale convergence of $u_\e$ and the choice of the  Dirichlet condition $\uEpsBC$, testing with any $\varphi \in C^\infty(\overline{\Sigma})$ yields
\begin{align*}
    0 &= \lim_{\e \to 0} \frac{1}{\e} \int_{\OmegaIdx{f}} \dive{u_\e(x)} \varphi(\tildex) \di{x}
    = \lim_{\e \to 0} \frac{1}{\e}\left[-\int_{\OmegaIdx{f}} u_\e(x)\cdot\begin{pmatrix}\nabla_{\tildex} \varphi(\tildex)\\ 0\end{pmatrix} \di{x}
      + \int_{\partial \OmegaIdx{f}} u_\e \cdot n \varphi(\tildex) \di{\sigma_x}\right]
    \\
    &= -\int_{\Sigma} \int_Z  u(\tildex, y) \nabla_{\tildex} \varphi(\tildex) \di{y} \di{\tildex}
      + \int_{\partial \Sigma} \left(\int_0^1 \uBC(\tildex, s) \di{s}\right)\cdot n_\Sigma \varphi(\tildex) \di{\sigma_{\tildex}}.    
\end{align*}
Integration by parts in the last line, lets us conclude that
\begin{equation}\label{eq:effecitve_boundary_u}
    \Bar{u}(t, \tildex) \cdot n_\Sigma = \int_0^1 \uBC(\tildex, s) \di{s} \cdot n_\Sigma
\end{equation}
on $\partial \Sigma$ and for almost all $t \in S$.
As a consequence,
\begin{multline}\label{eq:HomogenizationThermoF}
        |Z|\langle\partial_t\theta^f,\varphi_0\rangle_{L^2(S; H^1_{\SigmaIn}(\Sigma)')}
        +(\tilde{\kappa}\nabla_{\tilde{x}}\theta^f,\nabla_{\tilde{x}}\varphi_0)_{S\times\Sigma}
        +(\bar{u}\cdot\nabla_{\tilde{x}}\theta^f,\varphi_0)_{S\times\Sigma}\\
        -\alpha|\Gamma|(\theta^s-\theta^f,\varphi_0)_{S\times\Sigma}
        = (\bar{f}^f,\varphi_0)_{S\times\Sigma}
\end{multline}
for all $\varphi_0\in L^2(S;H^1_{\Sigma_{in}}(\Sigma))$.

Finally, we show that $\theta^f$ satisfies a homogeneous Dirichlet condition on  $\Sigma_{in}$.
To that end, take $\varphi \in C^\infty(\overline{S\times \Sigma})$ with $\varphi=0$ on $\Sigma_{out}$ and $\psi \in C_\#^\infty(Z)^d$ with $\dive_y{\psi}=0$ and $\psi = 0$ on $\GammaD$.
With $\phi_\e\in C^\infty(\overline{S\times\Omega})$ given by $\phi_\e(t,x)=\varphi(t, \tildex)\psi(\nicefrac{x}{\e})$, it holds
\begin{align*}
    \frac{1}{\e} \int_S \int_{\OmegaIdx{f}} \thetaIdx{f}(t,x) \dive_{x}\phi_\e(t,x) \di{x} \di{t}
    = - \frac{1}{\e} \int_S \int_{\OmegaIdx{f}} \nabla_x \thetaIdx{f}(t,x)\cdot \varphi(t,\tildex) \psi(\nicefrac{x}{\e})
    \di{x} \di{t}.
\end{align*}
On the left-hand side, we simplify via $\dive_x\phi_\e(t,x)=\nabla_{\tilde{x}}\varphi(t,\tilde{x})\cdot\psi(\nicefrac{x}{\e})$ (due to $\dive_y\psi=0$).
Passing to the limit $\e\to0$ on both sides, we then obtain
\begin{multline}\label{eq:homogenization_thermo}
    \int_S \int_{\Sigma}\int_{Z}  \theta^f(t,\tildex)\nabla_{\tilde{x}}\varphi(t,\tilde{x})\cdot\psi(y)
    \di{y} \di{\tildex} \di{t}\\
    =-\int_S \int_{\Sigma}\int_{Z} \left(\nabla_{\tildex}\theta^f(t, \tildex) + \nabla_y \theta^f_1(t, \tildex, y)\right) \cdot  \varphi(t,\tildex)\psi(y)
    \di{y} \di{\tildex} \di{t}
\end{multline}
Via an integration by parts, we can write the right-hand side as
\begin{multline}\label{eq:homogenization_thermo2}
    -\int_S \int_{\Sigma}\int_{Z} \left(\nabla_{\tildex}\theta^f(t, \tildex) + \nabla_y \theta^f_1(t, \tildex, y)\right) \cdot  \varphi(t,\tildex)\psi(y)
    \di{y} \di{\tildex} \di{t} \\
    = \int_S \int_{\Sigma}\int_{Z} \theta^f(t, \tildex)  \nabla_{\tilde{x}}\varphi(t,\tilde{x})\cdot\psi(y)
    \di{y} \di{\tildex} \di{t} - \int_S \int_{\SigmaIn}\theta^f(t, \tildex) \varphi(t,\tildex) \int_{Z}\psi(y)
    \di{y} \cdot n_\Sigma \di{\sigma_{\tildex}} \di{t}
    \\
    + \int_S \int_{\Sigma}\int_{Z} \theta^f_1(t, \tildex, y)\varphi(t,\tildex)  \dive_{y}{\psi(y)}
    \di{y} \di{\tildex} \di{t} - \int_S \int_{\Sigma} \int_{\partial Z} \theta^f_1(t, \tildex, y) \varphi(t,\tildex) \psi(y)
    \cdot n_Z \di{y} \di{\sigma_{\tildex}} \di{t}.
\end{multline}
Given the properties of $\psi$ and the periodicity of $\theta^f_1$, the last two integrals vanish.
In addition, seeing that the first term of the right-hand side of \cref{eq:homogenization_thermo2} cancels with the left-hand side of \cref{eq:homogenization_thermo}, we obtain 
\begin{equation*}
    \int_S \int_{\SigmaIn}\theta^f(t, \tildex) \varphi(t, \tildex) \int_{Z} \psi(y)
    \di{y} \cdot n_\Sigma \di{\sigma_{\tildex}} \di{t} = 0
\end{equation*}
for all $\psi \in C_\#^\infty(Z)^d$ with $\dive_y{\psi}=0$ in $Z$ and $\psi = 0$ on $\GammaD$.
Choosing any such $\psi$ with $\int_{Z} \psi(y)
    \di{y} \cdot n_\Sigma\neq0$, we conclude
\begin{equation*}
    \int_S \int_{\SigmaIn}\theta^f(t, \tildex) \varphi(t, \tildex) \di{\sigma_{\tildex}} \di{t} = 0.
\end{equation*}
By density, this holds for all $\varphi \in L^2(S; H^1_{\SigmaOut}(\Sigma))$, which implies $\theta^f = 0$ almost everywhere on $S \times \SigmaIn$ (note that $\varphi$ is not restricted on $\SigmaIn$). 
Combining \cref{eq:homogenization_thermo,eq:homogenization_thermo2}, we arrive at the following weak form as the limit of \cref{eq:weak_form_micro}:
\begin{equation}\label{eq:weak_formulation_temp_limit}
    \begin{split}
        &\langle\partial_t\theta^s,\varphi\rangle_{L^2(S; H^1(\Omega)')}
        +|Z|\langle\partial_t\theta^f,\psi\rangle_{L^2(S; H^1_{\SigmaIn}(\Sigma)')}
        +\kappa^s(\nabla\theta^s,\nabla\varphi)_{S\times\Omega}
        +(\tilde{\kappa}\nabla_{\tilde{x}}\theta^f,\nabla_{\tilde{x}}\psi)_{S\times\Sigma}
        \\
        &+(\bar{u}\cdot\nabla_{\tilde{x}}\theta^f,\psi)_{S\times\Sigma}
        +\alpha|\Gamma|(\theta^s-\theta^f,\varphi-\psi)_{S\times\Sigma}
        =(\bar{f}^s,\varphi)_{S\times\Sigma} 
        + (\bar{f}^f,\psi)_{S\times\Sigma}
    \end{split}
\end{equation}
which holds for all $(\varphi,\psi)\in L^2(S;H^1(\Omega)\times H^1_{\SigmaIn}(\Sigma))$.
Here, the entries of $\tilde{\kappa}$ and $\bar{f}^r$ ($r=f,s$) are given by
\begin{equation*}
    \tilde{\kappa}_{ij}= \int_{Z} \kappa^f (\delta_{ij} + \nabla_y \omega_i \cdot e_j) \di{y},\qquad
    \bar{f}^{r}(t,\tilde{x})=\int_\Gamma f^{r}(t,\tilde{x},y)\di{\sigma_y}.
\end{equation*}

\subsubsection{Limit procedure for the Stokes system.}\label{limit_procedure_stokes}
Shifting to the Stokes system, we observe that our setup is somewhat close to fluid flows in porous media.
The following steps are therefore similar to the classical derivation of the Darcy equation, see \cite{Allaire1989, Alouges2016}.
There is also extensive literature considering thin porous layers, which is even closer to the present setup, for example, \cite{Fabricius2023, Jager1998, Gahn2022Stokes}.
Nevertheless, to better understand the final homogenized model, we briefly present the procedure and point out any differences compared to deriving a Darcy equation.

We start by recalling that the $\e$-solutions $(u_\e,p_\e)\in V_\e\times Q_\e$ satisfy
\begin{equation}\label{eq:WeakStokes}
    (\mu(\theta_\e^f) \nabla u_\e , \nabla \varphi)_{\OmegaIdx{f}} 
        - (p_\e, \dive{\varphi})_{ \OmegaIdx{f}} = 0
\end{equation}
for all $\varphi\in H_0^1(\OmegaIdx{f})$ and almost all $t\in S$.
Let $\varphi \in C_c^\infty(\Sigma;C_\#^\infty(Z))^d$ with $\varphi = 0$ on $\Sigma\times\Gamma_d$.
We test the weak form \cref{eq:WeakStokes} with $\varphi_\e\in H_0^1(\OmegaIdx{f})^d$ given by $\varphi_\e(x)=\e^k \varphi(\tildex, \nicefrac{x}{\e})$ and get (for almost all $(t\in S)$)
\begin{equation*}
        \left(\mu(\theta_\e^f) \nabla u_\e, \left[\e^k \nabla_{\tildex} \varphi
        +\e^{k-1}\nabla_y \varphi\right](\tildex, \nicefrac{x}{\e})\right)_{\OmegaIdx{f}} 
        - \left(p_\e, \left[\e^k\dive_{\tildex}{\varphi}
        +\e^{k-1}\dive_y{\varphi}\right](\tildex, \nicefrac{x}{\e})\right)_{\OmegaIdx{f}} 
        = 0.    
\end{equation*}
We set $k=2$ and pass to the limit $\e \to 0$ with $\e^2p_\e\twosc\chi_Zp$ (see \cref{pro:existence_limit}.$(iv)$ and note that an additional $\e^{-1}$ is consumed via the two-scale limit, see \cref{def:two_scale}) to get
\begin{equation*}
    (p(t,\tildex,y), \dive_y{\varphi}(\tildex,y))_{\Sigma\times Z} = 0
\end{equation*}
which, as this holds for all $\varphi\in C_c^\infty(\Sigma;C_\#^\infty(Z))^d$ with $\varphi=0$ on $\Sigma\times\GammaD$, implies that $p$ is constant in $y$.

Next, we set $k=1$ and restrict testing to functions for which, additionally, $\dive_y{\varphi} = 0$ holds.
The corresponding limit yields 
\begin{equation}\label{eq:helper_p_equal_q}
    (\mu(\theta^f) \nabla_yu , \nabla_y \varphi)_{\Sigma\times Z} - (p , \dive_{\tildex} \varphi)_{\Sigma\times Z}  = 0
\end{equation}
where we used $\mu(\theta_\e^f)\twosc\chi_Z\mu(\theta_\e^f)$ (\cref{corrollary:mu}), $\e\nabla u_\e\twosc \chi_Z\nabla_yu$ (\cref{pro:existence_limit}.$(iii)$), and $\e^2p_\e\twosc\chi_Zp$ (\cref{pro:existence_limit}.$(iv)$).
\Cref{eq:helper_p_equal_q} must, by density, also hold for all $\varphi \in L^2(\Sigma; H^1_{\#,\GammaD}(Z))$ with $\dive_{\tildex}{\left(\int_Z \varphi\di{y} \right)} \in L^2(\Sigma)$.

Next, let us introduce
\begin{equation*}
    V = \left\{\varphi \in L^2(\Sigma; H^1_{\#}(Z))^d : 
    \, 
    \begin{array}{c}
         \dive_y{\varphi} = 0, \, \dive_{\tildex}{\left(\int_Z \varphi\di{y} \right)} =0, \,
         \\
         \varphi = 0 \text{ on } \GammaD\text{ and } \int_Z \varphi \di{y} \cdot n_\Sigma = 0 \text{ on } \partial \Sigma 
    \end{array}
    \right\}.
\end{equation*}
Naturally, \cref{eq:helper_p_equal_q} also holds for every $\varphi\in V$.
With the results from \cite[Lemma 5.3 and Remark 5.5.$(i)$]{Fabricius2023}, the orthogonal of $V$ is characterized by
\begin{equation*}
    V^\perp = \left\{q = \nabla_{\tildex}q_0 + \nabla_y q_1 :\,  q_0 \in H^1(\Sigma)/\R, \, q_1 \in L^2_0(\Sigma \times Z) \right \} \subset L^2(\Sigma; H^1_{\#,\GammaD}(Z)')^d.
\end{equation*}
This decomposition into $(q_0, q_1)$ is also unique for every $q\in V^\perp$.
Therefore, there exists a $q = (q_0, q_1) \in L^2(S; V^\perp)$ such that $\mu(\theta^f) \Delta_y u = \nabla_{\tildex}q_0 + \nabla_y q_1$ (for almost all $t\in S$) understood as an element in $L^2(\Sigma; H^1_{\#,\GammaD}(Z)')^d$.
With \cref{eq:helper_p_equal_q} it follows that $q_0=p$ almost everywhere up to a constant.

Collecting the boundary conditions of $u$ derived in \cref{eq:effecitve_boundary_u} and \cref{pro:existence_limit} as well as the above characterization, we see that
\[
(\bar{u},u,p,q_1)\in L^2(S;H^1(\Sigma)\times  L^2(\Sigma;H^1(Z))\times H^1(\Sigma)/\R \times L^2_0(\Sigma \times Z))
\]
are weak solutions to the following two-pressure Stokes model:
\begin{subequations}
\begin{alignat}{2}
    -\mu(\theta^f) \Delta_y u + \nabla_y q_1  &= - \nabla_{\tildex}p \quad &&\text{ in } S\times\Sigma \times Z,\label{eq:2pressureStokes:1} 
    \\
    \dive_y{u} &= 0 &&\text{ in } S\times\Sigma \times Z,\\
        u &= 0  &&\text{ on } S\times\Sigma \times \Gamma,
    \\
    u &= u_\text{motion} &&\text{ on } S\times\Sigma \times \{y_d = 0\},
    \\
    \bar{u}&=\int_Zu\di{y}&&\text{ on } S\times\Sigma,\\
    \dive_{\tildex}\bar{u} &= 0 &&\text{ in } S\times\Sigma ,
    \\
    \Bar{u}(t, \tildex) \cdot n_\Sigma &= \int_0^1 \uBC(\tildex, s) \di{s} \cdot n_\Sigma \quad &&\text{ on } S\times \partial \Sigma.
\end{alignat}
\end{subequations}
Note that both $u_\text{motion} \cdot e_d = 0$ and Assumption \ref{item:A5} (which guarantees $\int_{\partial \Sigma} n_\Sigma \cdot \int_0^1 \uBC = 0$) are essential to ensure that the boundary conditions are compatible with the incompressibility requirements.

We can further simplify the two-pressure Stokes-problem by constructing suitable cell problems.
The first set of cell problems are quite standard, see \cite{Allaire1989} where they are used to characterize the permeability in a porous medium: for $i=1,\dots, d-1$, find $(\xi_i, \eta_i) \in H^1_{\#}(Z)^d \times L^2_{0, \#}(Z)$ such that 
\begin{alignat}{2}\label{eq:cell_problem_stokes}
    -\Delta_y \xi_i - \nabla_y \eta_i &= e_i \quad &&\text{in } Z, 
    \notag 
    \\
    \dive_y{(\xi_i)} &= 0 &&\text{in } Z,
    \\
    \xi_i &= 0 &&\text{on } \GammaD. 
    \notag
\end{alignat}
Additionally, a cell problem originating from the nonhomogeneous Dirichlet condition at the bottom boundary appears: find $(\xi_0, \eta_0) \in H^1_{\#}(Z)^d \times L^2_{0, \#}(Z)$ such that 
\begin{alignat}{2}\label{eq:cell_problem_stokes_bc_movement}
    -\Delta_y \xi_0 - \nabla_y \eta_0 &= 0 \quad &&\text{in } Z, 
    \notag 
    \\
    \dive_y{(\xi_0)} &= 0 &&\text{in } Z,
    \\
    \xi_0 &= 0 &&\text{on } \Gamma, \notag
    \\ 
    \xi_0 &= u_\text{motion} \quad &&\text{on } \{y_d = 0\}. \notag
\end{alignat}
Note, the cell problem (\ref{eq:cell_problem_stokes_bc_movement}) is a well-defined equation and admits a unique solution, given that $u_\text{motion} \cdot e_d = 0$.
We have the following property of the cell solutions:
\begin{proposition}\label{pro:direction_d_zero}
    For the cell solutions $\xi_i\in H^1_{\#}(Z)^d$ ($i=0,\dots,d-1$) of \cref{eq:cell_problem_stokes} and \cref{eq:cell_problem_stokes_bc_movement}, respectively, it holds 
    \begin{equation*}
        \int_Z \xi_i \cdot e_d \di{y}= 0.
    \end{equation*}
\end{proposition}
\begin{proof}
    We present the arguments for $\xi_0$; the proof is the same for the other functions.
    A direct computation, including an integration by parts, yields
    \begin{equation*}
        0 = \int_Z \dive_y{(\xi_0)} y_d \di{y} = \int_Z \xi_0 \cdot e_d \di{y} + \int_{\partial Z} \xi_0 \cdot n  \, y_d \di{\sigma_y}.
    \end{equation*}
    By the choice of boundary conditions and periodicity, the above integral over $\partial Z$ vanishes and we obtain the desired result.
\end{proof}
Given the solutions of the cell problems (\ref{eq:cell_problem_stokes}) and (\ref{eq:cell_problem_stokes_bc_movement}), we obtain for a.a.~$(t, \tildex, y) \in S \times \Sigma \times Z$ the representations 
\begin{align}\label{eq:identiy_with_cell_solutions}
    u(t, \tildex, y) = \xi_0(y) - \sum_{i=1}^{d-1} \frac{\xi_i(y)} {\mu(\theta^f(t, \tildex))} \partial_{x_i} p(t, \tildex)
    , \quad  q_1(t, \tildex, y) = \eta_0(y) - \sum_{i=1}^{d-1} \frac{\eta_i(y)} {\mu(\theta^f(t, \tildex))} \partial_{x_i} p(t, \tildex).
\end{align}
This can be checked easily by plugging them into the two-pressure Stokes system.
Looking at $\bar{u}=\int_Z u \di{y}$, we can also check that with $\dive_{\tildex}\bar{u}=0$ and by testing \cref{eq:2pressureStokes:1} with $\psi \in H^1(\Sigma)$ while using \cref{eq:identiy_with_cell_solutions}:
\begin{equation}\label{eq:weak_formulation_fluid_limit}
    \begin{split}
        0 &= \left(\dive_{\tildex}{\bar{u}}, \psi \right)_{\Sigma} 
        = -\left(\bar{u}, \nabla_{\tildex} \psi \right)_{ \Sigma} + \left(\bar{u} \cdot n_\Sigma,\psi \right)_{\partial \Sigma}
        \\
        &= -\left(\int_Z \xi_0 \di{y}, \nabla_{\tildex} \psi \right)_{\Sigma}
        + \sum_{i=1}^{d-1}\left(\frac{\partial_i p}{\mu(\theta^f)}\int_Z \xi_i \di{y}, \nabla_{\tildex} \psi \right)_{\Sigma}
        + \left(\left( \int_0^1 \uBC(\cdot, s) \di{s}\right) \cdot n_\Sigma,\psi \right)_{\partial \Sigma}
        \\
        &= -(\bar{\xi}_0, \nabla_{\tildex} \psi )_{ \Sigma} + \left(\frac{1}{\mu(\theta^f)} K \nabla_{\tildex} p, \nabla_{\tildex} \psi \right)_{\Sigma} + (\bar{u}_\text{BC} \cdot n_\Sigma, \psi)_{\partial \Sigma}.
    \end{split}
\end{equation}
Where for the last step we defined
\begin{equation}\label{eq:effective_}
    \bar{\xi}_0 \coloneqq \int_Z \xi_0(y) \di{y},
    \quad
    K_{ij} \coloneqq \int_Z \xi_j(y) \cdot e_i \di{y} 
    \quad \text{and} \quad
    \bar{u}_\text{BC} \coloneqq\int_0^1 \uBC(\cdot, s) \di{s}.
\end{equation}
Note, \cref{eq:weak_formulation_fluid_limit} is similar to a Darcy equation, but with an additional term induced by the Dirichlet condition at $y_d = 0$.
This is precisely the weak form of the Darcy-limit given in \cref{thm:homogenized_model}.

\section{Simulations}\label{sec:simulations}
In addition to the previous analysis, we investigate the properties of the effective model with the help of numerical simulations and examine what aspects can be captured in the homogenized problem.
In particular, we are looking at the effects of the roughness patterns, viscosity function, and inflow conditions.
We also study the accuracy of the homogenization limit by direct comparison with the $\e$-model for small values of $\e$.
Note that we focus on the two-dimensional case for this comparison as the computational cost for the $\e$-problem is prohibitively high in the case $d=3$.

All simulations are carried out with the finite element solver FEniCS \cite{Fenics} and meshes are created with Gmsh \cite{gmsh}.
To stabilize the convection term in our model, a SUPG scheme \cite{BrooksS82} is applied.
We use continuous, piecewise linear finite elements to represent temperature fields $(\theta^f, \theta^s)$ and pressure $p$, while the velocity $u$ is modeled with continuous piecewise quadratic elements.
The time stepping is performed with a semi-implicit Euler scheme, where we linearize the convection term by using the velocity field from the previous time step.
Our implementation is freely available online \cite{Code}.

We start by collecting all the parameters that are fixed in the following simulations:
We consider the time interval $S = (0, 5)$, and the macroscopic domain is the unit square $\Omega = (0,1)^2$.
The conductivities are set to $\kappa^s = 0.5$ and $\kappa^f=0.1$, and the heat exchange parameter to $\alpha = 1.0$. 
The initial conditions are set to zero.
At $\Sigma$, the Dirichlet condition of the velocity is set to $u_\text{motion}=(1, 0)$.
The heat sources are given by $f^f_\e=0$ and 
\begin{equation*}
    f_\e^s(x) = \frac{1 - \nicefrac{x_d}{\e}}{1 - \gamma_0},
\end{equation*}
so that there is a higher production at the tip of the roughness, which is motivated by the underlying application of grinding and the fluid only heats up through exchange with the Robin condition. The temperature-dependent viscosity is modeled by a Vogel-Fulcher-Tammann equation:
\begin{equation*}
    \mu(\theta) = \mu_0 \exp{\left(\frac{\alpha}{\theta - T_0}\right)},
\end{equation*}
which is a reasonable approach for fluids \cite{Garca-Coln1989}.
To obtain a viscosity profile that varies in our temperature range, we choose $\mu_0=0.2$, $\alpha=3.0$, and $T_0=0.6$.
Unless stated otherwise, the time step $\Delta t = 0.05$ and the mesh resolution $h= 0.03$ are used, where the mesh is locally refined around the thin layer with resolution $h\e$.
A convergence study was carried out with respect to $\Delta t$ and $h$ and the expected linear or quadratic convergence was observed.
To study the influence of different roughness patterns, we model the solid-fluid interface by the graph of a periodic function $\gamma: Y^{d-1} \to [\gamma_0, 1]$, meaning $\Gamma = \{(\tildex, \gamma(\tildex)): \tildex \in Y^{d-1}\}$. 
We investigate two different setups:
\begin{itemize}
    \item A sine function 
        \begin{equation*}
            \gamma_\text{sine}(x_1; \gamma_0) = 
                \begin{cases}
                    1 \quad &\text{, if } x < 0.1 \text{ or } x > 0.9, \\
                    1 - \frac{1-\gamma_0}{2}\left( \sin{\left(2\pi \frac{(x_1-0.1)}{0.8}- \frac{\pi}{2}\right)}+1\right) &\text{, if } x \in [0.1, 0.9].
                \end{cases}
        \end{equation*}
    \item A pattern given by smoothed rectangles
        \begin{equation*}
            \gamma_\text{rect}(x_1; \gamma_0) = 
                \begin{cases}
                    1 \quad &\text{, if } x < 0.1 \text{ or } x > 0.9, \\
                    1 - (1 - \gamma_0)\left(-2\left(\frac{x_1-0.1}{0.1}\right)^3 + 3\left(\frac{x_1-0.1}{0.1}\right)^2\right) \quad &\text{, if } x \in [0.1, 0.2), \\
                    \gamma_0 \quad &\text{, if } x \in [0.2, 0.8], \\
                    \gamma_0 + (1 - \gamma_0)\left(-2\left(\frac{x_1-0.8}{0.1}\right)^3 + 3\left(\frac{x_1-0.8}{0.1}\right)^2\right) \quad &\text{, if } x \in (0.8, 0.9].
                \end{cases}
        \end{equation*}        
\end{itemize}
The reference cell $Z$ as well as the effective parameters for both roughness patterns, and three different $\gamma_0$, are stated in \cref{tab:effective_parameters}.
\begin{table}[ht]
    \centering
    \caption{Roughness patterns used in the numerical simulations, and the corresponding effective parameters. Since the setup is two-dimensional, all effective parameters reduce to scalar values.}
    \begin{tabular}{|c|c|c|c|}
    	 \hline
          \diagbox[width=\dimexpr \textwidth/8+2\tabcolsep\relax, height=1cm]{$Z$}{$\gamma_0$} & 0.1 & 0.5 & 0.9 \\
         \hline
         &&& 
         \\[\dimexpr-\normalbaselineskip+2pt]
         \noindent\parbox[c]{20mm}{\includegraphics[width=20mm, height=20mm]{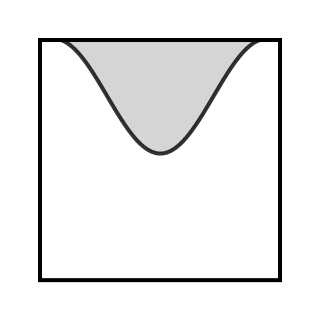}}
         & $\begin{array}{c}
              \Tilde{\kappa} = 0.238 \, \kappa^f \\
               K = 0.00048 \\
               \Bar{\xi}_0 = 0.0692 
         \end{array}$
         & $\begin{array}{c}
              \Tilde{\kappa} = 0.630 \, \kappa^f \\
               K = 0.01704 \\
               \Bar{\xi}_0 = 0.2871
         \end{array}$ 
         & $\begin{array}{c}
              \Tilde{\kappa} = 0.951 \, \kappa^f \\
               K = 0.07025 \\
               \Bar{\xi}_0 = 0.4716
         \end{array}$
         \\
        \hline
         &&& 
         \\[\dimexpr-\normalbaselineskip+2pt]
         \noindent\parbox[c]{20mm}{\includegraphics[width=20mm, height=20mm]{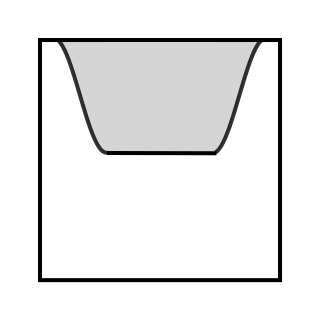}}
         & $\begin{array}{c} 
              \Tilde{\kappa} = 0.120 \, \kappa^f \\
               K = 0.00011 \\
               \Bar{\xi}_0 = 0.0526 
         \end{array}$
         & $\begin{array}{c} 
              \Tilde{\kappa} = 0.523 \, \kappa^f \\
               K = 0.01124 \\
               \Bar{\xi}_0 = 0.2556
         \end{array}$ 
         & $\begin{array}{c} 
              \Tilde{\kappa} = 0.917 \, \kappa^f \\
               K = 0.06292 \\
               \Bar{\xi}_0 = 0.4552
         \end{array}$
    \\\hline
    \end{tabular} 
    \label{tab:effective_parameters}
\end{table}

Three different inflow patterns are compared: a linear inflow, a quadratic inflow, and a linear inflow that is not zero above $\gamma_0$ given by
\begin{align*}
    &(i)\ \ u_\text{BC, lin} = \begin{cases}
    \left(1 - \frac{x_2}{\gamma_0}\right) u_\text{motion} &\text{, if } x_2 \leq \gamma_0, \\
    0 &\text{, if } x_2 > \gamma_0,
    \end{cases}
    \\
    &(ii)\ \  u_\text{BC, quad} = \begin{cases}
    \left(1 - \left(\frac{x_2}{\gamma_0}\right)^2 \right)u_\text{motion} &\text{, if } x_2 \leq  \gamma_0, \\
    0 &\text{, if } x_2 > \gamma_0,
    \end{cases}    
    \\
    &(iii)\ \ u_\text{BC, lin, 2} = (1 - x_2) u_\text{motion}.
\end{align*}
The effective inflow then equals $\Bar{u}_\text{BC, lin} = \frac{1}{2}\gamma_0$, $\Bar{u}_\text{BC, quad} = \frac{2}{3}\gamma_0$ or $\Bar{u}_\text{BC, lin, 2} = \frac{1}{2}$.
With $u_\text{BC, lin, 2}$ we construct an inflow that does not fulfill \ref{item:A5} (the inflow is nonzero also for $x_2 > \gamma_0$). Hence, we cannot use a simple scaling argument 
to build an extension of the boundary condition to the whole domain.
However, we observe that the effective model also works well for this larger class of functions.
If not specified otherwise, we use $u_\text{BC, lin}$ in the following simulation studies.
\begin{remark}
    Please note that the overall simulation scenarios are motivated by the grinding application.
    However, these simulation experiments are only qualitative, as we are not working with realistic and sensible parameter values for the conductivities, the heat source, etc., and instead only highlight the behavior of the effective model.
    We refer to \cite{Wiesener2023} where quantitative simulations using realistic material parameters were performed for a single-grain scenario. 
\end{remark}

\textbf{Comparison between $\e$-model and effective model}. We start by comparing the solutions to the effective model with solutions of the resolved $\e$-model (where we compare with $\e=0.2, 0.1, 0.05, 0.01$) for the sinusoidal roughness pattern and $\gamma_0 = 0.5$.
We present a general overview of the solution in \cref{fig:temperature_eps} and \cref{fig:fluid_pressure_eps}. Note that we multiply the microscale pressure $p_\e$ with $\e^2$, since convergence only holds for $\e^2 p_\e$.
In the following, we focus mainly on comparing the solutions inside the thin layer, as this is the critical area where the interesting features (rough boundary and thin domain) are present.

We observe that the effective model captures the temperature and pressure of the resolved $\e$-model fairly well (\cref{fig:temperature_eps} and \ref{fig:fluid_pressure_eps}). To better observe the microscopic fluctuations of temperature and pressure as well as the effective profile of the effective solution, the micro solution is visualized along the line $(0, 1) \times \{\nicefrac{\e}{2}\}$ in \Cref{fig:eps_temp_pressure} and compared with the effective solution. 
We obtain negative pressure at the inflow ($x_1=0$) since the amount of fluid pushed into the domain by $u_{BC, lin}$ is less than the fluid movement induced by the bottom boundary.
This becomes also apparent, by comparing $\bar{u}_\text{BC, lin}$ with the flow induced by the cell problem (\ref{eq:cell_problem_stokes_bc_movement}), which gives the value $\Bar{\xi}_0 = 0.29$ (see \cref{tab:effective_parameters}).
\begin{figure}[ht]
    \centering
    \begin{subfigure}[b]{0.22\textwidth}
         \centering
         \includegraphics[width=\textwidth]{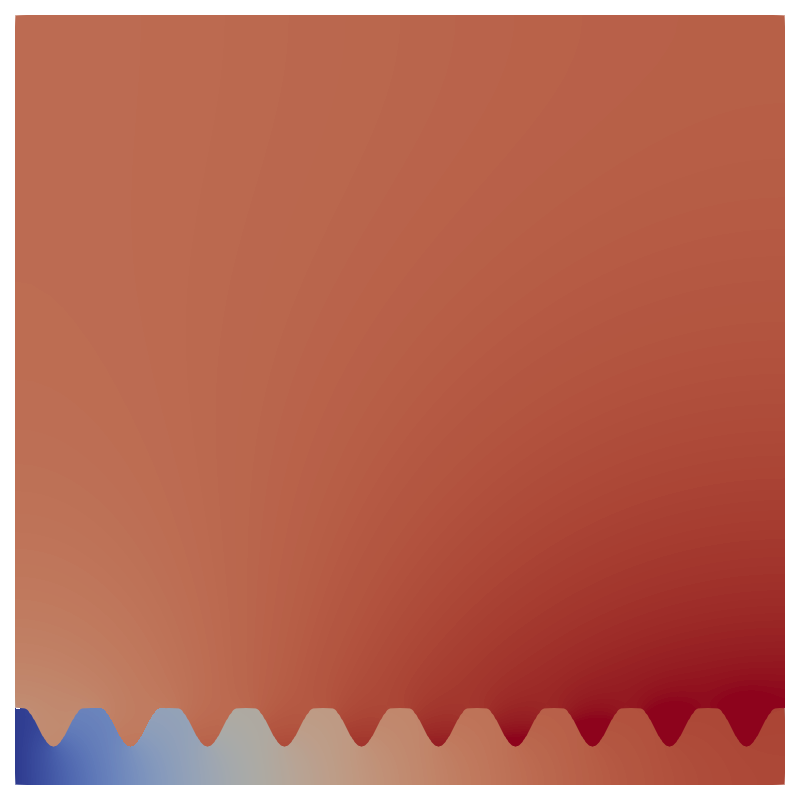}
         \caption{$\e = 0.1$}
    \end{subfigure}
    \hfill
    \begin{subfigure}[b]{0.22\textwidth}
         \centering
         \includegraphics[width=\textwidth]{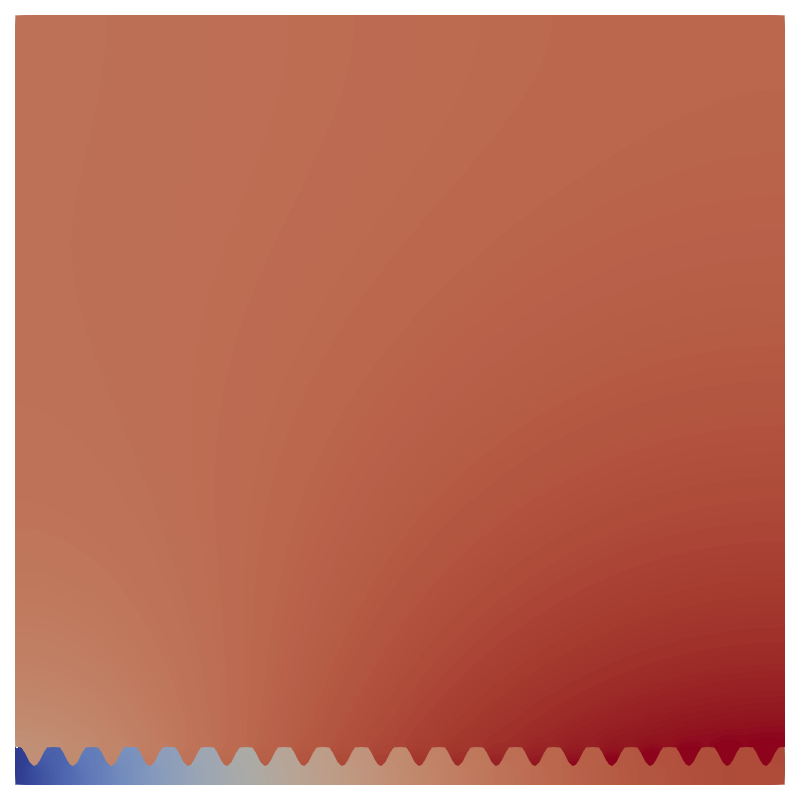}
         \caption{$\e = 0.05$}
    \end{subfigure}
    \hfill
    \begin{subfigure}[b]{0.22\textwidth}
         \centering
         \includegraphics[width=\textwidth]{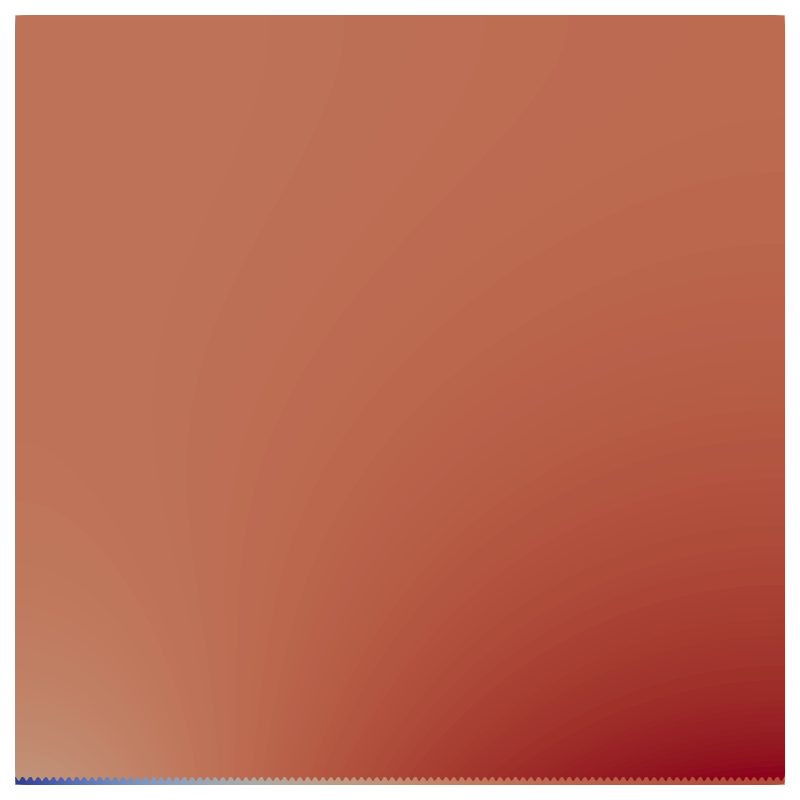}
         \caption{$\e = 0.01$}
    \end{subfigure}
    \hfill
    \begin{subfigure}[b]{0.22\textwidth}
         \centering
         \includegraphics[width=\textwidth]{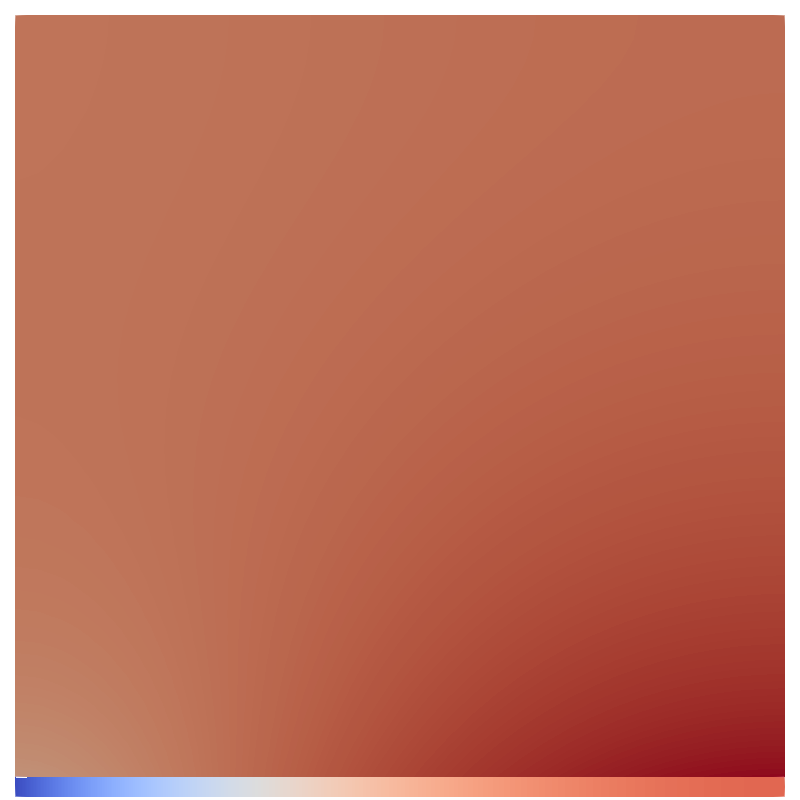}
         \caption{Effective}
    \end{subfigure}
    \hfill
    \begin{subfigure}[b]{0.06\textwidth}
         \centering
         \includegraphics[width=\textwidth]{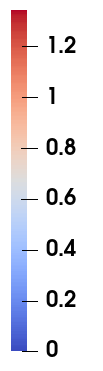}
         \vspace{2pt}
    \end{subfigure}
    \caption{Comparison of the temperature profiles, at $t=5$, for the micro and effective model, with sinusoidal roughness. The effective fluid temperature $\theta^f$ is only defined on the line $(0, 1) \times \{0\}$, but for visualization purposes, the line is being shown with a thickness.}
    \label{fig:temperature_eps}
\end{figure}
\begin{figure}[ht]
    \centering
    \begin{subfigure}[b]{0.22\textwidth}
         \centering
         \includegraphics[width=\textwidth]{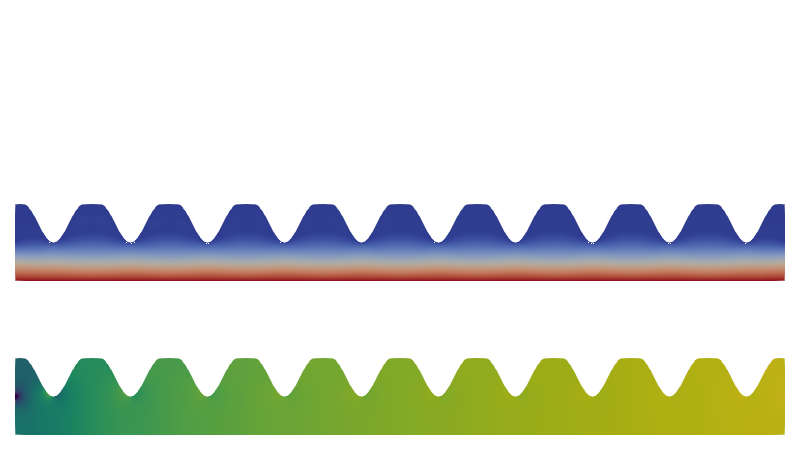}
         \caption{$\e = 0.1$}
    \end{subfigure}
    \hfill
    \begin{subfigure}[b]{0.22\textwidth}
         \centering
         \includegraphics[width=\textwidth]{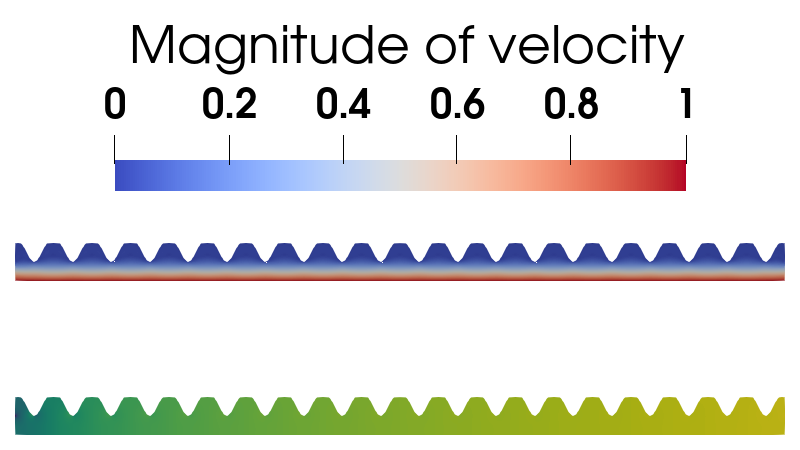}
         \caption{$\e = 0.05$}
    \end{subfigure}
    \hfill
    \begin{subfigure}[b]{0.22\textwidth}
         \centering
         \includegraphics[width=\textwidth]{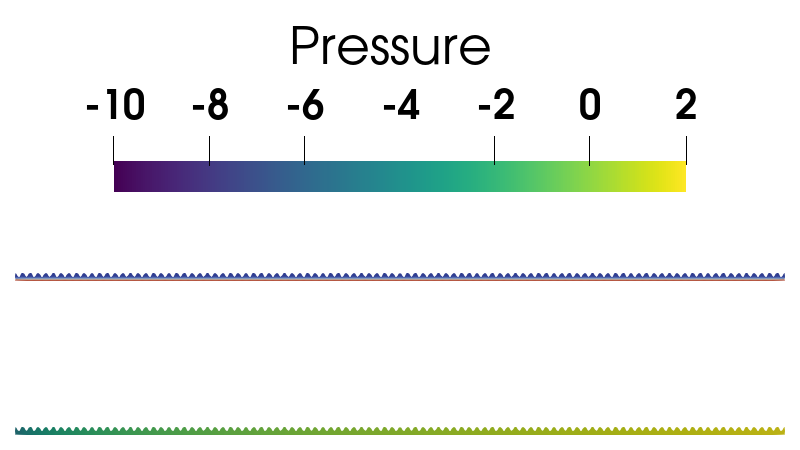}
         \caption{$\e = 0.01$}
    \end{subfigure}
    \hfill
    \begin{subfigure}[b]{0.22\textwidth}
         \centering
         \includegraphics[width=\textwidth]{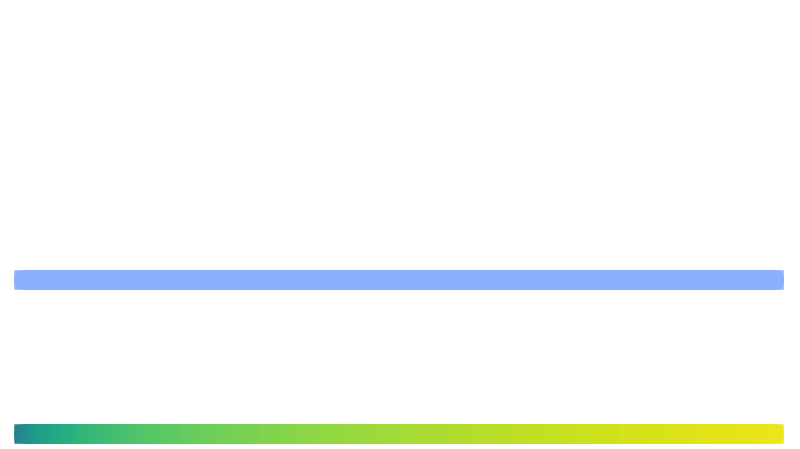}
         \caption{Effective}
    \end{subfigure}
    \hfill
    \caption{Comparison of the fluid (top) and pressure (bottom) profiles, at $t=5$, for the micro and effective model, with sinusoidal roughness. The line in the effective case is again visualized with a thickness.}
    \label{fig:fluid_pressure_eps}
\end{figure}
\begin{figure}[ht]
    \centering
    \includegraphics[width=\linewidth]{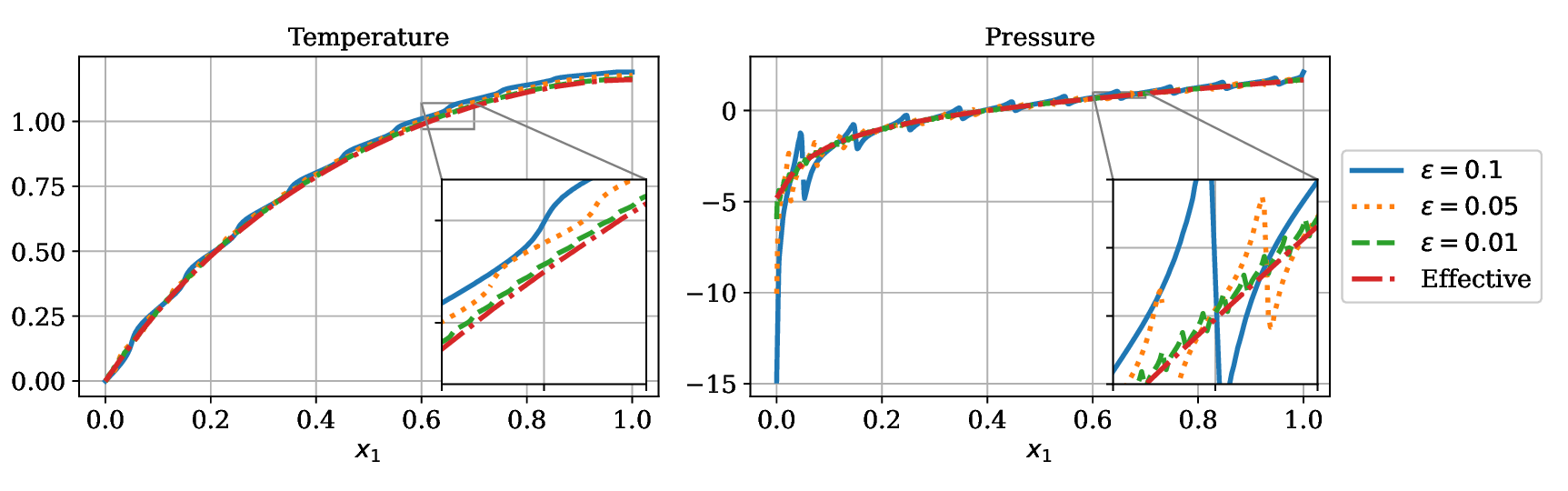}
    \caption{Comparison of the effective solutions with the microscopic temperature and pressure inside the layer. The micro solution is plotted along the line $(0, 1) \times \{\nicefrac{\e}{2}\}$. All solutions are shown at the final time step $t=5$.}
    \label{fig:eps_temp_pressure}
\end{figure}

We also want to study numerically the convergence behavior for $\e \to 0$ explicitly and not only compare visually. 
The solutions, however, are defined on different domains (the effective solutions are defined on a line and the $\e$-solutions on a thin domain), but given the results in \Cref{sec:homogenization}, the functions in the thin layer can be approximated by the effective solutions via 
\begin{equation}\label{eq:sim_reconstruction}
    \begin{split}
        \thetaIdx{f}(t, x) \approx \theta^f(t, &\tildex) + \e \theta^f_1(t, \tildex, \nicefrac{x}{\e}), \quad 
        \e^2 p_\e(t, x) \approx p(t, \tildex) + \e p_1(t, x, \nicefrac{x}{\e})
        \\
        &u(t, x)\approx \xi_0(\nicefrac{x}{\e}) - \frac{\xi_1(\nicefrac{x}{\e})} {\mu(\theta^f(t, \tildex))} \partial_{x_1} p(t, \tildex).
    \end{split}
\end{equation}
The functions $\theta^f_1$ and $p_1$ are defined in \Cref{eq:identiy_with_cell_solutions_temp} and (\ref{eq:identiy_with_cell_solutions}) and can be computed with the cell solutions given by \Cref{eq:cell_problem_temp}, (\ref{eq:cell_problem_stokes}), (\ref{eq:cell_problem_stokes_bc_movement}) and the homogenized solution.
With the reconstruction (\ref{eq:sim_reconstruction}), we compute the difference of the micro solution with the effective solution in the $L^2(S\times \OmegaIdx{f})$ norm.
The results are shown in \cref{fig:eps_convergence}.
For temperature and pressure, we observe a convergence of order $\mathcal{O}(\e^2)$. On the other hand, the velocity converges only with order $\mathcal{O}(\e)$, which is to be expected given that the temperature and pressure reconstruction also include higher order terms compared to the velocity.
Additionally, the temperature in the solid $\thetaIdx{s}$ converges also with order $\mathcal{O}(\e)$ in $L^2(S\times\OmegaIdx{s})$, but is not depicted for clearer visualization. 
\begin{figure}[ht]
    \centering
    \includegraphics[width=0.8\linewidth]{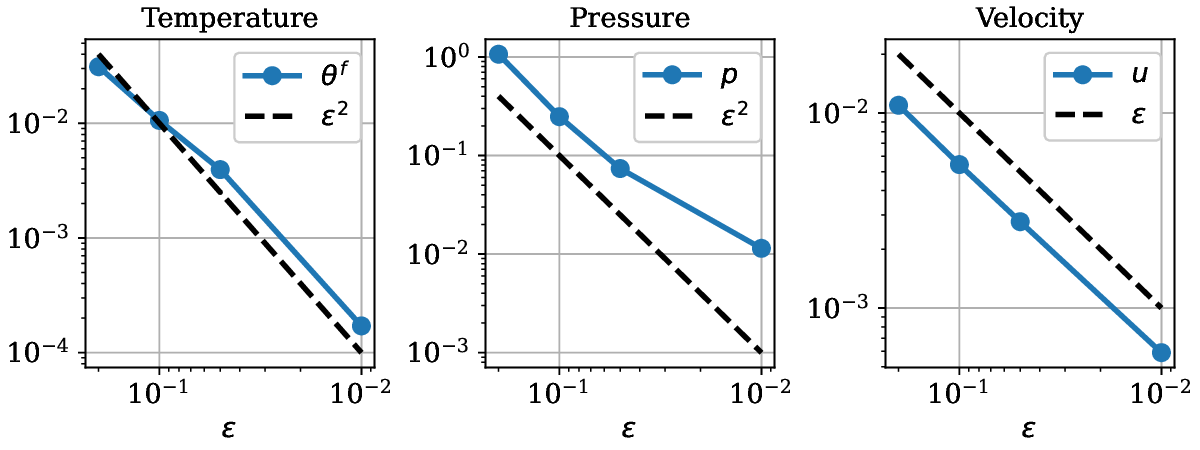}
    \caption{Convergence study for the limit $\e \to 0$. Both axes are on a logarithmic scale.
    The error between microscopic resolved simulations and the effective quantities reconstructed with \Cref{eq:sim_reconstruction} in the $L^2(S\times \OmegaIdx{f})$ norm.
    To better estimate the order of convergence, the black lines show different powers of $\e$.}
    \label{fig:eps_convergence}
\end{figure}

\textbf{Varying the inflow function}. To study how well the effective model captures the influence of the prescribed inflow, we compare different Dirichlet conditions for the roughness given by $\gamma_\text{sine}$ for $\gamma_0=0.5$ and fixed $\e=0.05$.
The resulting temperature, pressure, and velocity profiles at time $t=5$ are visualized in \cref{fig:inflow_compare}. Here we keep the visualization scheme from \Cref{fig:eps_temp_pressure} for the temperature and pressure to highlight the effective solution profile, and not explicitly compare the error in any norm. For the fluid velocity, we compare along the vertical line $\{0.5\} \times (0, \e)$ and use \Cref{eq:sim_reconstruction} to compute the effective profile.
We first observe that, naturally, a higher average fluid flow leads to a cooler system since the heat energy is removed faster from the system.
Furthermore, inflows $u_\text{BC, quad}$ and $u_\text{BC, lin, 2}$ result in a positive pressure at the inflow interface, since the amount of fluid pressed into the domain is larger than the amount moved by the bottom boundary condition.
The homogenized model can capture the influence of the different inflow conditions fairly well with regard to all quantities.
This holds true even for $u_\text{BC, lin, 2}$, which does not directly satisfy the assumptions made in \ref{item:A5}.
\begin{figure}[ht]
    \centering
    \includegraphics[width=\linewidth]{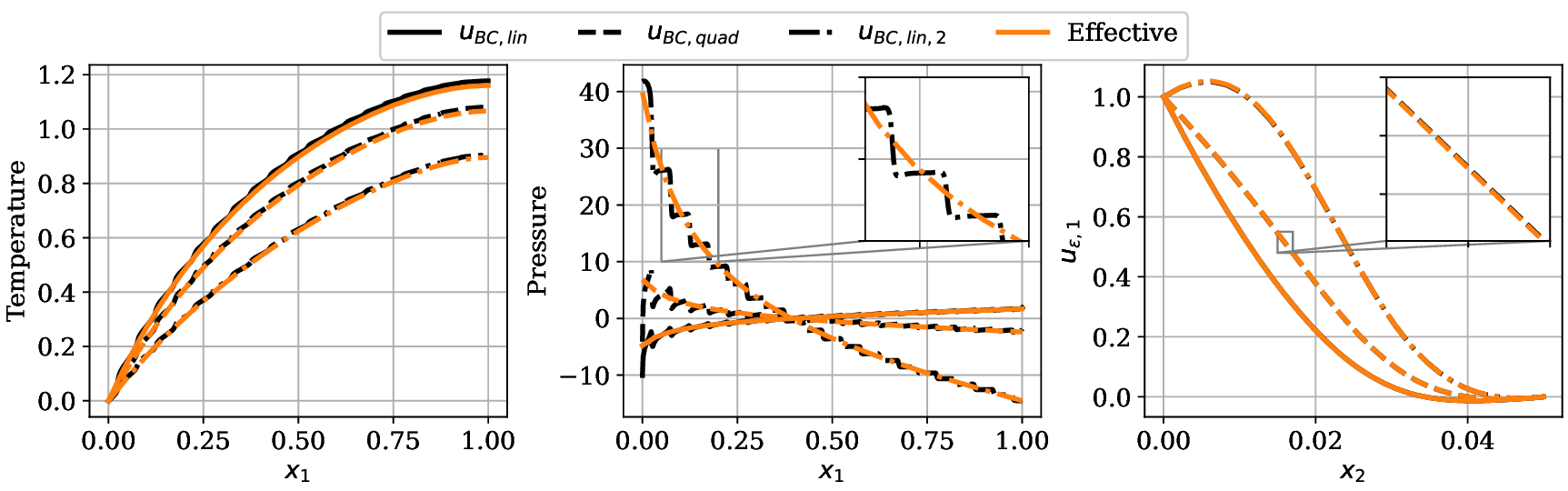}
    \caption{Comparison of the fluid temperature, pressure, and horizontal velocity component for different inflow functions, at $t=5$. In black, the micro solution, and in orange, the effective results. The velocity is compared along $\{0.5\} \times (0, \e)$.}
    \label{fig:inflow_compare}
\end{figure}

\textbf{Varying the grain height $\mathbf{\gamma_0}$}. We investigate the influence of the minimal height of the interface $\gamma_0$.
Here, we again focus on the roughness given by $\gamma_\text{sine}$ and fixed $\e = 0.05$.
We use the inflow function $u_\text{BC, lin, 2}$ as it is independent of $\gamma_0$; this allows us to focus on the influence of the interface height.
Similarly, we modify the heat source to $f_\e = \tfrac{1}{|\Gamma|}$, ensuring that the total amount of heat produced is one and is independent of $\gamma_0$.
The simulated solutions, again at $t=5$, are presented in \cref{fig:gamma_compare}.
As in the results before, the homogenized model shows good agreement with the exact solutions for $\e=0.05$. Given that we plot along the line at $x_2=\nicefrac{\e}{2}$, there appear jumps for the case $\gamma_0=0.1$ since the line is not always inside the domain.
The grain height has a strong influence on the pressure profile; since we have a fixed amount of inflow for all cases, a smaller gap necessarily leads to higher pressures inside the layer. 
At the same time, the fluid velocity in between the "humps" created by the roughness is also reduced, leading to slightly higher temperatures.
\begin{figure}[ht]
    \centering
    \includegraphics[width=\linewidth]{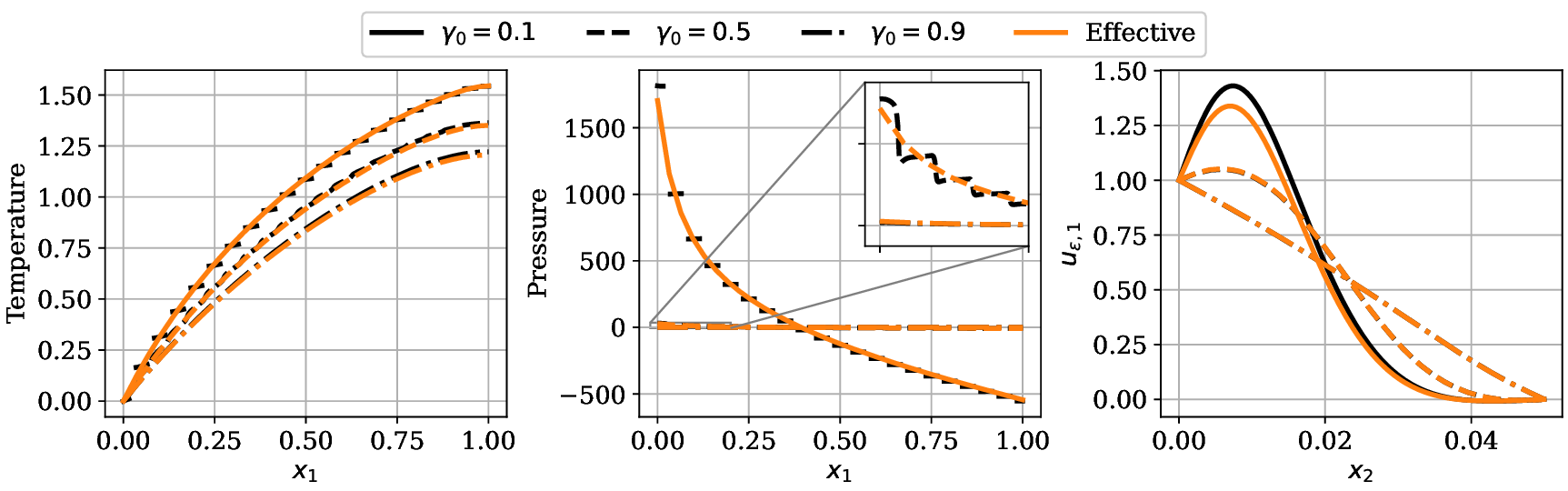}
    \caption{Comparison of the fluid temperature, pressure, and horizontal velocity component for different roughness heights $\gamma_0$, at $t=5$. The solution is visualized as in \cref{fig:inflow_compare}.}
    \label{fig:gamma_compare}
\end{figure}

\textbf{Varying the roughness pattern}. In this simulation scenario for the two-dimensional case, we demonstrate the influence of different roughness patterns.
We utilize the same setup as in the previous case for varying roughness height, but this time we fix $\gamma_0=0.5$ and compare the differences between $\gamma_\text{sine}$ and $\gamma_\text{rect}$.
As before, we present the results at the last time step, in \cref{fig:shape_compare}.
The homogenized model accurately reflects the influence of the roughness pattern on all solution fields.  
For the 2D case, we can conclude that the derived effective model yields a good approximation of the resolved micro model, and the temperature-dependent viscosity only has an impact on the pressure profile.
\begin{figure}[ht]
    \centering
    \includegraphics[width=\linewidth]{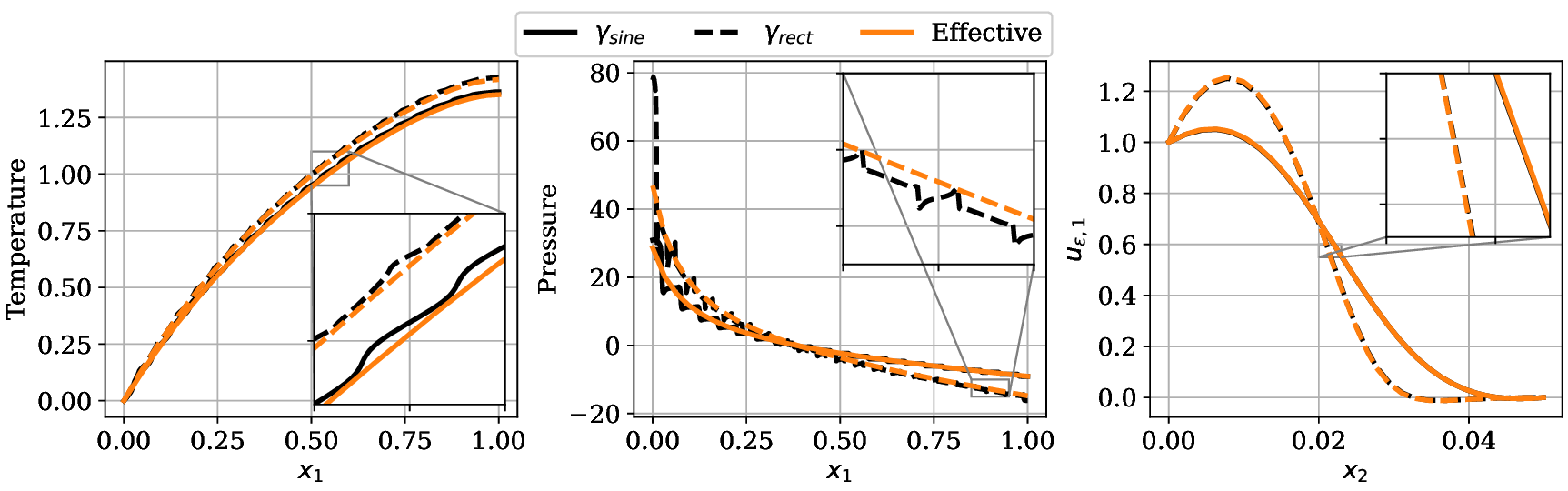}
    \caption{Comparison of the fluid temperature, pressure, and horizontal velocity component for different roughness types, at $t=5$. The solution is visualized as in \cref{fig:inflow_compare}.}
    \label{fig:shape_compare}
\end{figure}

\textbf{3D simulation study}. Lastly, we demonstrate the effects of the temperature-dependent viscosity on the fluid velocity $u$.
For this, a three-dimensional setup is required.
Unfortunately, the $\e$-problems become computationally very expensive in this case as they require a lot of memory.
Therefore, we only simulate the $\e$-model for $\e = 0.1$ (as opposed to the 2D case where we went down to $\e=0.01$).
To increase the effect of temperature-dependent viscosity on the flow field, we set $\kappa^f = 0.01$ and $C=0.4$. Additionally, we consider a local heat source given by
\begin{equation*}
    f_\e^s(x) = \begin{cases}
        1 \quad &\text{, if }  x_2 > 0.5 \text{ and } |x_1 - 0.5| < 0.1, 
        \\
        0 &\text{ else}.
    \end{cases}
\end{equation*}
This right-hand side is chosen in such a way that local variations in the heat production occur, leading to local differences in the viscosity.
The inflow is defined by $u_\text{BC, lin}$ with $\gamma_0=0.5$ and $u_\text{motion}= (1, 0, 0)^T$. For the roughness, we again model the reference cell as a graph and define
\begin{equation*}
    \gamma = 1 - 0.9\left(\cos{(2\pi(x_2 - 0.5)^2)} \cos(2\pi(x_1 - 0.25))\right)^2,
\end{equation*}
so that the effective quantities behave differently in directions $e_1$ and $e_2$. The corresponding effective parameters are collected in \cref{tab:effective_parameters_3d}.
\begin{table}[ht]
    \centering
    \caption{Effective parameters for the three-dimensional simulation. The effective values corresponding to dimension $d$ are not listed, since they are zero.} 
    \begin{tabular}{|c|c|c|c|}
    	 \hline
          $Z$ & $\Tilde{\kappa}$ & $K$ & $\Bar{\xi}_0$ \\
         \hline
         &&& 
         \\[\dimexpr-\normalbaselineskip+2pt]
         \noindent\parbox[c]{25mm}{\includegraphics[width=25mm, height=25mm]{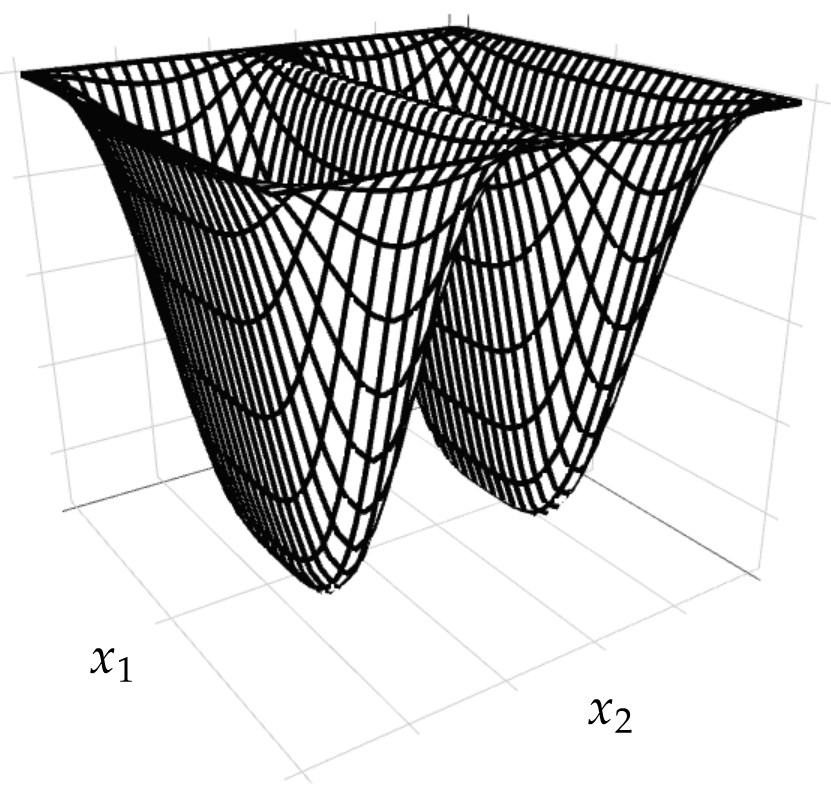}}
         & $\left(\begin{array}{cc}
                0.493 & 0.0 \\
                0.0 & 0.635
         \end{array}\right)$
         & $\left(\begin{array}{cc}
                0.004566 & 0.0 \\
                0.0 & 0.004947
         \end{array}\right)$
         & $\left(\begin{array}{c}
                0.134 \\
                0.0
         \end{array}\right)$
    \\\hline
    \end{tabular}
    \label{tab:effective_parameters_3d}
\end{table}
The nonlinearity is resolved by splitting up the problem and doing the time stepping for temperature and Stokes problem sequentially.
To be precise, we do a time step for the fluid velocity and pressure using the temperature from the previous time step, followed by a time step of the temperature problem with the velocity of the new time step. 

The simulation results at three distinct time steps are presented in \cref{fig:3d_simulations}.
The effect of the temperature-dependent viscosity can be observed in the flow profile.
At the start ($t=0.05$), the temperature is more or less constant, so there is a uniform velocity field.
The viscosity decreases locally ($t=1.85$) due to the localized heat source, and the fluid flows in a diagonal direction towards the area of lower viscosity.
As soon as the heat starts to spread through the advection, the observed effect decreases $(t=5.0)$.
We conclude, that the temperature-dependent viscosity can lead to effects that a constant viscosity could not depict, both in the microscale and effective model.
This effect appears in both models but is easier to recognize, in \cref{fig:3d_simulations}, for the homogenized case, since the mesh does not need to resolve the roughness and has a uniform structure.
Again, the effective model yields a good approximation of the resolved $\e$-model while requiring much less computational effort. 
\begin{figure}[ht]
    \centering
    \begin{subfigure}[b]{0.91\textwidth}
        \centering
        \begin{subfigure}[b]{0.25\textwidth}
             \centering
             \includegraphics[width=\textwidth]{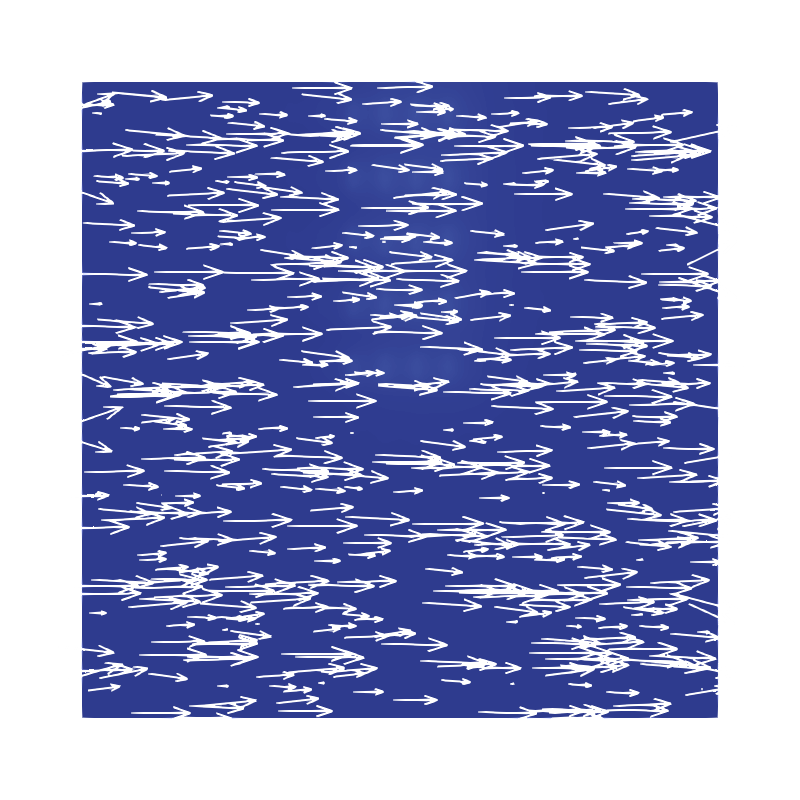}
        \end{subfigure}
        \begin{subfigure}[b]{0.25\textwidth}
             \centering
             \includegraphics[width=\textwidth]{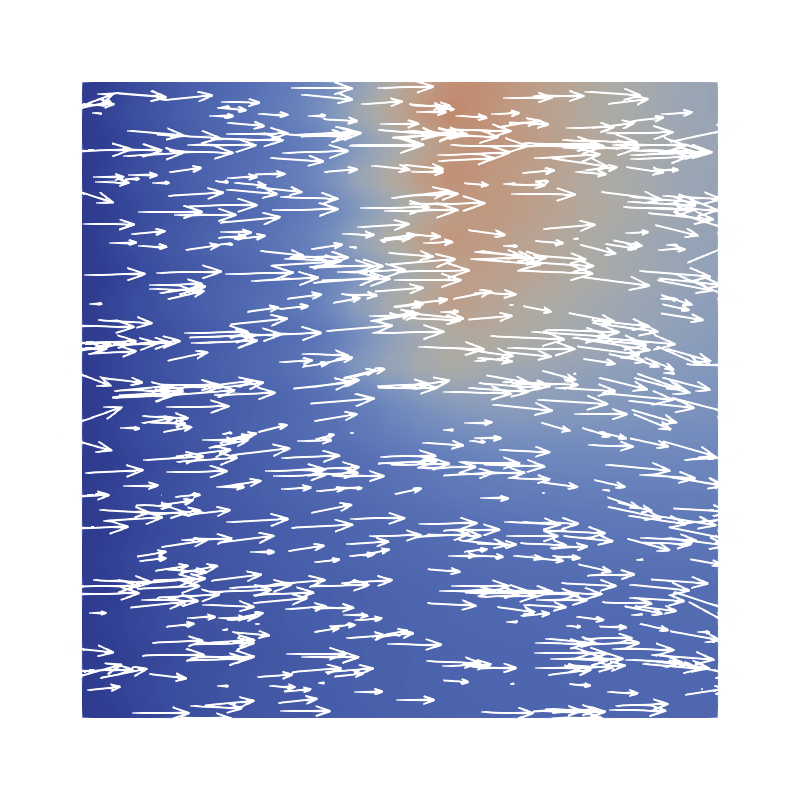}
        \end{subfigure}
        \begin{subfigure}[b]{0.25\textwidth}
             \centering
             \includegraphics[width=\textwidth]{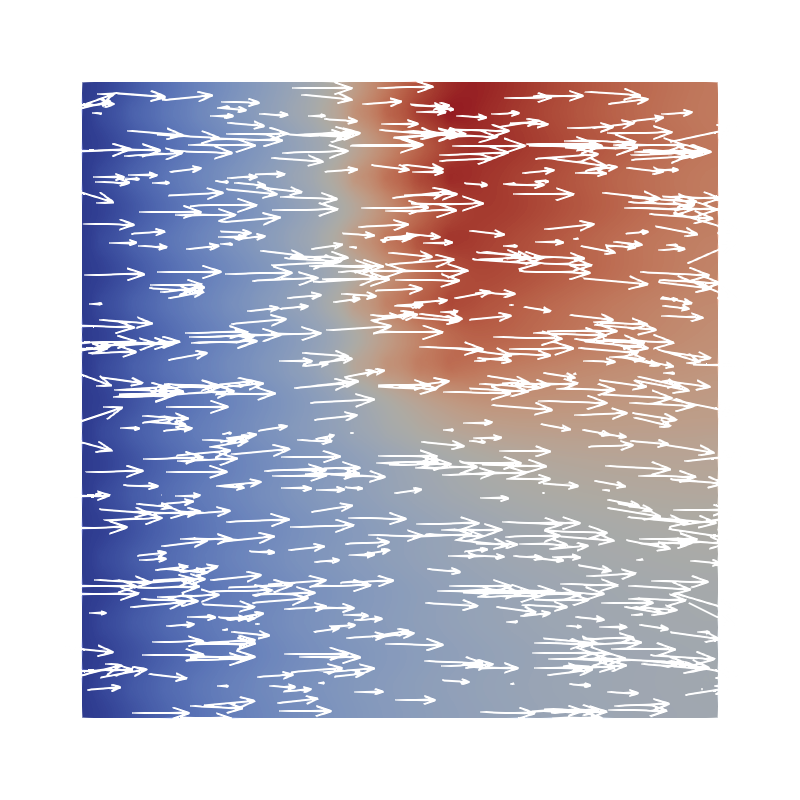}
        \end{subfigure}
        \\
        \begin{subfigure}[b]{0.25\textwidth}
             \centering
             \includegraphics[width=\textwidth]{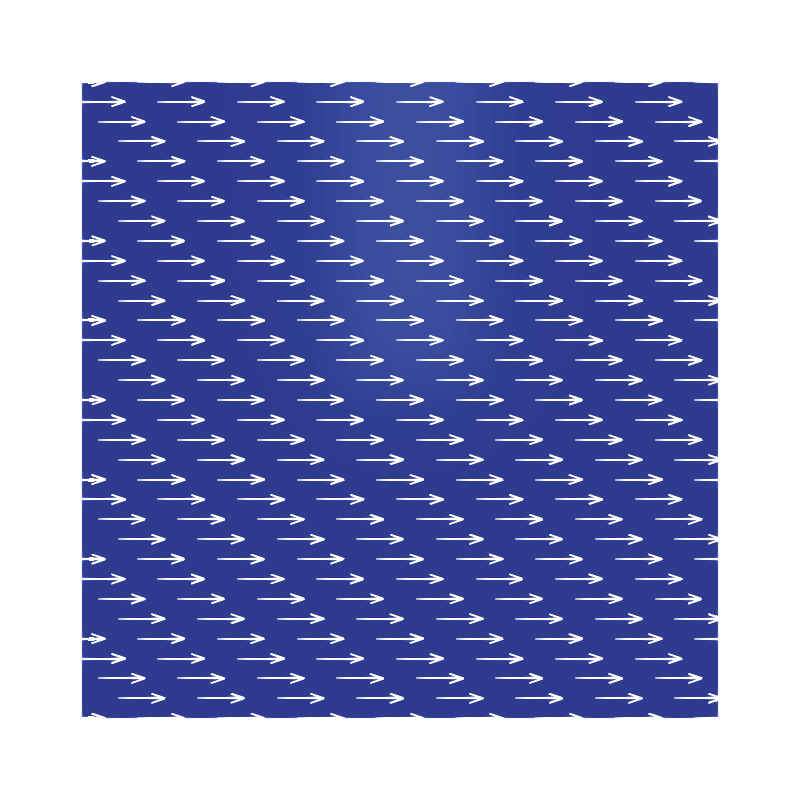}
             \caption{$t = 0.05$}
        \end{subfigure}
        \begin{subfigure}[b]{0.25\textwidth}
             \centering
             \includegraphics[width=\textwidth]{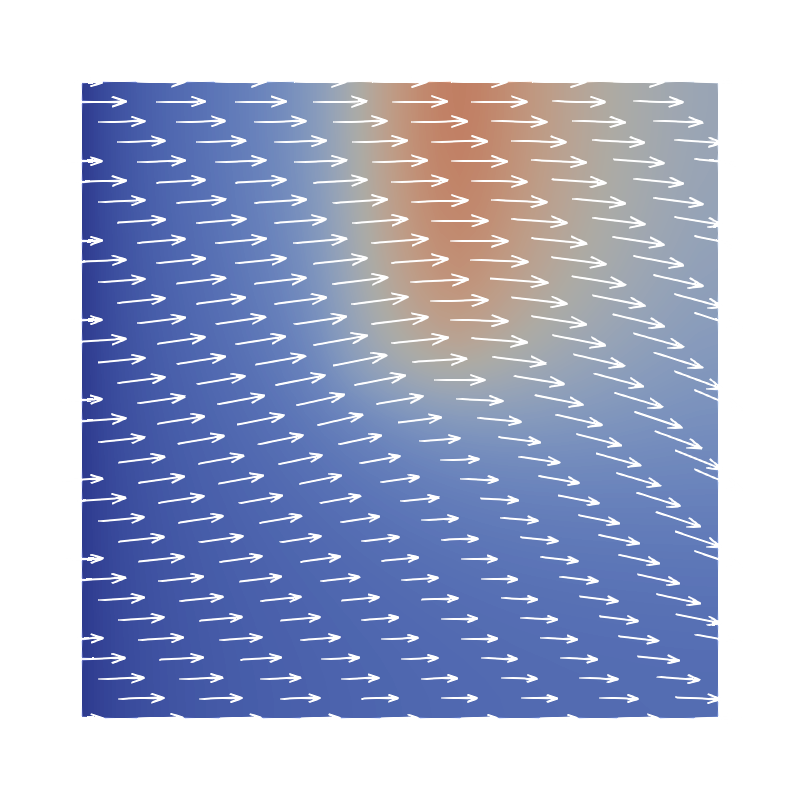}
             \caption{$t = 1.85$}
        \end{subfigure}
        \begin{subfigure}[b]{0.25\textwidth}
             \centering
             \includegraphics[width=\textwidth]{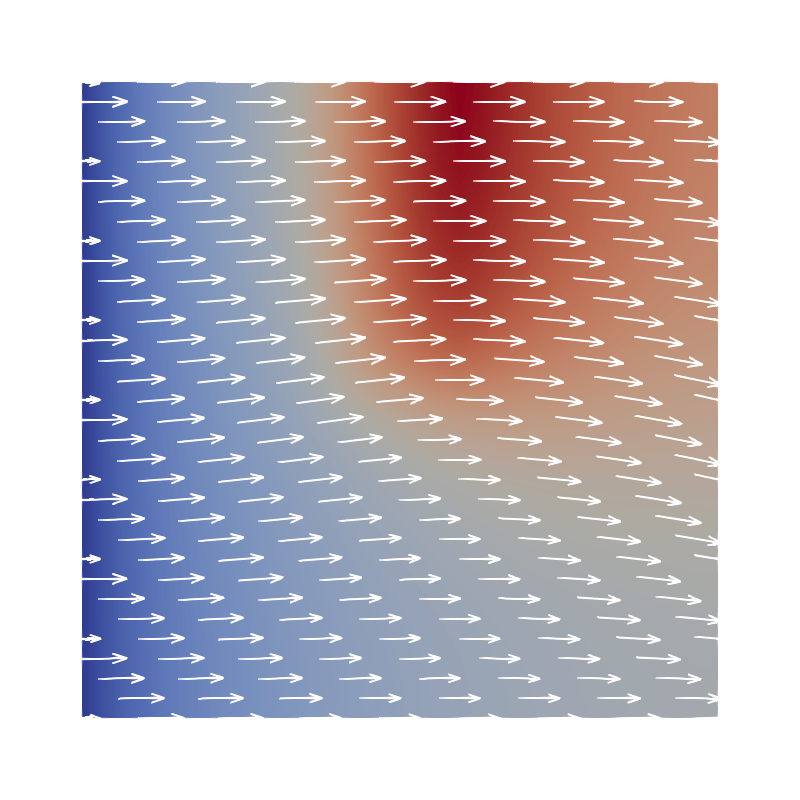}
             \caption{$t = 5.0$}
        \end{subfigure}
    \end{subfigure} 
    \hfill
    \begin{subfigure}[b]{0.07\textwidth}
        \hspace{-45pt}
        \includegraphics[width=\textwidth]{Images/simulations/EpsilonCompare/colorbar.png}
        \vspace{70pt}
    \end{subfigure}
    \caption{Temperature and velocity profile for the three-dimensional setup, at different snapshots in time. At the top, is the resolved micro model, and at the bottom, is the effective model. For the micro model, a cross-section at $x_3=0.01$ is visualized, such that the flow profile is better recognizable. Note, that the mesh for the micro model needs to resolve the roughness structure and therefore the representation of the velocity profile is rather disorganized, while in the effective case, we can use a uniform mesh.}
    \label{fig:3d_simulations}
\end{figure}
\section{Conclusion}
We investigated the effective behavior of a nonlinear Stokes-temperature system in a thin layer, where the nonlinearity arises through the convection term in the heat equation and temperature-dependent viscosity.
With the concept of two-scale convergence, an effective model was derived where the effective fluid temperature is given by an interface function. 
Multiple simulation results validate the derived model and show that temperature-dependent viscosity also plays a non-negligible role in the limit $\e \to 0$. 
The main contribution of this article is the $L^\infty$ estimate for the temperature equation in the thin rough layer and the homogenization of a highly non-linear coupled problem.

In reality, other parameters (such as conductivity or heat exchange parameter) may also be temperature dependent.
Similarly to the present approach, for example, by decoupling the system and showing the existence of solutions with a fixed-point argument, temperature-dependent $\kappa$ and $\alpha$ could be handled as well. Of course, with fitting assumptions for the parameters.
Another aspect that could be investigated further is the current $\e$-scaling of $\kappa$ in the thin layer.
Removing this scaling complicates the analysis and may lead to an effective equation without diffusion on the interface $\Sigma$, as in \cite{Almqvist2002, GU2004}.

Regarding the motivating application of a grinding process, we could derive a justification for the models used in the literature \cite{GU2004}, as explained already in the introduction.
Currently, we are looking at a strongly idealized scenario of periodic roughness and non-touching roughness with the bottom boundary.
In future research, we plan to extend this model to real grinding wheels and workpiece surfaces. Another interesting direction is studying the impact of direct contact between the workpiece and the wheel on the fluid behavior in the contact zone.
Additionally, a multiscale approach, also including microscopic material removal, is currently being investigated, studying the effect of microstructures on the whole process.   

\section*{Acknowledgements}
We thank the anonymous reviewers for their constructive and detailed feedback, which helped to improve the article. 
TF acknowledges funding by the Deutsche Forschungsgemeinschaft (DFG, German Research Foundation) -- project nr.  
281474342/GRK2224/2. Additionally, his research was also partially funded by the DFG -- project nr. 439916647.

The research activity of ME is funded by the Horizon 2022 research and innovation program of the European Union under the Marie Skłodowska-Curie fellowship project {\em{MATT}} (project nr. 101061956, \url{https://doi.org/10.3030/101061956}).

\section*{Conflict of interest}
The authors declare that there is no conflict of interest.

\section*{Data Availability Statement}
The research code associated with this article is available in the public repository \cite{Code}.

\bibliography{reference}

\begin{thebibliography}{10}
\expandafter\ifx\csname url\endcsname\relax
  \def\url#1{\burl{#1}}\fi
\expandafter\ifx\csname urlprefix\endcsname\relax\def\urlprefix{URL }\fi
\providecommand{\bibinfo}[2]{#2}
\providecommand{\eprint}[2][]{\url{#2}}
\providecommand{\doi}[1]{\url{https://doi.org/#1}}
\bibcommenthead

\bibitem{klocke2009}
\bibinfo{author}{Klocke, F.} \& \bibinfo{author}{Kuchle, A.}
\newblock \emph{\bibinfo{title}{Manufacturing Processes 2: Grinding, Honing, Lapping}} RWTHedition (\bibinfo{publisher}{Springer Berlin Heidelberg}, \bibinfo{year}{2009}).
\newblock \urlprefix\url{https://books.google.de/books?id=rHiNm-OlSN8C}.

\bibitem{TONSHOFF1992}
\bibinfo{author}{Tönshoff, H.}, \bibinfo{author}{Peters, J.}, \bibinfo{author}{Inasaki, I.} \& \bibinfo{author}{Paul, T.}
\newblock \bibinfo{title}{Modelling and simulation of grinding processes}.
\newblock \emph{\bibinfo{journal}{CIRP Annals}} \textbf{\bibinfo{volume}{41}}, \bibinfo{pages}{677--688} (\bibinfo{year}{1992}).
\newblock \urlprefix\url{https://www.sciencedirect.com/science/article/pii/S0007850607632545}.

\bibitem{Wiederkehr2023}
\bibinfo{author}{Wiederkehr, P.}, \bibinfo{author}{Grimmert, A.}, \bibinfo{author}{Heining, I.}, \bibinfo{author}{Siebrecht, T.} \& \bibinfo{author}{Wöste, F.}
\newblock \bibinfo{title}{Potentials of grinding process simulations for the analysis of individual grain engagement and complete grinding processes}.
\newblock \emph{\bibinfo{journal}{Frontiers in Manufacturing Technology}} \textbf{\bibinfo{volume}{2}} (\bibinfo{year}{2023}).
\newblock \urlprefix\url{https://www.frontiersin.org/articles/10.3389/fmtec.2022.1102140}.

\bibitem{GU2004}
\bibinfo{author}{Gu, R.}, \bibinfo{author}{Shillor, M.}, \bibinfo{author}{Barber, G.} \& \bibinfo{author}{Jen, T.}
\newblock \bibinfo{title}{Thermal analysis of the grinding process}.
\newblock \emph{\bibinfo{journal}{Mathematical and Computer Modelling}} \textbf{\bibinfo{volume}{39}}, \bibinfo{pages}{991--1003} (\bibinfo{year}{2004}).
\newblock \urlprefix\url{https://www.sciencedirect.com/science/article/pii/S0895717704905304}.

\bibitem{Garca-Coln1989}
\bibinfo{author}{Garca-Coln, L.~S.}, \bibinfo{author}{{del Castillo LF}} \& \bibinfo{author}{Goldstein, P.}
\newblock \bibinfo{title}{Theoretical basis for the {Vogel-Fulcher-Tammann} equation}.
\newblock \emph{\bibinfo{journal}{Phys. Rev. B Condens. Matter}} \textbf{\bibinfo{volume}{40}}, \bibinfo{pages}{7040--7044} (\bibinfo{year}{1989}).

\bibitem{Perez_2008}
\bibinfo{author}{Pérez, C.}, \bibinfo{author}{Thomas, J.-M.}, \bibinfo{author}{Blancher, S.} \& \bibinfo{author}{Creff, R.}
\newblock \bibinfo{title}{The steady {N}avier–{S}tokes/energy system with temperature‐dependent viscosity—part 1: {A}nalysis of the continuous problem}.
\newblock \emph{\bibinfo{journal}{International Journal for Numerical Methods in Fluids}} \textbf{\bibinfo{volume}{56}}, \bibinfo{pages}{63 -- 89} (\bibinfo{year}{2008}).

\bibitem{Paiva2021}
\bibinfo{author}{de~Paiva, R.~L.}, \bibinfo{author}{de~Souza~Ruzzi, R.} \& \bibinfo{author}{da~Silva, R.~B.}
\newblock \bibinfo{title}{An approach to reduce thermal damages on grinding of bearing steel by controlling cutting fluid temperature}.
\newblock \emph{\bibinfo{journal}{Metals}} \textbf{\bibinfo{volume}{11}} (\bibinfo{year}{2021}).
\newblock \urlprefix\url{https://www.mdpi.com/2075-4701/11/10/1660}.

\bibitem{Wiesener2023}
\bibinfo{author}{Wiesener, F.} \emph{et~al.}
\newblock \bibinfo{title}{Modeling of heat transfer in tool grinding for multiscale simulations}.
\newblock \emph{\bibinfo{journal}{Procedia CIRP}} \textbf{\bibinfo{volume}{117}}, \bibinfo{pages}{269–274} (\bibinfo{year}{2023}).
\newblock \urlprefix\url{http://dx.doi.org/10.1016/j.procir.2023.03.046}.

\bibitem{Thunich2023}
\bibinfo{author}{Thunich, P.} \emph{et~al.}
\newblock \bibinfo{title}{On the validity of the reynolds equation for the simulation of wet grinding processes}.
\newblock \emph{\bibinfo{journal}{PAMM}} \textbf{\bibinfo{volume}{23}}, \bibinfo{pages}{e202300109} (\bibinfo{year}{2023}).
\newblock \urlprefix\url{https://onlinelibrary.wiley.com/doi/abs/10.1002/pamm.202300109}.

\bibitem{VESALI2014}
\bibinfo{author}{Vesali, A.} \& \bibinfo{author}{Tawakoli, T.}
\newblock \bibinfo{title}{Study on hydrodynamic pressure in grinding contact zone considering grinding parameters and grinding wheel specifications}.
\newblock \emph{\bibinfo{journal}{Procedia CIRP}} \textbf{\bibinfo{volume}{14}}, \bibinfo{pages}{13--18} (\bibinfo{year}{2014}).
\newblock \urlprefix\url{https://www.sciencedirect.com/science/article/pii/S2212827114001966}.
\newblock \bibinfo{note}{6th CIRP International Conference on High Performance Cutting, HPC2014}.

\bibitem{BAIR2001}
\bibinfo{author}{Bair, S.}, \bibinfo{author}{Jarzynski, J.} \& \bibinfo{author}{Winer, W.~O.}
\newblock \bibinfo{title}{The temperature, pressure and time dependence of lubricant viscosity}.
\newblock \emph{\bibinfo{journal}{Tribology International}} \textbf{\bibinfo{volume}{34}}, \bibinfo{pages}{461--468} (\bibinfo{year}{2001}).
\newblock \urlprefix\url{https://www.sciencedirect.com/science/article/pii/S0301679X01000421}.

\bibitem{HABCHI2010}
\bibinfo{author}{Habchi, W.} \emph{et~al.}
\newblock \bibinfo{title}{Influence of pressure and temperature dependence of thermal properties of a lubricant on the behaviour of circular tehd contacts}.
\newblock \emph{\bibinfo{journal}{Tribology International}} \textbf{\bibinfo{volume}{43}}, \bibinfo{pages}{1842--1850} (\bibinfo{year}{2010}).
\newblock \urlprefix\url{https://www.sciencedirect.com/science/article/pii/S0301679X09003028}.
\newblock \bibinfo{note}{36th Leeds–Lyon Symposium Special Issue: Multi-facets of Tribology}.

\bibitem{KUMAR2020}
\bibinfo{author}{Kumar, K.}, \bibinfo{author}{List, F.}, \bibinfo{author}{Pop, I.~S.} \& \bibinfo{author}{Radu, F.~A.}
\newblock \bibinfo{title}{Formal upscaling and numerical validation of unsaturated flow models in fractured porous media}.
\newblock \emph{\bibinfo{journal}{Journal of Computational Physics}} \textbf{\bibinfo{volume}{407}}, \bibinfo{pages}{109138} (\bibinfo{year}{2020}).
\newblock \urlprefix\url{https://www.sciencedirect.com/science/article/pii/S0021999119308435}.

\bibitem{KRZACZEK2023}
\bibinfo{author}{Krzaczek, M.} \& \bibinfo{author}{Tejchman, J.}
\newblock \bibinfo{title}{Hydraulic fracturing process in rocks – small-scale simulations with a novel fully coupled dem/cfd-based thermo-hydro-mechanical approach}.
\newblock \emph{\bibinfo{journal}{Engineering Fracture Mechanics}} \textbf{\bibinfo{volume}{289}}, \bibinfo{pages}{109424} (\bibinfo{year}{2023}).
\newblock \urlprefix\url{https://www.sciencedirect.com/science/article/pii/S001379442300382X}.

\bibitem{Bassion2008}
\bibinfo{author}{Basson, A.} \& \bibinfo{author}{Gérard-Varet, D.}
\newblock \bibinfo{title}{Wall laws for fluid flows at a boundary with random roughness}.
\newblock \emph{\bibinfo{journal}{Communications on Pure and Applied Mathematics}} \textbf{\bibinfo{volume}{61}}, \bibinfo{pages}{941--987} (\bibinfo{year}{2008}).
\newblock \urlprefix\url{https://onlinelibrary.wiley.com/doi/abs/10.1002/cpa.20237}.

\bibitem{CHECHKIN1999}
\bibinfo{author}{Chechkin, G.~A.}, \bibinfo{author}{Friedman, A.} \& \bibinfo{author}{Piatnitski, A.~L.}
\newblock \bibinfo{title}{The boundary-value problem in domains with very rapidly oscillating boundary}.
\newblock \emph{\bibinfo{journal}{Journal of Mathematical Analysis and Applications}} \textbf{\bibinfo{volume}{231}}, \bibinfo{pages}{213--234} (\bibinfo{year}{1999}).
\newblock \urlprefix\url{https://www.sciencedirect.com/science/article/pii/S0022247X98962266}.

\bibitem{Mikeli2009}
\bibinfo{author}{Mikeli{\'c}, A.}
\newblock \bibinfo{title}{Rough boundaries and wall laws}.
\newblock \emph{\bibinfo{journal}{In:{Q}ualitative {P}roperties of {S}olutions to {P}artial {D}ifferential {E}quations, {J}indich {N}ecas {C}enter for {M}athematical {M}odeling {L}ecture {N}otes}} \textbf{\bibinfo{volume}{5}}, \bibinfo{pages}{103–134} (\bibinfo{year}{2009}).

\bibitem{Donato_2019}
\bibinfo{author}{Donato, P.}, \bibinfo{author}{Jose, E.} \& \bibinfo{author}{Onofrei, D.}
\newblock \bibinfo{title}{Asymptotic analysis of a multiscale parabolic problem with a rough fast oscillating interface}.
\newblock \emph{\bibinfo{journal}{Archive of Applied Mechanics}} \textbf{\bibinfo{volume}{89}}, \bibinfo{pages}{1--29} (\bibinfo{year}{2019}).
\newblock \urlprefix\url{https://www.doi.org/10.1007/s00419-018-1415-5}.

\bibitem{Nevard1997}
\bibinfo{author}{Nevard, J.} \& \bibinfo{author}{Keller, J.~B.}
\newblock \bibinfo{title}{Homogenization of rough boundaries and interfaces}.
\newblock \emph{\bibinfo{journal}{SIAM Journal on Applied Mathematics}} \textbf{\bibinfo{volume}{57}}, \bibinfo{pages}{1660--1686} (\bibinfo{year}{1997}).
\newblock \urlprefix\url{https://doi.org/10.1137/S0036139995291088}.

\bibitem{AHMED2017}
\bibinfo{author}{Ahmed, E.}, \bibinfo{author}{Jaffré, J.} \& \bibinfo{author}{Roberts, J.~E.}
\newblock \bibinfo{title}{A reduced fracture model for two-phase flow with different rock types}.
\newblock \emph{\bibinfo{journal}{Mathematics and Computers in Simulation}} \textbf{\bibinfo{volume}{137}}, \bibinfo{pages}{49--70} (\bibinfo{year}{2017}).
\newblock \urlprefix\url{https://www.sciencedirect.com/science/article/pii/S0378475416301987}.
\newblock \bibinfo{note}{MAMERN VI-2015: 6th International Conference on Approximation Methods and Numerical Modeling in Environment and Natural Resources}.

\bibitem{Pop2016}
\bibinfo{author}{Pop, I.}, \bibinfo{author}{Bogers, J.} \& \bibinfo{author}{Kumar, K.}
\newblock \bibinfo{title}{Analysis and upscaling of a reactive transport model in fractured porous media with nonlinear transmission condition}.
\newblock \emph{\bibinfo{journal}{Vietnam Journal of Mathematics}} \textbf{\bibinfo{volume}{45}} (\bibinfo{year}{2016}).

\bibitem{Gahn17}
\bibinfo{author}{Gahn, M.}, \bibinfo{author}{Neuss-Radu, M.} \& \bibinfo{author}{Knabner, P.}
\newblock \bibinfo{title}{Derivation of effective transmission conditions for domains separated by a membrane for different scaling of membrane diffusivity}.
\newblock \emph{\bibinfo{journal}{Discrete and Continuous Dynamical Systems - Series S}} \textbf{\bibinfo{volume}{10}}, \bibinfo{pages}{773--797} (\bibinfo{year}{2017}).
\newblock \urlprefix\url{https://www.doi.org/10.3934/dcdss.2017039}.

\bibitem{Neuss07}
\bibinfo{author}{Neuss-Radu, M.} \& \bibinfo{author}{J\"{a}ger, W.}
\newblock \bibinfo{title}{Effective transmission conditions for reaction-diffusion processes in domains separated by an interface}.
\newblock \emph{\bibinfo{journal}{SIAM Journal on Mathematical Analysis}} \textbf{\bibinfo{volume}{39}}, \bibinfo{pages}{687--720} (\bibinfo{year}{2007}).
\newblock \urlprefix\url{https://doi.org/10.1137/060665452}.

\bibitem{Bayada_1986}
\bibinfo{author}{Bayada, G.} \& \bibinfo{author}{Chambat, M.}
\newblock \bibinfo{title}{The transition between the {S}tokes equations and the {R}eynolds equation: A mathematical proof}.
\newblock \emph{\bibinfo{journal}{Applied Mathematics and Optimization}} \textbf{\bibinfo{volume}{14}}, \bibinfo{pages}{73--93} (\bibinfo{year}{1986}).
\newblock \urlprefix\url{https://doi.org/10.1007/BF01442229}.

\bibitem{Bayda_1989}
\bibinfo{author}{Bayada, G.} \& \bibinfo{author}{Chambat, M.}
\newblock \bibinfo{title}{Homogenization of the {S}tokes system in a thin film flow with rapidly varying thickness}.
\newblock \emph{\bibinfo{journal}{RAIRO. Modélisation Mathématique et Analyse Numérique}} \textbf{\bibinfo{volume}{23}} (\bibinfo{year}{1989}).

\bibitem{Reynolds1886}
\bibinfo{author}{Reynolds, O.}
\newblock \bibinfo{title}{On the theory of lubrication and its application to mr. beauchamp tower's experiments, including an experimental determination of the viscosity of olive oil}.
\newblock \emph{\bibinfo{journal}{Philosophical Transactions of the Royal Society of London}} \textbf{\bibinfo{volume}{177}}, \bibinfo{pages}{157--234} (\bibinfo{year}{1886}).
\newblock \urlprefix\url{http://www.jstor.org/stable/109480}.

\bibitem{ALMQVIST2007}
\bibinfo{author}{Almqvist, A.}, \bibinfo{author}{Essel, E.}, \bibinfo{author}{Persson, L.-E.} \& \bibinfo{author}{Wall, P.}
\newblock \bibinfo{title}{Homogenization of the unstationary incompressible reynolds equation}.
\newblock \emph{\bibinfo{journal}{Tribology International}} \textbf{\bibinfo{volume}{40}}, \bibinfo{pages}{1344--1350} (\bibinfo{year}{2007}).
\newblock \urlprefix\url{https://www.sciencedirect.com/science/article/pii/S0301679X0700062X}.

\bibitem{BENHABOUCHA_2024}
\bibinfo{author}{Benhaboucha, N.}, \bibinfo{author}{Chambat, M.} \& \bibinfo{author}{Ciuperca, I.}
\newblock \bibinfo{title}{Asymptotic behaviour of pressure and stresses in a thin film flow with a rough boundary}.
\newblock \emph{\bibinfo{journal}{Quarterly of Applied Mathematics}} \textbf{\bibinfo{volume}{63}}, \bibinfo{pages}{369--400} (\bibinfo{year}{2005}).
\newblock \urlprefix\url{http://www.jstor.org/stable/43638671}.

\bibitem{Chupin2012}
\bibinfo{author}{Chupin, L.} \& \bibinfo{author}{Martin, S.}
\newblock \bibinfo{title}{Rigorous derivation of the thin film approximation with roughness-induced correctors}.
\newblock \emph{\bibinfo{journal}{SIAM Journal on Mathematical Analysis}} \textbf{\bibinfo{volume}{44}}, \bibinfo{pages}{3041--3070} (\bibinfo{year}{2012}).
\newblock \urlprefix\url{https://doi.org/10.1137/110824371}.

\bibitem{Fabricius2017}
\bibinfo{author}{Fabricius, J.}, \bibinfo{author}{Tsandzana, A.}, \bibinfo{author}{Perez-Rafols, F.} \& \bibinfo{author}{Wall, P.}
\newblock \bibinfo{title}{{A Comparison of the Roughness Regimes in Hydrodynamic Lubrication}}.
\newblock \emph{\bibinfo{journal}{Journal of Tribology}} \textbf{\bibinfo{volume}{139}}, \bibinfo{pages}{051702} (\bibinfo{year}{2017}).
\newblock \urlprefix\url{https://doi.org/10.1115/1.4035868}.

\bibitem{LUKKASSEN_2011}
\bibinfo{author}{Lukkassen, D.}, \bibinfo{author}{Meidell, A.} \& \bibinfo{author}{Wall, P.}
\newblock \bibinfo{title}{Analysis of the effects of rough surfaces in compressible thin film flow by homogenization}.
\newblock \emph{\bibinfo{journal}{International Journal of Engineering Science}} \textbf{\bibinfo{volume}{49}}, \bibinfo{pages}{369--377} (\bibinfo{year}{2011}).
\newblock \urlprefix\url{https://www.sciencedirect.com/science/article/pii/S0020722511000115}.

\bibitem{Marusic2000}
\bibinfo{author}{Marusic, S.} \& \bibinfo{author}{Marusic-Paloka, E.}
\newblock \bibinfo{title}{Two-scale convergence for thin domains and its applications to some lower-dimensional models in fluid mechanics}.
\newblock \emph{\bibinfo{journal}{Asymptotic Analysis}} \textbf{\bibinfo{volume}{23}}, \bibinfo{pages}{23--57} (\bibinfo{year}{2000}).

\bibitem{Almqvist2002}
\bibinfo{author}{Almqvist, A.}, \bibinfo{author}{Burtseva, E.}, \bibinfo{author}{Rajagopal, K.} \& \bibinfo{author}{Wall, P.}
\newblock \bibinfo{title}{On lower-dimensional models of thin film flow, part c: Derivation of a {R}eynolds type of equation for fluids with temperature and pressure dependent viscosity}.
\newblock \emph{\bibinfo{journal}{Proceedings of the Institution of Mechanical Engineers, Part J: Journal of Engineering Tribology}} \textbf{\bibinfo{volume}{237}}, \bibinfo{pages}{135065012211352} (\bibinfo{year}{2022}).

\bibitem{Fabricius2022}
\bibinfo{author}{Fabricius, J.}, \bibinfo{author}{Manjate, S.} \& \bibinfo{author}{and, P.~W.}
\newblock \bibinfo{title}{On pressure-driven hele–shaw flow of power-law fluids}.
\newblock \emph{\bibinfo{journal}{Applicable Analysis}} \textbf{\bibinfo{volume}{101}}, \bibinfo{pages}{5107--5137} (\bibinfo{year}{2022}).
\newblock \urlprefix\url{https://doi.org/10.1080/00036811.2021.1880570}.

\bibitem{Fabricius2023}
\bibinfo{author}{Fabricius, J.} \& \bibinfo{author}{Gahn, M.}
\newblock \bibinfo{title}{Homogenization and dimension reduction of the {S}tokes problem with {N}avier-slip condition in thin perforated layers}.
\newblock \emph{\bibinfo{journal}{Multiscale Modeling \& Simulation}} \textbf{\bibinfo{volume}{21}}, \bibinfo{pages}{1502--1533} (\bibinfo{year}{2023}).
\newblock \urlprefix\url{https://doi.org/10.1137/22M1528860}.

\bibitem{Gahn2024}
\bibinfo{author}{Gahn, M.} \& \bibinfo{author}{Neuss-Radu, M.}
\newblock \bibinfo{title}{Effective interface laws for fluid flow and solute transport through thin reactive porous layers}.
\newblock \emph{\bibinfo{journal}{Journal of Evolution Equations}} \textbf{\bibinfo{volume}{25}}, \bibinfo{pages}{33} (\bibinfo{year}{2025}).
\newblock \urlprefix\url{https://doi.org/10.1007/s00028-025-01061-1}.

\bibitem{Ladyvzenskaja1968}
\bibinfo{author}{Lady{\v{z}}enskaja, O.}, \bibinfo{author}{Solonniov, V.} \& \bibinfo{author}{Ural'ceva, N.}
\newblock \bibinfo{title}{Linear and quasilinear equations of parabolic type (translated from the russian by {S}. {S}mith)}.
\newblock \emph{\bibinfo{journal}{American Mathematical Society}}  (\bibinfo{year}{1988}).

\bibitem{Bhattacharya2022}
\bibinfo{author}{Bhattacharya, A.}, \bibinfo{author}{Gahn, M.} \& \bibinfo{author}{Neuss-Radu, M.}
\newblock \bibinfo{title}{Effective transmission conditions for reaction–diffusion processes in domains separated by thin channels}.
\newblock \emph{\bibinfo{journal}{Applicable Analysis}} \textbf{\bibinfo{volume}{101}}, \bibinfo{pages}{1896--1910} (\bibinfo{year}{2022}).
\newblock \urlprefix\url{https://doi.org/10.1080/00036811.2020.1789599}.

\bibitem{Code}
\bibinfo{author}{Freudenberg, T.}
\newblock \bibinfo{title}{Github repository for “{A}nalysis and {S}imulation of a {C}oupled {F}luid-{H}eat {S}ystem in a {T}hin, {R}ough {L}ayer”}.
\newblock \bibinfo{howpublished}{\url{https://github.com/TomF98/Homogenization-thin-layer-with-temperature-dependent-viscosity}}.
\newblock \urlprefix\url{https://github.com/TomF98/Homogenization-thin-layer-with-temperature-dependent-viscosity}.
\newblock \bibinfo{note}{July 9, 2024}.

\bibitem{Allaire1989}
\bibinfo{author}{Allaire, G.}
\newblock \bibinfo{title}{Homogenization of the {S}tokes flow in a connected porous medium}.
\newblock \emph{\bibinfo{journal}{Asymptotic Analysis}} \textbf{\bibinfo{volume}{2}}, \bibinfo{pages}{203--222} (\bibinfo{year}{1989}).
\newblock \urlprefix\url{https://doi.org/10.3233/ASY-1989-2302}.
\newblock \bibinfo{note}{3}.

\bibitem{Eden2024}
\bibinfo{author}{Eden, M.} \& \bibinfo{author}{Freudenberg, T.}
\newblock \bibinfo{title}{Effective heat transfer between a porous medium and a fluid layer: Homogenization and simulation}.
\newblock \emph{\bibinfo{journal}{Multiscale Modeling \& Simulation}} \textbf{\bibinfo{volume}{22}}, \bibinfo{pages}{752--783} (\bibinfo{year}{2024}).
\newblock \urlprefix\url{https://doi.org/10.1137/22M1541794}.

\bibitem{Girault_1981}
\bibinfo{author}{Girault, V.} \& \bibinfo{author}{Raviart, P.-A.}
\newblock \emph{\bibinfo{title}{Finite element approximation of the {N}avier-{S}tokes equations}}  (\bibinfo{publisher}{Springer}, \bibinfo{year}{1981}).

\bibitem{ACERBI_1992}
\bibinfo{author}{Acerbi, E.}, \bibinfo{author}{{Chiadò Piat}, V.}, \bibinfo{author}{{Dal Maso}, G.} \& \bibinfo{author}{Percivale, D.}
\newblock \bibinfo{title}{An extension theorem from connected sets, and homogenization in general periodic domains}.
\newblock \emph{\bibinfo{journal}{Nonlinear Analysis: Theory, Methods \& Applications}} \textbf{\bibinfo{volume}{18}}, \bibinfo{pages}{481--496} (\bibinfo{year}{1992}).
\newblock \urlprefix\url{https://www.sciencedirect.com/science/article/pii/0362546X92900157}.

\bibitem{Gahn_2021}
\bibinfo{author}{Gahn, M.}, \bibinfo{author}{Neuss-Radu, M.} \& \bibinfo{author}{Jäger, W.}
\newblock \bibinfo{title}{Two-scale tools for homogenization and dimension reduction of perforated thin layers: {E}xtensions, {K}orn-inequalities, and two-scale compactness of scale-dependent sets in {S}obolev spaces} (\bibinfo{year}{2021}).
\newblock \urlprefix\url{arxiv.org/abs/2112.00559}.
\newblock \bibinfo{note}{Preprint, arxiv.org/abs/2112.00559}.

\bibitem{Gahn2016}
\bibinfo{author}{Gahn, M.}, \bibinfo{author}{Neuss-Radu, M.} \& \bibinfo{author}{Knabner, P.}
\newblock \bibinfo{title}{Homogenization of reaction--diffusion processes in a two-component porous medium with nonlinear flux conditions at the interface}.
\newblock \emph{\bibinfo{journal}{SIAM Journal on Applied Mathematics}} \textbf{\bibinfo{volume}{76}}, \bibinfo{pages}{1819--1843} (\bibinfo{year}{2016}).
\newblock \urlprefix\url{https://doi.org/10.1137/15M1018484}.

\bibitem{Showalter1997}
\bibinfo{author}{Showalter, R.}
\newblock \emph{\bibinfo{title}{Montone Operators in Banach Space and Nonlinear Partial Differential Equations}} Vol.~\bibinfo{volume}{49} (\bibinfo{publisher}{Mathematical Surveys and Monographs}, \bibinfo{year}{1997}).

\bibitem{zeidler1999applied}
\bibinfo{author}{Zeidler, E.}
\newblock \emph{\bibinfo{title}{Applied Functional Analysis: Applications to Mathematical Physics}} Applied Mathematical Sciences (\bibinfo{publisher}{Springer New York}, \bibinfo{year}{1999}).
\newblock \urlprefix\url{https://books.google.de/books?id=bFqRaovfY4MC}.

\bibitem{Freudenberg2024}
\bibinfo{author}{Freudenberg, T.} \& \bibinfo{author}{Eden, M.}
\newblock \bibinfo{title}{Homogenization and simulation of heat transfer through a thin grain layer}.
\newblock \emph{\bibinfo{journal}{Networks and Heterogeneous Media}} \textbf{\bibinfo{volume}{19}}, \bibinfo{pages}{569--596} (\bibinfo{year}{2024}).
\newblock \urlprefix\url{https://www.aimspress.com/article/doi/10.3934/nhm.2024025}.

\bibitem{Gahn2022Stokes}
\bibinfo{author}{Gahn, M.}, \bibinfo{author}{Jäger, W.} \& \bibinfo{author}{Neuss-Radu, M.}
\newblock \bibinfo{title}{Derivation of stokes-plate-equations modeling fluid flow interaction with thin porous elastic layers}.
\newblock \emph{\bibinfo{journal}{Applicable Analysis}} \textbf{\bibinfo{volume}{101}}, \bibinfo{pages}{4319--4348} (\bibinfo{year}{2022}).
\newblock \urlprefix\url{https://doi.org/10.1080/00036811.2022.2080673}.

\bibitem{Allaire92}
\bibinfo{author}{Allaire, G.}
\newblock \bibinfo{title}{Homogenization and two scale convergence}.
\newblock \emph{\bibinfo{journal}{{SIAM} J. Math. Anal.}} \textbf{\bibinfo{volume}{23}}, \bibinfo{pages}{1482--1518} (\bibinfo{year}{1992}).
\newblock \urlprefix\url{https://www.doi.org/10.1137/0523084}.

\bibitem{Alouges2016}
\bibinfo{author}{Alouges, F.}
\newblock \bibinfo{title}{Introduction to periodic homogenization}.
\newblock \emph{\bibinfo{journal}{Interdisciplinary Information Sciences}} \textbf{\bibinfo{volume}{22}}, \bibinfo{pages}{147--186} (\bibinfo{year}{2016}).

\bibitem{Jager1998}
\bibinfo{author}{Jäger, W.} \& \bibinfo{author}{Mikelić, A.}
\newblock \bibinfo{title}{On the effective equations of a viscous incompressible fluid flow through a filter of finite thickness}.
\newblock \emph{\bibinfo{journal}{Communications on Pure and Applied Mathematics}} \textbf{\bibinfo{volume}{51}}, \bibinfo{pages}{1073--1121} (\bibinfo{year}{1998}).
\newblock \urlprefix\url{https://onlinelibrary.wiley.com/doi/abs/10.1002/%28SICI%291097-0312%28199809/10%2951%3A9/10%3C1073%3A%3AAID-CPA6%3E3.0.CO%3B2-A}.

\bibitem{Fenics}
\bibinfo{author}{Alnaes, M.~S.} \emph{et~al.}
\newblock \bibinfo{title}{The {FEniCS} project version 1.5}.
\newblock \emph{\bibinfo{journal}{Arch. Num. Soft.}} \textbf{\bibinfo{volume}{3}} (\bibinfo{year}{2015}).
\newblock \urlprefix\url{https://www.doi.org/10.11588/ans.2015.100.20553}.

\bibitem{gmsh}
\bibinfo{author}{Geuzaine, C.} \& \bibinfo{author}{Remacle, J.-F.}
\newblock \bibinfo{title}{Gmsh: A 3-d finite element mesh generator with built-in pre- and post-processing facilities}.
\newblock \emph{\bibinfo{journal}{Int. J. Numer. Methods. Eng.}} \textbf{\bibinfo{volume}{79}}, \bibinfo{pages}{1309 -- 1331} (\bibinfo{year}{2009}).
\newblock \urlprefix\url{https://www.doi.org/10.1002/nme.2579}.

\bibitem{BrooksS82}
\bibinfo{author}{Brooks, A.~N.} \& \bibinfo{author}{Hughes, T.~J.}
\newblock \bibinfo{title}{Streamline upwind/petrov-galerkin formulations for convection dominated flows with particular emphasis on the incompressible {N}avier-{S}tokes equations}.
\newblock \emph{\bibinfo{journal}{Computer Methods in Applied Mechanics and Engineering}} \textbf{\bibinfo{volume}{32}}, \bibinfo{pages}{199--259} (\bibinfo{year}{1982}).
\newblock \urlprefix\url{https://www.doi.org/10.1016/0045-7825(82)90071-8}.

\end{thebibliography}

\end{document}